\documentclass[draft]{amsart}
\usepackage{amsmath}
\usepackage{amssymb}
\usepackage{latexsym}
\usepackage{graphicx}
\usepackage[final]{hyperref}
\hypersetup{colorlinks=true, linkcolor=blue, anchorcolor=blue, citecolor=red, filecolor=blue, menucolor=blue, pagecolor=blue, urlcolor=blue}

\newcommand{\boldN}{\boldsymbol{N}}
\newcommand{\boldL}{\boldsymbol{L}}

\newcommand{\Nmin}{N_{\mathrm{min}}}

\newcommand{\Nmax}{N_{\mathrm{max}}}
\newcommand{\Lmin}{L_{\mathrm{min}}}
\newcommand{\Lmed}{L_{\mathrm{med}}}
\newcommand{\Lmax}{L_{\mathrm{max}}}

\newcommand{\hypwt}{h}
\newcommand{\vangle}{\theta}

\newcommand{\abs}[1]{\left\vert #1 \right\vert}

\newcommand{\bigabs}[1]{\bigl\vert #1 \bigr\vert}
\newcommand{\Bigabs}[1]{\Bigl\vert #1 \Bigr\vert}

\newcommand{\norm}[1]{\left\Vert #1 \right\Vert}

\newcommand{\bignorm}[1]{\bigl\Vert #1 \bigr\Vert}

\newcommand{\C}{\mathbb{C}}
\newcommand{\N}{\mathbb{N}}
\newcommand{\Proj}{\mathbf{P}}

\newcommand{\R}{\mathbb{R}}
\newcommand{\Z}{\mathbb{Z}}

\newcommand{\Innerprod}[2]{\left\langle \, #1 , #2 \, \right\rangle}
\newcommand{\innerprod}[2]{\langle \, #1 , #2 \, \rangle}

\newcommand{\angles}[1]{\langle #1 \rangle}

\DeclareMathOperator{\diag}{diag}

\DeclareMathOperator{\supp}{supp}

\newtheorem{theorem}{Theorem}[section]

\newtheorem{lemma}{Lemma}[section]

\theoremstyle{definition}

\theoremstyle{remark}

\numberwithin{equation}{section}
\title[Global solutions of 2d Maxwell-Dirac]{Global well-posedness of the Maxwell-Dirac system in two space dimensions}

\author[P. D'Ancona]{Piero D'Ancona}
\address{Department of Mathematics\\
University of Rome ``La Sapienza''\\
Piazzale Aldo Moro 2\\
I-00185 Rome\\ Italy}
\email{dancona@mat.uniroma1.it}

\author[S. Selberg]{Sigmund Selberg}
\address{Department of Mathematical Sciences\\
Norwegian University of Science and Technology\\
Alfred Getz' vei 1\\
N-7491 Trondheim\\ Norway}
\email{sigmund.selberg@math.ntnu.no}

\subjclass[2000]{35Q40; 35L70}

\begin{document}

\begin{abstract}
In recent work, Gr\"unrock and Pecher proved that the Dirac-Klein-Gordon system in 2d is globally well-posed in the charge class (data in $L^2$ for the spinor and in a suitable Sobolev space for the scalar field). Here we obtain the analogous result for the full Maxwell-Dirac system in 2d. Making use of the null structure of the system, found in earlier joint work with Damiano Foschi, we first prove local well-posedness in the charge class. To extend the solutions globally we build on an idea due to Colliander, Holmer and Tzirakis. For this we rely on the fact that MD is charge subcritical in two space dimensions, and make use of the null structure of the Maxwell part. 
\end{abstract}

\maketitle
\tableofcontents

\section{Introduction}\label{A}

The Maxwell-Dirac system (MD) describes the motion of an electron interacting with an electromagnetic field. Here we study the 2d (two space dimensions) case, where the electron is restricted to move in  the $(x^1,x^2)$-plane. Then the electric field $\mathbf E$ is constrained to the same plane, the magnetic field $\mathbf B$ is perpendicular to it, and all fields depend only on $(t,x^1,x^2)$ (not on $x^3$), so we write $x=(x^1,x^2)$, and occasionally $t=x^0$. The partial derivative with respect to $x^\mu$ is denoted $\partial_\mu$ for $\mu=0,1,2$; we write $\partial_t=\partial_0$, and $\nabla$ denotes the spatial gradient. The summation convention is in effect: Roman indices $j,k,\dots$ run over $\{ 1,2 \}$, greek indices $\mu,\nu,\dots$ over $\{ 0,1,2 \}$, and repeated upper/lower indices are implicitly summed over these ranges. Indices are raised and lowered using the metric $\diag(-1,1,1)$ on $\R^{1+2}$. 

In terms of a potential $A=\{A_\mu\}_{\mu=0,1,2}$ with $A_\mu \colon \R^{1+2} \to \R$,
$$
  \mathbf B = \nabla \times \mathbf A = (0,0,\partial_1 A_2 - \partial_2 A_1),
  \qquad
  \mathbf E = \nabla A_0 - \partial_t \mathbf A,
$$
where $\mathbf A = (A_1,A_2,0)$ denotes the spatial part of $A$. Expressing Maxwell's equations in terms of $A$, and imposing the Lorenz gauge condition
$$
  \partial^\mu A_\mu = 0 \qquad (\iff \partial_t A_0 = \nabla \cdot \mathbf A),
$$
the MD system reads (see e.g.~\cite{Selberg:2008b})
\begin{equation}\label{MD}
  \left\{
  \begin{aligned}
  \left(-i\boldsymbol\alpha^\mu \partial_\mu + M\boldsymbol\beta\right)\psi &= A_\mu \boldsymbol\alpha^\mu \psi,
  \\
  \square A_\mu &= - \innerprod{\boldsymbol\alpha_\mu\psi}{\psi},
\end{aligned}
\right.
\end{equation}
where $\psi \colon \R^{1+2} \to \C^N$ is the Dirac spinor, $M \in \R$ is a constant and $\square = \partial_\mu \partial^\mu = -\partial_t^2 + \Delta_x$ is the D'Alembertian on $\R^{1+2}$. Since we work in 2d, the smallest possible dimension of the spinor space is $N=2$, and then for the $2 \times 2$ Dirac matrices we can take the representation $\boldsymbol\alpha^0 = \mathbf I_{2 \times 2}$, $\boldsymbol\alpha^1 = \boldsymbol\sigma^1$, $\boldsymbol\alpha^2 = \boldsymbol\sigma^2$, $\boldsymbol\beta = \boldsymbol\sigma^3$, where the $\boldsymbol\sigma^j$ are the Pauli matrices. Finally, $\innerprod{\cdot}{\cdot}$ is the standard $\C^2$ inner product.

Recently there has been significant progress in the regularity theory for MD and the simpler Dirac-Klein-Gordon system (DKG),
\begin{equation}\label{DKG}
  \left\{
  \begin{aligned}
  \left(-i\boldsymbol\alpha^\mu \partial_\mu + M\boldsymbol\beta\right)\psi &= \phi \boldsymbol\beta \psi,
  \\
  (-\square + m^2) \phi &= \innerprod{\boldsymbol\beta\psi}{\psi},
\end{aligned}
\right.
\end{equation}
where $\phi$ is real-valued and $m \in \R$ is a constant.

A key question for both systems is whether \emph{global regularity} holds, i.e.~starting from  smooth initial data, does the solution exist for all time and stay smooth? For small data this has been answered affirmatively by Georgiev \cite{Georgiev:1991} in 3d, but for large data there was until quite recently only the 1d result of Chadam \cite{Chadam:1973}.

To make progress on the large data question in 2d and 3d, a natural strategy is to study local (in time) well-posedness for rough data and exploit conservation laws to extend the solutions globally.

But for both DKG and MD, the energy lacks a definite sign (see \cite{Glassey:1979}), so the only conserved quantity that appears to be immediately useful is the charge:
$$
  \norm{\psi(t)}_{L^2}^2 = \text{const}.
$$
This constant will be called the \emph{charge constant} in what follows.

The charge conservation was of course a key ingredient in Chadam's global result for 1d MD \cite{Chadam:1973}, later improved for the 1d DKG case by Bournaveas \cite{Bournaveas:2000}, in the sense that the regularity requirements were lowered to the charge class (data in $L^2$ for the spinor and in some Sobolev space for the scalar field). Since then a number of papers improving the local and global theory for 1d DKG have appeared, see \cite{Fang:2004, BournaveasGibbeson:2006, Machihara:2007, Pecher:2006, Pecher:2008, Selberg:2007c, Selberg:2008a, Tesfahun:2009, Machihara:2010}.

As the space dimension increases, however, it becomes much more difficult to prove local existence in the charge class, and therefore correspondingly difficult to exploit the charge conservation. Indeed, it was to take more than thirty years from the 1d result of Chadam until the next major breakthrough in the global theory was achieved quite recently by Gr\"unrock and Pecher \cite{Gruenrock:2009}, who proved global well-posedness for 2d DKG. At the same time, but independently, Ovcharov \cite{Ovcharov:2009} proved a corresponding result under a spherical symmetry assumption.

Decisive improvements in the local theory have been made possible through the discovery, by the authors in joint work with Damiano Foschi, of the complete null structure of first DKG, in \cite{Selberg:2007d}, and then MD, in \cite{Selberg:2008b}, permitting significant progress compared to earlier local results such as \cite{Gross:1966, Bournaveas:1996, Bournaveas:2001, Masmoudi:2003, Fang:2005, Selberg:2005}, where at most partial null structure was used.

In \cite{Gruenrock:2009}, Gr\"unrock and Pecher use the DKG null structure combined with bilinear estimates similar to those used in \cite{Selberg:2007b}, where in particular it was shown that 2d DKG is locally well posed for data 
\begin{equation}\label{DKGdata}
  \psi(0) \in L^2(\R^2), \quad \phi(0) \in H^{1/2}(\R^2), \quad \partial_t \phi(0) \in H^{-1/2}(\R^2),
\end{equation}
with a time of existence depending only on the size of the data norm. Thus, to get a global result it suffices---in view of the conservation of charge---to show that
$$
  \norm{\phi(t)}_{H^{1/2}(\R^2)} +  \norm{\partial_t\phi(t)}_{H^{-1/2}(\R^2)}
$$
cannot blow up in finite time. In fact, Gr\"unrock and Pecher prove this for an equivalent norm which we shall denote $D(t)$. In our reformulation, they prove: 

\begin{theorem}\label{GPthm} \emph{(\cite{Gruenrock:2009}.)}
The local solution of 2d DKG exists up to a time $T > 0$ determined by
\begin{equation}\label{Time}
  T^{1/2} [1 + D(0)] = \varepsilon,
\end{equation}
where $\varepsilon > 0$ depends on the charge constant. Moreover, if $D(0) \ge 1$ then
\begin{equation}\label{Growth}
  \sup_{0 \le t \le T} D(t) \le D(0) + CT^{1/2},
\end{equation}
where $C$ depends on the charge constant.
\end{theorem}

Both DKG and MD are charge subcritical in 2d (whereas the 3d problems are charge critical).
To be precise, the critical regularity determined by scaling is half a derivative below the regularity of the charge class data \eqref{DKGdata}, hence the half power of $T$ in \eqref{Time} is optimal, and in fact so is the half power in \eqref{Growth}. The fact that the two exponents add up to $1$ enabled Gr\"unrock and Pecher to apply a scheme devised by Colliander, Holmer and Tzirakis \cite{Colliander:2008} to extend solutions globally. We recall the argument here since a modified version of it will be used for MD.

Since the only possible impediment to global existence is $D(t)$ becoming large, one may assume $D(t) \gg 1$ for all $t \ge 0$ for which the solution exists. Now as long as $D(t) \le 2D(0)$, Theorem \ref{GPthm} can be applied repeatedly with a uniform time increment $T$ given by $T^{1/2} [1+2D(0)] = \varepsilon$. In view of \eqref{Growth} the theorem can be applied $n$ times, where $n$ is the smallest integer such that $nCT^{1/2} > D(0)$. In this way one covers a total time interval of length
$$
  nT = nCT^{1/2} \frac{1}{C}T^{1/2}
  > D(0) \frac{\varepsilon}{C[1+2D(0)]} 
  \ge D(0) \frac{\varepsilon}{C[3D(0)]} 
  = \frac{\varepsilon}{3C} > 0,
$$
the crucial point being that $\varepsilon/3C$ is independent of $D(0)$. Repeating the whole argument one can therefore cover a time interval of arbitrary length.

The purpose of the present paper is to extend the result of Gr\"unrock and Pecher to the full MD system. This adds significant difficulties since MD has a far more complicated null structure than DKG, and since instead of a single scalar field $\phi$ we have to deal with the electromagnetic field $(\mathbf E,\mathbf B)$. Because of these additional difficulties, we have to face the following two issues, affecting the above global existence argument:
\begin{enumerate}
  \item For MD we are only able to prove the analog of \eqref{Growth} up to a logarithmic loss in the factor $T^{1/2}$, i.e.~the term $CT^{1/2}$ on the right hand side is replaced by $CT^{1/2}\log(1/T)$, where $D(t)$ is now a certain norm of $(\mathbf E,\mathbf B)(t)$ such that local existence holds up to a time $0 < T \ll 1$ determined by \eqref{Time}.
  \item\label{Issue2} The norm $D(t)$ actually depends implicitly on $T$.
\end{enumerate}

Because of these issues, we are not able to apply the scheme of Colliander, Holmer and Tzirakis in its original form, but with some extra work---exploiting in particular a crucial monotonicity property of our data norm with respect to $T$---we are nevertheless able to obtain a global existence result. The detailed argument is given in section \ref{LocalToGlobal}, but as a warm-up we sketch here the argument in the much simpler situation where we ignore the implicit dependence of $D(t)$ on $T$: The local result can then be iterated until $nCT^{1/2}\log(1/T) > D(0)$, giving a total time
$$
  \Delta \equiv nT
  > \frac{D(0)}{\log(1/T)} \frac{\varepsilon}{C[1+2D(0)]} 
  \sim \frac{1}{\log(1/T)} \sim \frac{1}{\log D(0)},
$$
where \eqref{Time} was used. Moreover, one can easily show $D(\Delta) \le 3D(0)$, so by a further iteration one covers successive time intervals of length $\Delta_1,\Delta_2,\dots$ such that
$$
  \Delta_{j+1} \gtrsim \frac{1}{\log(3^j D(0))} \sim \frac{1}{j+1}
$$
for $j \ge 0$, hence $\sum_{j=1}^\infty \Delta_j = \infty$.

\emph{Some notation:} The Fourier transforms on $\R^2$ and $\R^{1+2}$ are defined by
$$
  \widehat f(\xi) = \int_{\R^2} e^{-ix\cdot\xi} f(x) \, dx,
  \qquad
  \widetilde u(X) = \int_{\R^{1+2}} e^{-i(t\tau+x\cdot\xi)} u(t,x) \, dt \, dx,
$$
where $\xi \in \R^2$, $\tau \in \R$ and $X=(\tau,\xi)$. We also write $\mathcal F u = \widetilde u$.

If $A$ is a subset of $\R^{1+2}$, or a condition describing such a set, the multiplier $\Proj_A$ is defined by
$$
  \widetilde{\Proj_A u}(X) = \chi_A(X)\widetilde u(X),
$$
where $\chi_A$ is the characteristic function of $A$, and similarly if $A \subset \R^2$.

We write $D=-i\nabla$, and given $h : \R^2 \to \C$ we denote by $h(D)$ the multiplier defined by $$
  \widehat{h(D)f}(\xi) = h(\xi)\widehat f(\xi).
$$

The notation $\norm{\cdot}$ is reserved for the $L^2$-norms on both $\R^2$ and $\R^{1+2}$ (which one it is will be clear from the context):
$$
  \norm{f} = \left( \int_{\R^2} \abs{f(x)}^2 \, dx \right)^{1/2},
  \qquad
  \norm{u} = \left( \int_{\R^{1+2}} \abs{u(t,x)}^2 \, dt \, dx \right)^{1/2},
$$
and similarly in Fourier space. For $s \in \R$, the Sobolev space $H^s=H^s(\R^2)$ is defined as the completion of the Schwartz space $\mathcal S(\R^2)$ with respect to the norm
$$
  \norm{f}_{H^s} = \norm{\angles{D}^s f},
$$
where $\angles{\xi} = (1+\abs{\xi}^2)^{1/2}$. The Besov space $\dot B^s_{2,1}=\dot B^s_{2,1}(\R^2)$ is the completion of $\mathcal S(\R^2)$ with respect to the norm
$$
  \norm{f}_{\dot B^s_{2,1}} = \sum_{N > 0} N^s \norm{\Proj_{\abs{\xi} \sim N} f},
$$
where $N$ is understood to be dyadic, i.e.~of the form $2^j$ with $j \in \Z$.

In estimates we use the shorthand $X \lesssim Y$ for $X \le CY$, where $C \gg 1$ is either an absolute constant or depends only on quantities that are considered fixed; $X=O(R)$ is short for $\abs{X} \lesssim R$; $X \sim Y$ means $X \lesssim Y \lesssim X$; $X \ll Y$ stands for $X \le C^{-1} Y$, with $C$ as above. We write $\simeq$ for equality up to multiplication by an absolute constant (typically factors involving $2\pi$).

\section{Main results}\label{Main}

\subsection{Local well-posedness} We consider the initial value problem for 2d MD starting from data
\begin{equation*}
  \psi(0,x) = \psi_0(x),
  \qquad \mathbf E(0,x) = \mathbf E_0(x),
  \\
  \qquad \mathbf B(0,x) = \mathbf B_0(x)=(0,0,B_0^3),
\end{equation*}
which by Maxwell's equations [see \eqref{Maxwell} below] must satisfy $\nabla \cdot \mathbf E_0 = \abs{\psi_0}^2$ and $\nabla \cdot \mathbf B_0 = 0$. But the latter automatically holds in 2d, since $\mathbf B = (0,0,B^3)$ does not depend on $x^3$, whereas the constraint $\nabla \cdot \mathbf E_0 = \abs{\psi_0}^2$ determines the curl-free part\footnote{Recall the splitting of $\mathbf E$ (or indeed any vector field) into divergence-free and curl-free parts:
$
\mathbf E = -\Delta^{-1} \nabla \times (\nabla
\times \mathbf E)
  + \Delta^{-1} \nabla ( \nabla \cdot \mathbf E )
  \equiv
  \mathbf E^{\text{df}} + \mathbf E^{\text{cf}}.
$
} of $\mathbf E_0$, so we only specify data $\mathbf E^{\mathrm{df}}_0$ for the divergence-free part $\mathbf E^{\mathrm{df}}$. Thus,
$$
  \mathbf E_0 = \mathbf E^{\mathrm{df}}_0 + \Delta^{-1}\nabla(\abs{\psi_0}^2).
$$
The data for the potential $A$,
$$
  A_\mu(0,x) = a_\mu(x), \qquad
  \partial_t A_\mu(0,x) = \dot a_\mu(x) \qquad (\mu=0,1,2),
$$
are fixed by choosing
$$
  a_0 = \dot a_0 = 0.
$$
Then the spatial parts $\mathbf a = (a_1,a_2,0)$ and $\dot{\mathbf a} = (\dot a_1,\dot a_2,0)$ are given by, since $\nabla \cdot \mathbf a = 0$ by the Lorenz condition,
$$
  \mathbf a = - \Delta^{-1} (\partial_2 B^3_0, - \partial_1 B^3_0, 0),
  \qquad
  \dot{\mathbf a} = - \mathbf E_0.
$$

Solving the second equation in \eqref{MD} and splitting $A_\mu$ into its homogeneous and inhomogeneous parts, we reduce MD to a nonlinear Dirac equation
\begin{equation}\label{MDreduced}
  \left(-i\boldsymbol\alpha^\mu \partial_\mu + M \boldsymbol\beta\right)\psi
  = A_\mu^{\mathrm{hom.}} \boldsymbol\alpha^\mu \psi
  - \mathcal N(\psi,\psi,\psi),
\end{equation}
where
$$
  \square A_\mu^{\mathrm{hom.}} = 0,
  \qquad
  A_\mu^{\mathrm{hom.}}(0,x) = a_\mu(x),
  \qquad
  \partial_t A_\mu^{\mathrm{hom.}}(0,x) = \dot a_\mu(x),
$$
and
$$
  \mathcal N(\psi_1,\psi_2,\psi_3)
  =
  \left( \square^{-1} \Innerprod{\boldsymbol\alpha_\mu\psi_1}{\psi_2} \right) 
  \boldsymbol\alpha^\mu\psi_3.
$$
Here $\square^{-1}F$ denotes the solution of $\square u = F$ on $\R^{1+2}$ with vanishing data at  $t=0$.

Assuming the following data regularity:
\begin{equation}\label{Data}
\left\{\begin{gathered}
  \psi_0 \in L^2(\R^2,\C^2),
  \\
  \Proj_{\abs{\xi} \ge 1} \mathbf E_0^{\text{df}} \in H^{-1/2}(\R^2,\R^2),
  \qquad
  \Proj_{\abs{\xi} <  1} \mathbf E_0^{\text{df}} \in \dot B^0_{2,1}(\R^2,\R^2),
  \\
  \Proj_{\abs{\xi} \ge 1}B^3_0 \in H^{-1/2}(\R^2,\R),
  \qquad
  \Proj_{\abs{\xi} < 1}B^3_0 \in \dot B^0_{2,1}(\R^2,\R),
\end{gathered}
\right.
\end{equation}
we can prove existence up to a time $T > 0$ determined by a condition like~\eqref{Time} in Theorem \ref{GPthm}, but with a norm depending implicitly on $T$, namely
\begin{equation}\label{DataNorm}
  D_T(t)
  = \norm{\mathbf E^{\text{df}}(t)}_{(T)}
  + \norm{B^3(t)}_{(T)},
\end{equation}
where we use the norm $\norm{\cdot}_{(T)}$ defined by
\begin{equation}\label{MagicNorm}
  \norm{f}_{(T)} 
  =
  \norm{\Proj_{\abs{\xi} \ge 1/T} f}_{H^{-1/2}}
  + T^{1/2} \sum_{0 < N < 1/T} \norm{\Proj_{\abs{\xi} \sim N} f},
\end{equation}
the sum being over dyadic $N$'s. Recall that $\norm{\cdot}$ denotes the $L^2$-norm.

\begin{theorem}\label{LWP}
Given initial data as in \eqref{Data}, construct data for $A$ by choosing $a_0 = \dot a_0 = 0$ and setting $\mathbf a = - \Delta^{-1} (\partial_2 B^3_0, - \partial_1 B^3_0, 0)$ and $\dot{\mathbf a} = - \mathbf E_0$, and consider the 2d MD equation \eqref{MDreduced}.

There exists a constant $\varepsilon > 0$, depending only on $\abs{M}$ and the charge constant $\norm{\psi_0}_{L^2}^2$, such that if $T > 0$ is so small that
\begin{equation}\label{MDtime}
  T^{1/2} [1 + D_T(0)] \le \varepsilon,
\end{equation}
then \eqref{MDreduced} has a solution
$$
  \psi \in C\left([-T,T];L^2(\R^2,\C^2)\right)
$$
satisfying $\psi(0) = \psi_0$.

Moreover, the solution is unique in a certain subspace of $C\left([-T,T];L^2\right)$, and depends continuously on the data. Persistence of higher regularity holds, and in particular, if the data $\psi_0$, $\mathbf E_0^{\mathrm{df}}$ and $B^3_0$ are smooth, then so is $\psi$.
\end{theorem}

Here we mean solution in the sense of distributions on $(-T,T) \times \R^2$. The fact that the right hand side of \eqref{MDreduced} makes sense as a distribution is far from obvious, but follows from the very estimates that will be used to close the iteration argument used to prove existence.

As we show later (see Lemma \ref{MagicLemma}), $D_T(0) \le C D_1(0)$ for $0 < T \le 1$, hence \eqref{MDtime} is indeed satisfied for $T > 0$ sufficiently small.

\subsection{Growth estimate for the electromagnetic field}

Having obtained $\psi$, we reconstruct the full potential
$$
  A_\mu = A_\mu^{\mathrm{hom.}} - \square^{-1} \Innerprod{\boldsymbol\alpha_\mu\psi}{\psi},
$$
which by the definition of the data $(a_\mu,\dot a_\mu)$ satisfies the Lorenz gauge condition $\partial^\mu A_\mu = 0$ (see \cite{Selberg:2008b}). Now define
$$
  \mathbf B = \nabla \times \mathbf A = (0,0,\partial_1 A_2 - \partial_2 A_1),
  \qquad
  \mathbf E = \nabla A_0 - \partial_t \mathbf A.
$$
Since $\square A_\mu = - \innerprod{\boldsymbol\alpha^\mu\psi}{\psi}$, it follows that Maxwell's equations hold:
\begin{equation}\label{Maxwell}
  \nabla \cdot \mathbf E = \rho,
  \qquad
  \nabla \cdot \mathbf B = 0,
  \qquad
  \nabla \times \mathbf E + \partial_t \mathbf B = 0,
  \qquad
  \nabla \times \mathbf B - \partial_t \mathbf E = \mathbf J,
\end{equation}
where
$$
  \rho = J^0 = \abs{\psi}^2,
  \qquad
  \mathbf J = (J^1,J^2,0),
  \qquad
  J^\mu = \innerprod{\boldsymbol\alpha^\mu\psi}{\psi}.
$$

The first equation in \eqref{Maxwell} determines the curl-free part of $\mathbf E$ and implies
$$
  \mathbf E = \mathbf E^{\mathrm{df}} + \Delta^{-1}\nabla(\abs{\psi}^2),
$$
where $\mathbf E^{\mathrm{df}} = \mathcal P_{\mathrm{df}} \mathbf E$ is the divergence-free part of $\mathbf E$. Here $\mathcal P_{\mathrm{df}} = - \Delta^{-1} {\nabla \times} {\nabla \times}$ is the projection onto divergence-free fields. From Maxwell's equations we know that $\square \mathbf E = \nabla \rho + \partial_t \mathbf J$ and $\square \mathbf B = - \nabla \times \mathbf J$, hence 
\begin{equation}\label{WaveEqE}
\left\{
\begin{aligned}
  &\square \mathbf E^{\mathrm{df}}
  = \mathcal P_{\text{df}}(-\nabla J_0 + \partial_t \mathbf J),
  \\
  &\mathbf E^{\mathrm{df}}(0) = \mathbf E_0^{\text{df}},
  \qquad
  \partial_t \mathbf E^{\mathrm{df}}(0) = \nabla \times (0,0,B^3_0) - \mathcal P_{\text{df}} \mathbf J(0),
\end{aligned}
\right.
\end{equation}
and
\begin{equation}\label{WaveEqB}
\left\{
\begin{aligned}
  &\square B^3 = \partial_1 J_2 - \partial_2 J_1,
  \\
  &B^3(0) = B^3_0, \qquad \partial_t B^3(0) = - ( \nabla \times \mathbf E^{\mathrm{df}}_0 )^3.
\end{aligned}
\right.
\end{equation}

We want to use these wave equations to prove an estimate analogous to \eqref{Growth} in Theorem \ref{GPthm} for our norm $D_T(t)$. To be precise, we aim to prove
$$
  \sup_{0 \le t \le T} D_T(t) \le D_T(0) + CT^{1/2}\log(1/T),
$$
but in order to avoid a constant factor $C > 1$ in front of the first term on the right hand side, we first split the wave equations into first order equations and modify $D_T(t)$ accordingly.

Recall that the splitting $u = u_+ + u_-$ given by
\begin{equation}\label{WaveSplitting}
  u_\pm = \frac12 \left( u \pm i\abs{D}^{-1} \partial_t u \right)
\end{equation}
transforms $\square u = F$ into
$$
  \left(-i\partial_t \pm \abs{D}\right) u_\pm = - (\pm 2\abs{D})^{-1} F.
$$
The term $\abs{D}^{-1} \partial_t u$ in \eqref{WaveSplitting} causes problems at low frequency if $u=\mathbf E^{\mathrm{df}}$, however. To avoid this we use a general trick going back at least as far as \cite{Pecher:2008}, and used also in \cite{Gruenrock:2009}: Adding $-\mathbf E^{\mathrm{df}}$ to both sides of \eqref{WaveEqE} gives the Klein-Gordon equation
\begin{equation}\label{KleinGordonEqE}
\left\{
\begin{aligned}
  &(\square-1) \mathbf E^{\mathrm{df}}
  = \mathcal P_{\text{df}}(-\nabla J_0 + \partial_t \mathbf J) - \mathbf E^{\mathrm{df}},
  \\
  &\mathbf E^{\mathrm{df}}(0) = \mathbf E_0^{\text{df}},
  \qquad
  \partial_t \mathbf E^{\mathrm{df}}(0) = \nabla \times (0,0,B^3_0) - \mathcal P_{\text{df}} \mathbf J(0).
\end{aligned}
\right.
\end{equation}
The extra term $-\mathbf E^{\mathrm{df}}$ on the right hand side is relatively easy to handle due to the gain in regularity, and the key advantage is that we can now use the analog of \eqref{WaveSplitting} for the Klein-Gordon equation: The splitting $v = v_+ + v_-$ given by
\begin{equation}\label{KGSplitting}
  v_\pm = \frac12 \left( v \pm i\angles{D}^{-1} \partial_t v \right)
\end{equation}
transforms $(\square-1) v = G$ into
$$
  \left(-i\partial_t \pm \angles{D} \right) v_\pm = - (\pm 2\angles{D})^{-1} G.
$$

Applying \eqref{KGSplitting} to $\mathbf E^{\mathrm{df}}$ and \eqref{WaveSplitting} to $B^3$, we now write $\mathbf E^{\mathrm{df}} = \mathbf E^{\mathrm{df}}_+ + \mathbf E^{\mathrm{df}}_-$ and $B^3 = B^3_+ + B^3_-$, where
\begin{align}
  \label{Esplit}
  2\mathbf E^{\mathrm{df}}_\pm &= \mathbf E^{\mathrm{df}} \pm i\angles{D}^{-1} \partial_t \mathbf E^{\mathrm{df}}
  =
  \mathbf E^{\mathrm{df}} \pm i\angles{D}^{-1} \left[ \nabla \times (0,0,B^3) - \mathcal P_{\text{df}} \mathbf J \right],
  \\
  \label{Bsplit}
  2B^3_\pm &= B^3 \pm i\abs{D}^{-1} \partial_t B^3
  =
  B^3 \pm i\abs{D}^{-1} \left[- ( \nabla \times \mathbf E^{\mathrm{df}} )^3 \right]
\end{align}
satisfy
\begin{align}
  \label{Eeq}
  \left(-i\partial_t \pm \angles{D}\right) \mathbf E^{\mathrm{df}}_\pm
  &= - (\pm 2\angles{D})^{-1} \left[ \mathcal P_{\text{df}}(-\nabla J_0 + \partial_t \mathbf J) - \mathbf E^{\mathrm{df}} \right],
  \\
  \label{Beq}
  \left(-i\partial_t \pm \abs{D}\right) B^3_\pm
  &= - (\pm 2\abs{D})^{-1} \left(\partial_1 J_2 - \partial_2 J_1\right).
\end{align}
Define the corresponding norm
\begin{equation}\label{DataNormModified}
\begin{aligned}
  \tilde D_T(t) &= \norm{\mathbf E^{\text{df}}_+(t)}_{(T)}
  + \norm{\mathbf E^{\text{df}}_-(t)}_{(T)}
  + \norm{B^3_+(t)}_{(T)}
  + \norm{B^3_-(t)}_{(T)},
\end{aligned}
\end{equation}
and note that $\tilde D_T(0) < \infty$. Indeed, $\tilde D_T(0) \le C\tilde D_1(0)$ by Lemma \ref{MagicLemma} below, and $\tilde D_1(0) < \infty$ in view of the assumption \eqref{Data} and some straightforward Sobolev estimates for $\mathbf J$ [see \eqref{Jest1} and \eqref{Jest2} below].

Since $D_T(t) \le \tilde D_T(t)$ by the triangle inequality, the iteration argument used to prove Theorem \ref{LWP} will also immediately give us:

\begin{theorem}\label{LWPnew}
Theorem \ref{LWP} still holds with $D_T(0)$ replaced by $\tilde D_T(0)$ in \eqref{MDtime}:
\begin{equation}\label{MDtimeNew}
  T^{1/2} [1 + \tilde D_T(0)] \le \varepsilon,
\end{equation}
where $\varepsilon > 0$ depends only on the charge constant and $\abs{M}$.
\end{theorem}

We shall prove the following growth estimate for $\tilde D_T(t)$.

\begin{theorem}\label{EMgrowth}
Let $\psi$ be the solution of 2d MD obtained in Theorem \ref{LWPnew}, with existence time $T$ satisfying \eqref{MDtimeNew}, and reconstruct the electromagnetic field as above. Then $\mathbf E^{\mathrm{df}}_\pm$ and $B^3_\pm$, as functions of $t \in [-T,T]$, describe continuous curves in the data space \eqref{Data}, hence the same is true for $\mathbf E^{\mathrm{df}}=\mathbf E^{\mathrm{df}}_+ + \mathbf E^{\mathrm{df}}_-$ and $B^3=B^3_+ + B^3_-$. Moreover, we have
\begin{equation}\label{ModifiedEMgrowth}
  \sup_{0 \le t \le T} \tilde D_T(t) \le \tilde D_T(0) + CT^{1/2}\log(1/T),
\end{equation}
where $C$ depends only on the charge constant and $\abs{M}$.
\end{theorem}

Combining Theorems \ref{LWPnew} and \ref{EMgrowth}, we shall obtain the global well-posedness:

\begin{theorem}\label{GWP}
The solution of 2d MD obtained in Theorem \ref{LWPnew} extends globally in time. In particular, for smooth data the solution is smooth on $\R^{1+2}$, so global regularity holds for 2d MD.
\end{theorem}

The rest of this paper is organized as follows: In the next section we prove Theorem \ref{GWP}, in section \ref{Notation} we introduce various notation and functions spaces needed for the proof of Theorems \ref{LWP} and \ref{LWPnew}, given in sections~\ref{LWPsection}--\ref{S}. Finally, in section~\ref{N} we prove Theorem~\ref{EMgrowth}.

\section{From local to global solutions}\label{LocalToGlobal}

Here we prove that if the conclusions of Theorems \ref{LWPnew} and \ref{EMgrowth} hold, then the solutions extend globally in time, hence we obtain Theorem \ref{GWP}. We follow as closely as possible the argument outlined at the end of section \ref{A}, but the fact that our norm depends implicitly on $T$ creates some difficulties. To resolve these we rely crucially on the following monotonicity property of the norm \eqref{MagicNorm}:

\begin{lemma}\label{MagicLemma}
There exists $C > 1$ such that for all $0 < S < T \le 1$ and $f \in \mathcal S(\R^2)$,
$$
  \norm{f}_{(S)} \le C \norm{f}_{(T)}.
$$
\end{lemma}

\begin{proof} By definition,
$$
  \norm{f}_{(S)} 
  =
  \norm{\Proj_{\abs{\xi} \ge 1/S} f}_{H^{-1/2}}
  + S^{1/2} \sum_{0 < N < 1/S} \norm{\Proj_{\abs{\xi} \sim N} f},
$$
but the second term is clearly bounded by
$$
  T^{1/2}\sum_{0 < N < 1/T} \norm{\Proj_{\abs{\xi} \sim N} f}
  + S^{1/2}\sum_{1/T \le N < 1/S} \norm{\Proj_{\abs{\xi} \sim N} f},
$$
where in turn the second term is bounded by an absolute constant times
$$
  S^{1/2} \sum_{N < 1/S} N^{1/2} \norm{\Proj_{1/T \le \abs{\xi} < 1/S} f}_{H^{-1/2}}
  \lesssim \norm{\Proj_{1/T \le \abs{\xi} < 1/S} f}_{H^{-1/2}},
$$
hence
\begin{align*}
  \norm{f}_{(S)} 
  &\lesssim
  \norm{\Proj_{\abs{\xi} \ge 1/S} f}_{H^{-1/2}}
  + \norm{\Proj_{1/T \le \abs{\xi} < 1/S} f}_{H^{-1/2}}
  + T^{1/2}\sum_{0 < N < 1/T} \norm{\Proj_{\abs{\xi} \sim N} f}
  \\
  &\lesssim
  \norm{\Proj_{\abs{\xi} \ge 1/T} f}_{H^{-1/2}}
  + T^{1/2}\sum_{0 < N < 1/T} \norm{\Proj_{\abs{\xi} \sim N} f}
  = \norm{f}_{(T)},
\end{align*}
where the implicit constants are absolute.
\end{proof}

We now proceed in two steps, first iterating the local existence result with a fixed time increment. Then in the second step we iterate the entire first step.

\subsection{First iteration} Since $\tilde D_T(0) \le C\tilde D_1(0)$, there clearly exists $0 < T \ll 1$ such that
\begin{equation}\label{MDtimeNew2}
  T^{1/2} [1 + \tilde D_T(0)] = \frac{\varepsilon}{2},
\end{equation}
with $\varepsilon$ as in \eqref{MDtimeNew}. Then as long as
$$
  \tilde D_T(t) \le 2\tilde D_T(0)
$$
we will have
$$
  T^{1/2} [1 + \tilde D_T(t)] \le \varepsilon,
$$
so that the solution can be continued on $[t,t+T]$, by Theorem \ref{LWPnew}. Thus we obtain existence on successive time intervals $[0,T],[2T,3T],\dots,[(n-1)T,nT]$, and in view of the estimate \eqref{ModifiedEMgrowth} from Theorem \ref{EMgrowth}, we must stop at the first $n$ for which
\begin{equation}\label{Stop}
  nCT^{1/2}\log(1/T) > \tilde D_T(0),
\end{equation}
at which point we have covered a total time interval of length
\begin{equation}\label{Delta}
  \Delta \equiv nT
  > \frac{\tilde D_T(0)}{C\log(1/T)} \frac{\varepsilon}{2[1+\tilde D_T(0)]} 
  \sim \frac{1}{\log(1/T)} \sim \frac{1}{\log \tilde D_T(0)},
\end{equation}
where we used the fact, justified below, that $\tilde D_T(0)$ can be assumed as large as we like:
\begin{equation}\label{Dlarge}
  \tilde D_T(0) \gg 1,
\end{equation}
so in particular
\begin{equation}\label{Tlog}
  \log(1/T) \sim \log \tilde D_T(0),
\end{equation}
in view of \eqref{MDtimeNew2}.

Moreover, we claim that
\begin{equation}\label{Triple}
  \tilde D_T(\Delta) \le 3\tilde D_T(0).
\end{equation}
To see this, first note that by \eqref{Stop}, and using \eqref{MDtimeNew2}, \eqref{Dlarge} and \eqref{Tlog},
$$
  n \gtrsim \frac{\tilde D_T(0)^2}{\log D(0)},
$$
so by \eqref{Dlarge} we may assume $n \gg 1$, and using the definition of $n$ we then get
$$
  \tilde D_T(0) \ge (n-1)CT^{1/2}\log(1/T) \ge \frac{n}{2}CT^{1/2}\log(1/T),
$$
which together with \eqref{ModifiedEMgrowth} proves \eqref{Triple}.

Finally, to justify \eqref{Dlarge}, consider the maximal interval of existence $[0,T^*)$. We assume $T^* < \infty$, as otherwise we already have global existence and there is nothing to prove. But then by translating the time origin sufficiently close to $T^*$ we may in fact assume $T^*$ as small as we like, and we observe that \eqref{MDtimeNew2} implies
$$
  \tilde D_T(0) \sim T^{-1/2} > (T^*)^{-1/2}
$$
for small $T^* > 0$. This proves \eqref{Dlarge}.

\subsection{Second iteration}

Now we iterate the first iteration, introducing a subscript $j=1,2,\dots$ on $T$, $n$ and $\Delta$ belonging to the $j$-th interation step. Define $S_0 = 0$ and $S_j = S_{j-1} + \Delta_j$ for $j \ge 1$.

The initial data at the $j$-th step are then taken at time $t=S_{j-1}$, and the time increment $T_j$ is determined by the condition
\begin{equation}\label{Tj}
  T_j^{1/2} [1 + \tilde D_{T_j}(S_{j-1})] = \frac{\varepsilon}{2},
\end{equation}
and the first iteration allows us to move forward by a time step
\begin{equation}\label{Deltaj}
  \Delta_j = n_jT_j \sim \frac{1}{\log \tilde D_{T_j}(S_{j-1})},
\end{equation}
so we reach the time $S_j =S_{j-1}+\Delta_j$, at which the data norm can at most have tripled in size:
\begin{equation}\label{Triplej}
  \tilde D_{T_j}(S_j) \le 3\tilde D_{T_j}(S_{j-1}).
\end{equation}

But in order to relate $\Delta_{j+1}$ to $\Delta_j$, we need to compare $\tilde D_{T_{j+1}}(S_{j})$ and $\tilde D_{T_j}(S_{j-1})$, whereas \eqref{Triplej} only provides a comparison of $\tilde D_{T_{j}}(S_{j})$ and $\tilde D_{T_j}(S_{j-1})$. We bridge this gap by the following argument: 
\begin{itemize}
\item
If $T_{j+1} \le T_j$, then Lemma \ref{MagicLemma} gives
$$
  \tilde D_{T_{j+1}}(S_{j})
  \le C \tilde D_{T_j}(S_{j})
  \le 3C\tilde D_{T_j}(S_{j-1}),
$$
where we used \eqref{Triplej} at the end.
\item
If $T_{j+1} > T_j$, comparison of \eqref{Tj} for $j$ and $j+1$ gives
$$
  \tilde D_{T_{j+1}}(S_{j}) < \tilde D_{T_j}(S_{j-1}).
$$
\end{itemize}
Thus, in both cases,
$$
  \tilde D_{T_{j+1}}(S_{j})
  \le 3C\tilde D_{T_j}(S_{j-1})
$$
for $j \ge 1$, and induction gives
$$
  \tilde D_{T_{j+1}}(S_{j})
  \le (3C)^{j} \tilde D_{T_1}(0)
$$
for $j \ge 0$, so by \eqref{Deltaj},
$$
  \Delta_{j+1} \gtrsim \frac{1}{\log( (3C)^{j} \tilde D_{T_1}(0))}
  \sim \frac{1}{j+1}
$$
for $j \ge 0$, hence
$$
  \sum_{j=0}^\infty \Delta_{j+1} = \infty,
$$
proving global existence.

\section{Preliminaries}\label{Notation}

In this section we prepare the ground for the proof of Theorem \ref{LWPnew}.

\subsection{Function spaces} As is usual, we split
$$
  \psi=\psi_+ + \psi_-, \qquad \psi_\pm \equiv \mathbf\Pi_\pm \psi,
$$
using the Dirac projections $\mathbf\Pi_\pm = \mathbf\Pi(\pm D)$, defined in terms of the symbol
$$
  \mathbf\Pi(\xi)
  =
  \frac12 \bigl( \mathbf I_{2 \times 2} + \frac{\xi^j}{\abs{\xi}} \boldsymbol\alpha_j \bigr).
$$
The projections are self-adjoint and orthogonal, i.e.~$\mathbf\Pi_+ \mathbf\Pi_- = \mathbf\Pi_- \mathbf\Pi_+ = 0$, so in particular
$
  \norm{\psi(t)}^2 = \norm{\psi_+(t)}^2 + \norm{\psi_-(t)}^2.
$

Now \eqref{MDreduced} splits into two equations:
\begin{equation}\label{SplitMD}
  \left(-i\partial_t\pm\abs{D}\right) \psi_\pm
  =
  - \mathbf\Pi_\pm ( M \boldsymbol\beta \psi )
  + \mathbf\Pi_\pm \left( A_\mu^{\mathrm{hom.}} \boldsymbol\alpha^\mu \psi \right)
  - \mathbf\Pi_\pm \mathcal N(\psi,\psi,\psi),
\end{equation}
and we introduce $X^{s,b}$ spaces corresponding to $\left(-i\partial_t\pm\abs{D}\right)$. More generally, consider an equation of the form
$$
  \left[-i\partial_t+\phi(D)\right] u
  = F,
$$
where $\phi \colon \R^2 \to \R$ is a given function. Define $X_{\phi(\xi)}^{s,b}$ (for $s,b \in \R$)  as the completion of $\mathcal{S}(\R^{1+2})$ with respect to the norm
$$
  \norm{u}_{X_{\phi(\xi)}^{s,b}} = \bignorm{ \angles{\xi}^s
  \angles{\tau+\phi(\xi)}^b \,\widetilde
  u(\tau,\xi)}_{L^2_{\tau,\xi}},
$$
where
$$
  \angles{\xi} = (1+\abs{\xi}^2)^{1/2}.
$$
In fact, we use either $\phi(\xi) = \pm\abs{\xi}$ or $\phi(\xi) = \pm\angles{\xi}$, but since $\angles{\tau\pm\abs{\xi}} \sim \angles{\tau\pm\angles{\xi}}$, the corresponding norms are equivalent, hence the spaces $X_{\pm\abs{\xi}}^{s,b}$ and $X_{\pm\angles{\xi}}^{s,b}$ are identical, and we denote them simply by $X_\pm^{s,b}$.

Estimating $\psi_\pm$ in $X_\pm^{s,b}$, however, one can only get the estimates in Theorems \ref{LWP} and \ref{EMgrowth} with $T^{1/2}$ replaced by $T^{1/2-\delta}$ for arbitrarily small $\delta > 0$. To avoid this loss, we use instead some Besov versions of $X^{s,b}_\pm$, as  was done in \cite{Gruenrock:2009}. Similar spaces have been used in \cite{Bejenaru:2009} and \cite{Colliander:2003}.

Specifically, we shall use $X_{\phi(\xi)}^{s,b;1}$ and $X_{\phi(\xi)}^{s,b;\infty}$, defined as the completions of $\mathcal{S}(\R^{1+2})$ with respect to the norms
\begin{align*}
  \norm{u}_{X_{\phi(\xi)}^{s,b;1}}
  &=
  \sum_{L \ge 1} L^b \norm{\angles{D}^s\Proj_{\angles{\tau+\phi(\xi)} \sim L} u},
  \\
  \norm{u}_{X_{\phi(\xi)}^{s,b;\infty}}
  &=
  \sup_{L \ge 1} L^b \norm{\angles{D}^s \Proj_{\angles{\tau+\phi(\xi)} \sim L} u},
\end{align*}
where $L \ge 1$ is restricted to the dyadic numbers. The spaces corresponding to $\phi(\xi) = \pm\abs{\xi}$ or $\phi(\xi) = \pm\angles{\xi}$ coincide, and we simply write
$$
  X_\pm^{s,b;p} = X_{\pm\abs{\xi}}^{s,b;p} = X_{\pm\angles{\xi}}^{s,b;p}.
$$

Restriction to the time-slab
$$
  S_T = (-T,T) \times \R^2
$$
is handled in the usual way. Define
 $$
  \norm{u}_{X_{\phi(\xi)}^{s,b;p}(S_T)}
  =
  \inf_{\text{$v = u$ on $S_T$}} \norm{v}_{X_{\phi(\xi)}^{s,b;p}}.
$$
This is a seminorm on $X_{\phi(\xi)}^{s,b;p}$, but becomes a norm if we identify elements which agree on $S_T$, and the resulting space is denoted $X_{\phi(\xi)}^{s,b;p}(S_T)$. In other words, $X_{\phi(\xi)}^{s,b;p}(S_T)$ is the quotient $X_{\phi(\xi)}^{s,b;p}/\mathcal M$, where $\mathcal M = \{ v \in X_{\phi(\xi)}^{s,b;p} : \text{$v = 0$ on $S_T$} \}$. Since $\mathcal M$ is a closed subspace of $X_{\phi(\xi)}^{s,b;p}$, we conclude from general facts about quotient spaces (see e.g.~\cite[Section 5.1]{Folland:1999}) that $X_{\phi(\xi)}^{s,b;p}(S_T)$ a Banach space.

\subsection{Basic properties of $X^{s,b;p}_{\phi(\xi)}$}

First observe that
\begin{equation}\label{Embedding1}
  \norm{u}_{X_{\phi(\xi)}^{s,b;1}} \le C_{b,b'} \norm{u}_{X_{\phi(\xi)}^{s,b';\infty}}
  \qquad \text{for $b < b'$},
\end{equation}
since $\sum_{L \ge 1} L^{b-b'} < \infty$ for dyadic $L$'s.

Second, by standard methods one finds that
\begin{align}\label{Duality1}
  \norm{u}_{X_{\phi(\xi)}^{s,b;1}}
  &=
  \sup_{\text{$v$ s.t.~$\norm{v}_{X_{\phi(\xi)}^{-s,-b;\infty}} = 1$}}
  \abs{\int u \overline v \, dt \, dx},
  \\
  \label{Duality2}
  \norm{u}_{X_{\phi(\xi)}^{s,b;\infty}}
  &=
  \sup_{\text{$v$ s.t.~$\norm{v}_{X_{\phi(\xi)}^{-s,-b;1}} = 1$}}
  \abs{\int u \overline v \, dt \, dx},
\end{align}
and similarly for spinor-valued $u$ and $v$, replacing $u \overline v$ by $\Innerprod{u}{v}$.

Next, observing that by the Hausdorff-Young inequality followed by H\"older's inequality one has
$$
  \norm{\angles{D}^s\Proj_{\angles{\tau+\phi(\xi)} \sim L} u}_{L_t^p L_x^2}
  \lesssim L^{1/2-1/p} \norm{\angles{D}^s\Proj_{\angles{\tau+\phi(\xi)} \sim L} u}
  \qquad (2 \le p \le \infty),
$$
it follows that
\begin{equation}\label{Embedding2}
  \norm{u}_{L_t^p H^s}
  \lesssim
  \sum_L \norm{\angles{D}^s\Proj_{\angles{\tau+\phi(\xi)} \sim L} u}_{L_t^p L_x^2}
  \lesssim
  \norm{u}_{X_{\phi(\xi)}^{s,1/2-1/p;1}},
\end{equation}
implying the embedding
$$
  X_{\phi(\xi)}^{s,1/2;1} \hookrightarrow C_tH^s
$$
and also, writing $\rho_T(t) = \rho(t/T)$, where $\rho$ is a smooth cutoff function satisfying $\rho(t) = 1$ for $\abs{t} \le 1$ and $\rho(t) = 0$ for $\abs{t} \ge 2$,
\begin{equation}\label{Cutoff1}
  \norm{\rho_T u}
  \le \norm{\rho_T}_{L^p_t} \norm{u}_{L_t^{2p/(p-2)} L_x^2}
  \lesssim T^{1/p} \norm{u}_{X_{\phi(\xi)}^{0,1/p;1}} \qquad (2 \le p \le \infty).
\end{equation}
Moreover, one has (see \cite[Proposition 2.1(iii)]{Gruenrock:2009})
\begin{equation}\label{Cutoff2}
  \norm{\rho_T u}_{X_{\phi(\xi)}^{s,b;1}}
  \lesssim T^{1/2-b} \norm{u}_{X_{\phi(\xi)}^{s,1/2;1}} \qquad \text{for $0 < b \le 1/2$}.
\end{equation}

Finally, consider the solution of the initial value problem
\begin{equation}\label{IVP}
  \left[-i\partial_t +\phi(D)\right] u = F \quad \text{on $S_T$},
  \qquad
  u(0) = f,
\end{equation}
given (for sufficiently regular $f$ and $F$) by the Duhamel formula
\begin{equation}\label{Duhamel}
  u(t) = e^{-it\phi(D)}f + \int_0^t e^{-i(t-t')\phi(D)}F(t') \, dt'.
\end{equation}
Then for any $s \in \R$ and $0 < T \le 1$, the following estimates hold:
\begin{align}
  \label{Linear1}
  \norm{u}_{X_{\phi(\xi)}^{s,1/2;1}(S_T)}
  &\lesssim
  \norm{f}_{H^s}
  +
  \norm{F}_{X_{\phi(\xi)}^{s,-1/2;1}(S_T)},
  \\
  \label{Linear2}
  \norm{u}_{X_{\phi(\xi)}^{s,1/2;1}(S_T)}
  &\lesssim
  \norm{f}_{H^s}
  +
  T^{1/2+b}\norm{F}_{X_{\phi(\xi)}^{s,b;\infty}(S_T)}
  \qquad \text{for $-1/2 < b < 1/2$}.
\end{align}
See section \ref{LemmaProofs} for the proof, by standard methods. We remark that \eqref{Linear2} is included in \cite[Proposition 2.1]{Gruenrock:2009}, but only for $-1/2 < b < 0$.

Moreover, we will need:
\begin{equation}\label{Linear3}
  \sup_{t \in \R} \norm{\int_0^t e^{-i(t-t')\phi(D)}F(t') \, dt'}_{H^s}
  \lesssim
  \norm{\angles{\xi}^s \int \frac{\bigabs{\widetilde F(\tau,\xi)}}{\angles{\tau+\phi(\xi)}} \, d\tau}_{L^2_{\xi}},
\end{equation}
which is also proved in section \ref{LemmaProofs}.

\subsection{A Sobolev product estimate} We will need the following elementary fact:

\begin{lemma}\label{SobolevLemma}
If $a,b \in \R$ satisfy $a < 1$ and $a+b > 1$, then for all $f,g \in L^2(\R^2)$,
$$
  \norm{\abs{D}^{-a}\angles{D}^{-b}(fg)} \le C_{a,b} \norm{f}\norm{g}
$$
\end{lemma}

\begin{proof} Note that
\begin{equation}\label{BallSobolev}
  \norm{\Proj_{\abs{\xi_0} \lesssim N_0} (fg)}
  \lesssim N_0 \norm{f}\norm{g}
\end{equation}
by Plancherel and Cauchy-Schwarz:
$$
  \norm{\Proj_{\abs{\xi_0} \lesssim N_0} (fg)}
  \simeq
  \norm{\chi_{\abs{\xi_0} \lesssim N_0} \int \widehat f(\xi_1) \widehat g(\xi_0-\xi_1) \, d\xi_1}_{L^2_{\xi_0}}
  \lesssim
  \norm{\chi_{\abs{\xi_0} \lesssim N_0}}_{L^2_{\xi_0}} \bignorm{\widehat f\,} \norm{\widehat g}.
$$
Thus
\begin{align*}
  \norm{\abs{D}^{-a}\angles{D}^{-b}(fg)}
  &\le
  \sum_{0 < N_0 < 1} N_0^{-a} \norm{\Proj_{\abs{\xi_0} \sim N_0} (fg)}
  +
  \sum_{N_0 \ge 1} N_0^{-a-b} \norm{\Proj_{\abs{\xi_0} \sim N_0} (fg)}
  \\
  &\le
  \left(
  \sum_{0 < N_0 < 1} N_0^{1-a}
  +
  \sum_{N_0 \ge 1} N_0^{1-a-b} \right)
  \norm{f}\norm{g},
\end{align*}
and the last two sums are finite if and only if $a < 1$ and $a+b > 1$.
\end{proof}

In particular, we then obtain the following estimates for the current, already used in section \ref{Main} to see that the data for $\mathbf E^{\text{df}}_\pm$ are in the correct space. First,
\begin{equation}\label{Jest1}
  \sum_{0 < N_0 < 1} \norm{\Proj_{\abs{\xi_0} \sim N_0} \mathbf J(t)}
  \lesssim
  \sum_{0 < N_0 < 1} N_0 \norm{\psi(t)}^2 \sim \norm{\psi(t)}^2,
\end{equation}
where \eqref{BallSobolev} was used. Second,
\begin{equation}\label{Jest2}
  \norm{\mathbf J(t)}_{H^{-3/2}} \lesssim \norm{\psi(t)}^2,
\end{equation}
by Lemma \ref{SobolevLemma}.

\subsection{Some special sets} For $N,L \ge 1$, $r,\gamma > 0$ and $\omega \in \mathbb S^1$, where $\mathbb S^1 \subset \R^2$ is the unit circle, define
\begin{align*}
  \Gamma_\gamma(\omega) &= \left\{ \xi \in \R^2 \colon \theta(\xi,\omega) \le \gamma \right\},
  \\
  T_r(\omega) &= \left\{ \xi \in \R^2 \colon \abs{P_{\omega^{\perp}} \xi} \lesssim r \right\},
  \\
  K^\pm_{L}
  &= \left\{ (\tau,\xi) \in \R^{1+2} \colon \angles{\tau\pm\abs{\xi}} \sim L \right\},
  \\
  K^\pm_{N,L}
  &= \left\{ (\tau,\xi) \in \R^{1+2} \colon \angles{\xi} \sim N, \; \angles{\tau\pm\abs{\xi}} \sim L \right\},
  \\
  K^{\pm}_{N,L,\gamma}(\omega)
  &= \left\{ (\tau,\xi) \in \R^{1+2} \colon \angles{\xi} \sim N, \; \pm\xi \in \Gamma_{\gamma}(\omega), \; \angles{\tau\pm\abs{\xi}} \sim L \right\},
  \\
  H_d(\omega) &= \left\{ (\tau,\xi) \in \R^{1+2} \colon \abs{\tau+\xi\cdot\omega} \lesssim d \right\},
\end{align*}
where $\theta(a,b)$ denotes the angle between nonzero vectors $a,b \in \R^2$ and $P_{\omega^{\perp}}$ is the projection onto the orthogonal complement $\omega^\perp$ of $\omega$ in $\R^2$. For later use we note the elementary fact (see \cite{Selberg:2008b}) that
\begin{equation}\label{B:112}
  K^{\pm}_{N,L,\gamma}(\omega)
  \subset H_{\max(L,N\gamma^2)}(\omega).
\end{equation}

We shall also need the following:

\begin{lemma}\label{HyperLemma} Suppose $N,d,\gamma > 0$. The estimate
$$
  \sum_{\omega \in \Omega(\gamma)} \chi_{H_d(\omega)}(\tau,\xi)
  \lesssim 1 + \left(\frac{d}{N\gamma^2}\right)^{1/2}
$$
holds for all $(\tau,\xi) \in \R^{1+2}$ with $\abs{\xi} \sim N$.
\end{lemma}

\begin{proof}
The left side equals $\# \left\{ \omega \in \Omega(\gamma) \colon \omega \in A \right\}$ where $A$ is the set of $\omega \in \mathbb S^1$ such that $\abs{\tau+\xi\cdot\omega} \lesssim d$. Without loss of generality assume $\xi = (\abs{\xi},0)$. Then
$$
  A \subset A' \equiv \left\{ \omega = (\omega^1,\omega^2) \in \mathbb S^1 \colon \omega^1 = -\frac{\tau}{\abs{\xi}} + O\left(\frac{d}{N}\right) \right\}.
$$
Thus, $A'$ is the intersection of $\mathbb S^1$ and a strip of thickness comparable to $d/N$, so
$$
  \# \left\{ \omega \in \Omega(\gamma)\colon \omega \in A' \right\}
  \lesssim 1 + \frac{\text{length}(A')}{\gamma}.
$$
But $\text{length}(A') \lesssim (d/N)^{1/2}$, and the proof is complete.
\end{proof}

\subsection{Angular decompositions} For $\gamma \in (0,\pi]$, let $\Omega(\gamma)$ denote a maximal $\gamma$-separated subset of the unit circle. We recall the following angular Whitney decomposition:

\begin{lemma}\label{WhitneyLemma1} We have
$$
  1 \sim \sum_{\genfrac{}{}{0pt}{1}{0 < \gamma < 1}{\text{$\gamma$ \emph{dyadic}}}} 
  \sum_{\genfrac{}{}{0pt}{1}{\omega_1,\omega_2 \in \Omega(\gamma)}{3\gamma \le \vangle(\omega_1,\omega_2) \le 12\gamma}}
  \chi_{\Gamma_\gamma(\omega_1)}(\xi_1) \chi_{\Gamma_\gamma(\omega_2)}(\xi_2),
$$
for all $\xi_1,\xi_2 \in \R^2 \setminus \{0\}$ with $\vangle(\xi_1,\xi_2) > 0$.
\end{lemma}

The straightforward proof is omitted. The condition $\vangle(\omega_1,\omega_2) \ge 3\gamma$ ensures that the sectors $\Gamma_{\gamma}(\omega_1)$ and $\Gamma_{\gamma}(\omega_2)$ are well-separated. If separation is not needed, it is better to use the following variation (again, we skip the easy proof):

\begin{lemma}\label{WhitneyLemma2} For any $0 < \gamma < 1$ and $k \in \N$,
$$
  \chi_{\vangle(\xi_1,\xi_2) \le k\gamma} \lesssim \sum_{\genfrac{}{}{0pt}{1}{\omega_1,\omega_2 \in \Omega(\gamma)}{\vangle(\omega_1,\omega_2) \le (k+2)\gamma}}
  \chi_{\Gamma_\gamma(\omega_1)}(\xi_1) \chi_{\Gamma_\gamma(\omega_2)}(\xi_2),
$$
for all $\xi_1,\xi_2 \in \R^2 \setminus \{0\}$.
\end{lemma}

Writing $u^{\gamma,\omega} = \Proj_{\pm\xi \in \Gamma_{\gamma}(\omega)} u$ for a given sign, we note that
\begin{equation}\label{OmegaSum0}
  \norm{u}^2 \sim \sum_{\omega \in \Omega(\gamma)} \norm{u^{\gamma,\omega}}^2
\end{equation}
and (given signs $\pm_1$ and $\pm_2$)
\begin{equation}\label{OmegaSum}
  \sum_{\genfrac{}{}{0pt}{1}{\omega_1,\omega_2 \in \Omega(\gamma)}
  {\theta(\omega_1,\omega_2) \lesssim \gamma}}
  \norm{u_1^{\gamma,\omega_1}} \norm{u_2^{\gamma,\omega_2}}
  \lesssim \norm{u_1}\norm{u_2},
\end{equation}
where we used the Cauchy-Schwarz inequality, \eqref{OmegaSum0} and the fact that, given $\omega_2$, the set of $\omega_1 \in \Omega(\gamma)$ satisfying $\theta(\omega_1,\omega_2) \le k\gamma$ has cardinality at most $2k+1$.

\section{Local well-posedness}\label{LWPsection}

The iterates $\{\psi^{(n)}_\pm\}_{n=-1}^\infty$ for \eqref{SplitMD} are defined in the standard way, i.e.~$\psi^{(-1)}_\pm$ is taken to be identically zero, and in the general inductive step, $\psi^{(n)}_\pm$ is obtained by solving \eqref{SplitMD} on $S_T$ with the previous iterate $\psi^{(n-1)}_\pm$ inserted on the right hand side, and with initial data $\mathbf\Pi_\pm\psi_0$. Note that $\mathbf\Pi_\pm \psi^{(n)}_\pm = \psi^{(n)}_\pm$ on $S_T$. 

We shall estimate the iterates in the norm
$$
  p_n(T) = \bignorm{\psi^{(n)}_+}_{X^{0,1/2;1}_+(S_T)} + \bignorm{\psi^{(n)}_-}_{X^{0,1/2;1}_-(S_T)},
$$
where $T > 0$ remains to be fixed. We also need estimates for the difference of two successive iterates,
$$
  q_n(T) = \sum_\pm \bignorm{\psi^{(n)}_\pm - \psi^{(n-1)}_\pm}_{X^{0,1/2;1}_\pm(S_T)}.
$$

We claim that to prove Theorem \ref{LWP}, it suffices to show, for $0 < T \le 1$,
\begin{align}
  \label{Contraction1}
  p_{n+1}(T) &\le C_1 + C_2T^{1/2}[1+D_T(0)]p_n(T) + C_3T^\delta p_n(T)^3,
  \\
  \label{Contraction2}
  q_{n+1}(T) &\le C_2T^{1/2}[1+D_T(0)]q_n(T) + C_3T^\delta p_n(T)^2 q_n(T),
\end{align}
where $C_1$ and $C_2$ depend on the charge constant, $C_2$ depends in addition on $\abs{M}$, $C_3$ is an absolute constants, and $\delta > 0$ is some small number.

In fact, the verification of the above claim consists of a completely standard argument, which we only sketch here.

First one uses \eqref{Contraction1} to verify that
\begin{equation}\label{IterationBound}
  p_n(T) \le 2C_1
\end{equation}
for all $n$ if $T > 0$ is small enough. Indeed, this clearly holds for $n=-1$ and all $0 < T \le 1$, and then it follows for all $n \ge 0$ by induction, provided that $2C_2 T^{1/2}[1+D_T(0)] \le 1/2$ and $8C_1^2C_3T^\delta \le 1/2$. The latter condition simply says that $T \le \varepsilon$ for some $\varepsilon > 0$ depending only on the charge constant, whereas the former (and stronger) condition says that
$$
  T^{1/2} [1+D_T(0)] \le \varepsilon
$$
for some $\varepsilon > 0$ depending only on the charge constant and $M$, so this is exactly condition \eqref{MDtime} in Theorem \ref{LWP}.

Second one uses \eqref{Contraction2} to verify that, with the same condition on $T$, the sequence of iterates $\psi^{(n)}_\pm$ is Cauchy in $X^{0,1/2;1}_\pm(S_T)$, hence converges in that space to a solution of 2d MD on $S_T=(-T,T) \times \R^2$. Indeed, \eqref{Contraction2} implies $q_{n+1}(T) \le \frac12 q_{n+1}(T)$.

This proves the local existence part of Theorem \ref{LWP}. Uniqueness in the iteration space follows by \eqref{Contraction2} (or rather its analog for the difference of any two solutions instead of two iterates). Finally, continuous dependence on the data and persistence of higher regularity follow from standard arguments which we do not repeat here.

Note that the same argument immediately gives Theorem \ref{LWPnew}, since we can apply the estimate $D_T(0) \le \tilde D_T(0)$ in the right hand sides of \eqref{Contraction1} and \eqref{Contraction2}.

So we need to prove \eqref{Contraction1} and \eqref{Contraction2}.

The first term on the right hand side of \eqref{Contraction1} comes from applying \eqref{Linear1} to the homogeneous part $\psi^{(0)}_\pm$ of $\psi^{(n+1)}_\pm$, while the remaining terms come from the inhomogeneous part, which we split into three parts corresponding to the three terms on the right hand side of \eqref{SplitMD}. Applying \eqref{Linear2} with $b=0$ and $b=-1/4$, respectively, to the first two terms, and \eqref{Linear1} to third, we reduce \eqref{Contraction1} [and in fact also \eqref{Contraction2}, since all the terms in \eqref{SplitMD} are either linear or trilinear in $\psi$] to the following three estimates, where $\pm_1,\dots,\pm_4$ denote independent signs and the implicit constants are absolute:

First, we need
$$
  \norm{M \mathbf\Pi_{\pm_2}\boldsymbol\beta \psi}_{X_{\pm_2}^{0,0;\infty}(S_T)}
  \lesssim \abs{M} \norm{\psi}_{X_{\pm_1}^{0,1/2;1}(S_T)},
$$
but this is trivial since $X_{\pm_2}^{0,0;\infty} = X_{\pm_1}^{0,0;\infty}$. Second, we need
$$
  \norm{\mathbf\Pi_{\pm_2}\left( A_\mu^{\mathrm{hom.}} \boldsymbol\alpha^\mu \mathbf\Pi_{\pm_1} \psi_1 \right)}_{X_{\pm_2}^{0,-1/4;\infty}(S_T)}
  \lesssim T^{1/4} [\norm{\psi_0}^2+D_T(0)] \norm{\psi_1}_{X_{\pm_1}^{0,1/2;1}(S_T)},
$$
and third,
$$
  \bignorm{\mathbf\Pi_{\pm_4}\mathcal N\bigl(\mathbf\Pi_{\pm_1}\psi_1,\mathbf\Pi_{\pm_2}\psi_2,\mathbf\Pi_{\pm_3}\psi_3\bigr)}_{X_{\pm_4}^{0,-1/2;1}(S_T)}
  \lesssim T^\delta \prod_{j=1}^3 \norm{\psi_j}_{X_{\pm_j}^{0,1/2;1}(S_T)}.
$$
It suffices to prove these without the restriction to $S_T = (-T,T) \times \R^2$, but of course we can then insert a smooth time cutoff $\rho_T(t) = \rho(t/T)$, where $\rho(t)=1$ for $\abs{t} \le 1$ and $\rho(t)=0$ for $\abs{t} \ge 2$. By \eqref{Duality1} and \eqref{Duality2} we therefore reduce to proving
\begin{equation}\label{Trilinear}
  \abs{I^{\pm_1,\pm_2}}
  \lesssim T^{1/4} [\norm{\psi_0}^2+D_T(0)] \norm{\psi_1}_{X_{\pm_1}^{0,1/2;1}}
  \norm{\psi_2}_{X_{\pm_2}^{0,1/4;1}}
\end{equation}
and
\begin{equation}\label{Quadrilinear}
  \abs{J^{\pm_1,\dots,\pm_4}}
  \lesssim T^\delta \norm{\psi_1}_{X_{\pm_1}^{0,1/2;1}}
  \norm{\psi_2}_{X_{\pm_2}^{0,1/2;1}}
  \norm{\psi_3}_{X_{\pm_3}^{0,1/2;1}}
  \norm{\psi_4}_{X_{\pm_4}^{0,1/2;\infty}},
\end{equation}
where
\begin{align*}
  I^{\pm_1,\pm_2}
  &=
  \int \rho A_\mu^{\mathrm{hom.}}
  \Innerprod{\boldsymbol\alpha^\mu \mathbf\Pi_{\pm_1} \psi_1}{\mathbf\Pi_{\pm_2} \psi_2} \, dt \, dx,
  \\
  J^{\pm_1,\dots,\pm_4}
  &=
  \int \rho \square^{-1} \Innerprod{\boldsymbol\alpha^\mu\mathbf\Pi_{\pm_1}\psi_1}{\mathbf\Pi_{\pm_2}\psi_2} \cdot\Innerprod{\boldsymbol\alpha_\mu\mathbf\Pi_{\pm_3}\psi_3}{\mathbf\Pi_{\pm_4}\psi_4} \, dt \, dx,
\end{align*}
and the $\psi_j \in \mathcal S(\R^{1+2})$ are $\C^2$-valued. Moreover, we can freely replace $\psi_j$ by $\rho_T\psi_j$ in the above integrals whenever it may be needed.

We concentrate first on the quadrilinear estimate \eqref{Quadrilinear}, proved in the next four sections by adapting the proof of the analogous estimate in 3d from \cite{Selberg:2008b}. We make a dyadic decomposition, use the null structure of the quadrilinear form in the integral, reduce to various $L^2$ bilinear estimates, and finally sum the dyadic pieces to obtain \eqref{Quadrilinear}. The main difference from the 3d case is that the $L^2$ bilinear estimates are different in 2d; the estimates we need have been proved by the second author in \cite{Selberg:2010a}. The trilinear estimate \eqref{Trilinear} is proved in section \ref{S}.

\section{The quadrilinear estimate}\label{QuadrilinearSection}

Here we begin the proof of \eqref{Quadrilinear}. First we switch to Fourier variables in $J^{\pm_1,\dots,\pm_4}$ by Plancherel's theorem. To this end we recall the following representation of $\square^{-1}$,  derived from Duhamel's formula (see \cite[Lemma 4.4]{Klainerman:1995b}).

\begin{lemma}\label{DuhamelLemma} Given $G \in \mathcal S(\R^{1+2})$, set $u = \square^{-1} G$ and consider the splitting $u = u_+ + u_-$ defined by \eqref{WaveSplitting}. Then
$$
  \widehat{u_\pm}(t,\xi) = \pm \frac{e^{\mp it \abs{\xi}}}{4\pi\abs{\xi}} \int_{-\infty}^\infty \frac{e^{it(\tau'\pm\abs{\xi})}-1}{\tau'\pm\abs{\xi}} \widetilde G(\tau',\xi) \, d\tau'.
$$
Moreover, multiplying by the cutoff $\rho(t)$ and taking Fourier transform also in time,
$$
  \widetilde{\rho u_\pm}(\tau,\xi)
  \\
  = \int_{-\infty}^\infty
  \frac{\kappa_{\pm}(\tau,\tau';\xi)}{4\pi\abs{\xi}}
  \widetilde G(\tau',\xi) \, d\tau',
$$
where
$$
  \kappa_{\pm}(\tau,\tau';\xi)
  = \pm \frac{\widehat\rho(\tau-\tau')-\widehat\rho(\tau\pm\abs{\xi})}{\tau'\pm\abs{\xi}}
$$
and $\widehat \rho(\tau)$ denotes the Fourier transform of $\rho(t)$.
\end{lemma}

Thus, writing $\widetilde{\psi_j} = z_j \bigabs{\widetilde{\psi_j}}$, where $z_j : \R^{1+2} \to \C^2$ with $\abs{z_j}=1$, and applying the convolution formula
\begin{equation}\label{Convolution}
  \widetilde{u_1\overline{u_2}}(X_0)
  \simeq \int
  \widetilde{u_1}(X_1)\,
  \overline{\widetilde{u_2}(X_2)}
  \, d\mu^{12}_{X_0},
  \quad
  d\mu^{12}_{X_0} \equiv \delta(X_0-X_1+X_2) \, dX_1 \, dX_2,
\end{equation}
twice, we see that it suffices to prove \eqref{Quadrilinear} for
$$
  J^{\pm_0,\pm_1,\dots,\pm_4}
  =
  \int
  \frac{\kappa_{\pm_0}(\tau_0,\tau_0';\xi_0)}{\abs{\xi_0}}
  \, q_{1234} \, \bigabs{\widetilde{\psi_j}(X_j)}
  \, d\mu^{12}_{X_0'} \, d\mu^{43}_{X_0}
  \, d\tau_0' \, d\tau_0 \, d\xi_0,
$$
where $X_0' = (\tau_0',\xi_0)$, $X_0 = (\tau_0,\xi_0)$, $X_j = (\tau_j,\xi_j)$ for $j=1,\dots,4$,
$$
  q_{1234} = \innerprod{\boldsymbol\alpha^\mu\mathbf\Pi(e_1)z_1(X_1)}{\mathbf\Pi(e_2)z_2(X_2)}
  \innerprod{\boldsymbol\alpha_\mu\mathbf\Pi(e_3)z_3(X_3)}{\mathbf\Pi(e_4)z_4(X_4)}
$$
and $e_j = \pm_j \xi_j/\abs{\xi_j}$. We may restrict the integration to $\xi_j \neq 0$ for $j=0,\dots,4$, hence the unit vectors $e_j$ are well-defined, as are the angles $$\theta_{jk} = \vangle(e_j,e_k) = \vangle(\pm_j\xi_j,\pm_k\xi_k),$$ in terms of which the null structure of $q_{1234}$ will be expressed. Note that
\begin{gather*}
  X_0' = X_1 - X_2, \qquad
  X_0 = X_4 - X_3,
  \\
  \tau_0'=\tau_1-\tau_2, \qquad \tau_0=\tau_4-\tau_3,
  \qquad
  \xi_0=\xi_1-\xi_2=\xi_4-\xi_3,
\end{gather*}
in the above integral. For simplicity we will just write $J$ instead of $J^{\pm_0,\pm_1,\dots,\pm_4}$ from now on. Split
$$
  J
  =
  J_{\abs{\xi_0} < 1}
  +
  J_{\abs{\xi_0} \ge 1}
$$
by restricting the integration to $\abs{\xi_0} < 1$ and $\abs{\xi_0} \ge 1$, respectively. We first dispose of the easy low frequency part.

\subsection{Estimate for $J_{\abs{\xi_0} < 1}$}

From Plancherel's theorem one infers
$$
  \norm{\Proj_{\abs{\xi} < 1} f} \le \abs{B(0,1)}^{1/2} \norm{f}_{L^1},
$$
where $B(0,1) = \{ \xi \in \R^2 \colon \abs{\xi} < 1 \}$. Applying also $\bignorm{\rho{\square}^{-1}F} \lesssim \norm{F}$, which follows from \cite[Lemma 4.3]{Klainerman:1995b}, we estimate
\begin{align*}
  J_{\abs{\xi_0} < 1} &\le \bignorm{\rho\square^{-1} \Proj_{\abs{\xi} < 1}\innerprod{\boldsymbol\alpha^\mu\mathbf\Pi_{\pm_1}\psi_1}{\mathbf\Pi_{\pm_2}\psi_2}} \norm{\Proj_{\abs{\xi} < 1} \innerprod{\boldsymbol\alpha_\mu\mathbf\Pi_{\pm_3}\psi_3}{\mathbf\Pi_{\pm_4}\psi_4}}
  \\
  &\lesssim \norm{\Proj_{\abs{\xi} < 1}\innerprod{\boldsymbol\alpha^\mu\mathbf\Pi_{\pm_1}\psi_1}{\mathbf\Pi_{\pm_2}\psi_2}} \norm{\Proj_{\abs{\xi} < 1} \innerprod{\boldsymbol\alpha_\mu\mathbf\Pi_{\pm_3}\psi_3}{\mathbf\Pi_{\pm_4}\psi_4}}
  \\
  &\lesssim \norm{\innerprod{\boldsymbol\alpha^\mu\mathbf\Pi_{\pm_1}\psi_1}{\mathbf\Pi_{\pm_2}\psi_2}}_{L_t^2L_x^1}
  \norm{\innerprod{\boldsymbol\alpha_\mu\mathbf\Pi_{\pm_3}\psi_3}{\mathbf\Pi_{\pm_4}\psi_4}}_{L_t^2L_x^1}
  \\
  &\lesssim \norm{\psi_1}_{L_t^4L_x^2} \norm{\psi_2}_{L_t^4L_x^2} \norm{\psi_3}_{L_t^4L_x^2} \norm{\psi_4}_{L_t^4L_x^2}.
\end{align*}
Recalling that we can replace $\psi_j$ by $\rho_T\psi_j$, we then get the desired estimate \eqref{Quadrilinear} for the low frequency part by applying \eqref{Embedding2} and \eqref{Cutoff2} to the norms of $\psi_1$, $\psi_2$ and $\psi_3$, whereas for $\psi_4$ we use \eqref{Embedding2} followed by \eqref{Embedding1}.

\subsection{Dyadic decomposition of $J_{\abs{\xi_0} \ge 1}$}

Letting $N$'s and $L$'s denote dyadic numbers greater than or equal to one, we assign dyadic sizes to the weights, writing $\angles{\tau_0'\pm_0\abs{\xi_0}} \sim L_0'$, $\angles{\tau_j\pm_j\abs{\xi_j}} \sim L_j$ and $\angles{\xi_j} \sim N_j$ for $j=0,\dots,4$, and we set $\boldN = (N_0,\dots,N_4)$ and $\boldL = (L_0,L_0',L_1,\dots,L_4)$. We shall use the shorthand $\Nmin^{012}$ for the minimum of $N_0$, $N_1$ and $N_2$, and similarly for other index sets than $012$, for the $L$'s, and for maxima. Since $\xi_0=\xi_1-\xi_2$ in $J$, one of the following must hold:
$$
  \begin{alignedat}{2}
  N_0 &\ll N_1 \sim N_2& \qquad &(\text{``low output''}),
  \\
  N_0 &\sim \Nmax^{12} \ge \Nmin^{12}& \qquad &(\text{``high output''}),
\end{alignedat}
$$
and similarly for the index 034. In particular, the two largest of $N_0$, $N_1$ and $N_2$ must be comparable, and $\Nmin^{012} \Nmax^{012} \sim N_0\Nmin^{12}$.

As shown in \cite{Selberg:2008b}, $\kappa_{\pm}(\tau_0,\tau_0';\xi_0) \lesssim (L_0L_0')^{-1/2}\sigma_{L_0,L_0'}(\tau_0-\tau_0')$, where
$$
  \sigma_{L_0,L_0'}(r) 
  =
  \begin{cases}
  \angles{r}^{-2} &\text{if $L_0 \sim L_0'$},
  \\
  (L_0L_0')^{-1/2} &\text{otherwise},
  \end{cases}
$$
hence
$$
  \abs{J_{\abs{\xi_0} \ge 1}}
  \lesssim \sum_{\boldN,\boldL} \frac{J_{\boldN,\boldL}}{N_0(L_0L_0')^{1/2}},
$$
where
\begin{multline*}
  J_{\boldN,\boldL}
  =
  \int
  \abs{q_{1234}}
  \sigma_{L_0,L_0'}(\tau_0-\tau_0') \,
  \chi_{K^{\pm_0}_{N_0,L_0}}\!\!(X_0) \,
  \chi_{K^{\pm_0}_{N_0,L_0'}}\!(X_0')
  \\
  \times
  \prod_{j=1}^4 \chi_{K^{\pm_j}_{N_j,L_j}}(X_j) \bigabs{\widetilde{\psi_j}(X_j)}
  \, d\mu^{12}_{X_0'} \, d\mu^{43}_{X_0}
  \, d\tau_0' \, d\tau_0 \, d\xi_0.
\end{multline*}
To ease the notation we define $u_j$ (implicitly depending on $N_j$, $L_j$ and $\pm_j$) by
$$
  \widetilde{u_j} = \chi_{K^{\pm_j}_{N_j,L_j}} \bigabs{\widetilde{\psi_j}}.
$$
Recall that $K^\pm_{N,L} = \left\{ (\tau,\xi) \in \R^{1+2} \colon \angles{\xi} \sim N, \; \angles{\tau\pm\abs{\xi}} \sim L \right\}$.
 
We claim that it suffices to prove, for some $\varepsilon > 0$,
\begin{equation}\label{MainDyadic}
  J_{\boldN,\boldL}
  \lesssim
  N_0^{1-\varepsilon} \left(L_0'L_0L_1 L_2 L_3 L_4 \right)^{1/2-\varepsilon}
  \norm{u_1} \norm{u_2} \norm{u_3} \norm{u_4}.
\end{equation}
Indeed, this gives
$$
  \abs{J_{\abs{\xi_0} \ge 1}}
  \lesssim \sum_{\boldN,\boldL} \frac{(L_1^{1/2-\varepsilon}\norm{u_1}) (L_2^{1/2-\varepsilon}\norm{u_2}) (L_3^{1/2-\varepsilon}\norm{u_3}) (L_4^{1/2}\norm{u_4})}{N_0^\varepsilon(L_0L_0'L_4)^\varepsilon},
$$
and we sum the $N$'s using the general estimate
\begin{equation}\label{DyadicSummation}
  \sum_{N_0,N_1,N_2} N_0^{-\varepsilon/2} a_{N_1} b_{N_2}
  \le C_\varepsilon \left( \sum_{N_1} a_{N_1}^2 \right)^{1/2}
  \left( \sum_{N_2} b_{N_2}^2 \right)^{1/2},
\end{equation}
valid for nonnegative sequences $a_{N_1}$, $b_{N_2}$ and dyadic $N_0,N_1,N_2 \ge 1$, the largest two of which are assumed comparable: By symmetry it suffices to consider $N_0 \lesssim N_1 \sim N_2$ and $N_1 \lesssim N_0 \sim N_2$. First, if $N_0 \lesssim N_1 \sim N_2$, then we sum $N_1 \sim N_2$ by Cauchy-Schwarz, and $N_0$ using $N_0^{-\varepsilon/2}$. Second, if $N_1 \lesssim N_0 \sim N_2$, then we estimate $N_0^{-\varepsilon/2} \lesssim N_0^{-\varepsilon/4}N_1^{-\varepsilon/2}$, so we can sum both $N_1$ and $N_0 \sim N_2$ without problems.

Applying \eqref{DyadicSummation} to the estimate for $\abs{J_{\abs{\xi_0} \ge 1}}$ above, we get
\begin{align*}
  \abs{J_{\abs{\xi_0} \ge 1}}
  &\lesssim \sum_{\boldL} (L_0L_0'L_4)^{-\varepsilon}
  \left( \prod_{j=1}^3 L_j^{1/2-\varepsilon} \norm{\Proj_{K^{\pm_j}_{L_j}} \psi_j} \right)
  L_4^{1/2} \norm{\Proj_{K^{\pm_4}_{L_4}} \psi_4}
  \\
  &\lesssim
  \norm{\psi_1}_{X^{0,1/2-\varepsilon;1}_{\pm_1}}
  \norm{\psi_2}_{X^{0,1/2-\varepsilon;1}_{\pm_2}}
  \norm{\psi_3}_{X^{0,1/2-\varepsilon;1}_{\pm_3}}
  \norm{\psi_4}_{X^{0,1/2;\infty}_{\pm_4}}.
\end{align*}
Since we may replace $\psi_j$ by $\rho_T\psi_j$, we now get \eqref{Quadrilinear} for $J_{\abs{\xi_0} \ge 1}$ by applying \eqref{Cutoff2} to the norms of $\psi_1$, $\psi_2$ and $\psi_3$.

So we have reduced \eqref{Quadrilinear} to proving the dyadic estimate \eqref{MainDyadic}. For this, we need to use the null structure of the quadrilinear form, obtained in \cite{Selberg:2008b}:

\begin{lemma}\label{NullLemma1} \emph{(\cite{Selberg:2008b}.)} Consider the symbol appearing in $J$,
$$
  q_{1234} = \innerprod{\boldsymbol\alpha^\mu\mathbf\Pi(e_1)z_1}{\mathbf\Pi(e_2)z_2}
  \innerprod{\boldsymbol\alpha_\mu\mathbf\Pi(e_3)z_3}{\mathbf\Pi(e_4)z_4},
$$
where the $e_j \in \R^2$ and $z_j \in \C^2$ are unit vectors. Defining the angles
$$
  \theta_{jk} = \vangle(e_j,e_k),
  \qquad
  \phi = \min \left\{ \theta_{13}, \theta_{14}, \theta_{23}, \theta_{24} \right\},
$$
we have
$$
  \abs{q_{1234}}
  \lesssim
  \theta_{12}\theta_{34} + \phi \max(\theta_{12},\theta_{34}) + \phi^2.
$$
\end{lemma}

When applying this, it is natural to distinguish the cases
\begin{gather}
  \label{D:97}
  \phi \lesssim \min(\theta_{12},\theta_{34}),
  \\
  \label{D:98}
   \min(\theta_{12},\theta_{34}) \ll \phi \lesssim \max(\theta_{12},\theta_{34}),
  \\
  \label{D:99}
  \max(\theta_{12},\theta_{34}) \ll \phi.
\end{gather}
In certain situations, the last two cases can be treated simultaneously, by virtue of the following simplified estimate:

\begin{lemma}\label{NullLemma2} \emph{(\cite{Selberg:2008b}.)} In cases \eqref{D:98} and \eqref{D:99},
$
  \abs{q_{1234}}
  \lesssim
  \theta_{13}\theta_{24}.
$
\end{lemma}

To end this section we prove the dyadic estimate \eqref{MainDyadic} in the case \eqref{D:97}. This particularly simple case essentially corresponds, as discussed in \cite{Selberg:2008b}, to solving the Dirac-Klein-Gordon system instead of Maxwell-Dirac. The cases \eqref{D:98} and \eqref{D:99} are far more difficult and will be handled in the next few sections.

\subsection{The case $\phi \lesssim \min(\theta_{12},\theta_{34})$}\label{SimpleCase}

Then
$$
  \abs{q_{1234}} \lesssim \theta_{12} \theta_{34},
$$
hence
$$
  J_{\boldN,\boldL}
  \lesssim
  \int
  T^{\pm_0}_{L_0,L_0'}\mathcal F \Proj_{K^{\pm_0}_{N_0,L_0'}}\mathfrak B_{\theta_{12}}(u_1,u_2)(X_0)
  \cdot
  \mathcal F \Proj_{K^{\pm_0}_{N_0,L_0'}}\mathfrak B_{\theta_{34}}(u_3,u_4)(-X_0)
  \, dX_0,
$$
where the null form $\mathfrak B_{\theta_{12}}(u_1,u_2)$ is defined on the Fourier transform side by inserting the angle $\theta_{12} = \theta(\pm_1\xi_1,\pm_2\xi_2)$ in the right hand side of the  convolution formula \eqref{Convolution}, and the operator $T^{\pm_0}_{L_0,L_0'}$ is defined by
$$
  T^{\pm_0}_{L_0,L_0'}F(\tau_0,\xi_0)
  =
  \int a^{\pm_0}_{L_0,L_0'}(\tau_0,\tau_0',\xi_0)
  F(\tau_0',\xi_0) \, d\tau_0',
$$
where
$$
  a^{\pm_0}_{L_0,L_0'}(\tau_0,\tau_0',\xi_0) 
  =
  \begin{cases}
  \angles{\tau_0-\tau_0'}^{-2} &\text{if $L_0 \sim L_0'$},
  \\
  (L_0L_0')^{-1/2}
  \chi_{\tau_0\pm_0\abs{\xi_0}=O(L_0)}
  \chi_{\tau_0'\pm_0\abs{\xi_0}=O(L_0')} &\text{otherwise}.
  \end{cases}
$$
This family of operators is uniformly bounded on $L^2$ (see \cite[Lemma 3.3]{Selberg:2008b}):

\begin{lemma}\label{Tlemma} $\bignorm{T^{\pm_0}_{L_0,L_0'}F} \lesssim \norm{F}$ for $F \in L^2(\R^{1+2})$.
\end{lemma}

Applying this, we get
$$
  J_{\boldN,\boldL}
  \lesssim
  \norm{\Proj_{K^{\pm_0}_{N_0,L_0'}}\mathfrak B_{\theta_{12}}(u_1,u_2)}
  \norm{\Proj_{K^{\pm_0}_{N_0,L_0'}}\mathfrak B_{\theta_{34}}(u_3,u_4)},
$$
and to finish we use the following null form estimate (proved in the next section):

\begin{lemma}\label{NullFormLemma} For all $u_1,u_2 \in L^2(\R^{1+2})$ such that $\widetilde{u_j}$ is supported in $K^{\pm_j}_{N_j,L_j}$,
$$
  \norm{\Proj_{K^{\pm_0}_{N_0,L_0}}\mathfrak B_{\theta_{12}}(u_1,u_2)}
  \lesssim
  \left( N_0 L_0 L_1 L_2 \right)^{3/8}
  \norm{u_1}\norm{u_2}.
$$
\end{lemma}

Thus,
$$
  J_{\boldN,\boldL}
  \lesssim
  \left( N_0 L_0' L_1 L_2 \right)^{3/8}
  \left( N_0 L_0 L_1 L_2 \right)^{3/8}
  \norm{u_1}\norm{u_2}\norm{u_3}\norm{u_4},
$$
proving \eqref{MainDyadic} in the case $\phi \lesssim \min(\theta_{12},\theta_{34})$.

In the next section we prepare the ground for the proof of the other cases, by recalling various bilinear and null form estimates proved in \cite{Selberg:2010a}. In particular, we prove Lemma \ref{NullFormLemma}.

For later use we record here the following variation on Lemma \ref{Tlemma}:

\begin{lemma}\label{Tlemma2} \emph{(\cite{Selberg:2008b}.)} Assume that $L_0 \ll L_0'$ or $L_0' \ll L_0$. Let $\omega, \omega' \in \mathbb S^1$, $c,c' \in \R$ and $d,d' > 0$. For $F,G \in L^2(\R^{1+2})$ satisfying
\begin{align*}
  \supp F &\subset \left\{ (\tau_0',\xi_0) \colon
  \tau_0'+\xi_0\cdot\omega' = c' + O(d') \right\},
  \\
  \supp G &\subset \left\{ (\tau_0,\xi_0) \colon
  \tau_0+\xi_0\cdot\omega = c + O(d) \right\},
\end{align*}
we have, for any $0 \le p \le 1/2$,
$$
  \bignorm{T^{\pm_0}_{L_0,L_0'}F}
  \lesssim \left(\frac{d'}{L_0'}\right)^{p} \norm{F},
$$
and
$$
  \abs{\int T^{\pm_0}_{L_0,L_0'}F(\tau_0,\xi_0) \cdot G(\tau_0,\xi_0) \, d\tau_0 \, d\xi_0}
  \lesssim 
  \left(\frac{dd'}{L_0L_0'}\right)^{p} \norm{F} \norm{G}.
$$
\end{lemma}

\section{Bilinear and null form estimates}\label{BilinearSection}

A key ingredient needed for the proof of Lemma \ref{NullFormLemma} is:

\begin{theorem}\label{BasicBilinearThm} \emph{(\cite{Selberg:2010a}.)} For all $u_1,u_2 \in L^2(\R^{1+2})$ such that $\widetilde{u_j}$ is supported in $K^{\pm_j}_{N_j,L_j}$, the estimate
$$
  \bignorm{\Proj_{K^{\pm_0}_{N_0,L_0}} ( u_1 \overline{u_2} )}
  \le C
  \norm{u_1}
  \norm{u_2}
$$
holds with
\begin{align}
  \label{Bilinear1}
  C
  &\sim \bigl( \Nmin^{012}\Lmin^{12} \bigr)^{1/2} \bigl( \Nmin^{12}\Lmax^{12} \bigr)^{1/4},
  \\
  \label{Bilinear2}
  C
  &\sim \bigl( \Nmin^{012}\Lmin^{0j} \bigr)^{1/2} \bigl( \Nmin^{0j}\Lmax^{0j} \bigr)^{1/4} \qquad (j=1,2),
  \\
  \label{Bilinear3}
  C
  &\sim \bigl( \Nmin^{012} \Nmin^{12} N_0 \Lmed^{012} \bigr)^{1/4} \bigl(\Lmin^{012}\bigr)^{1/2},
  \\
  \label{SobolevType}
  C
  &\sim \bigl( (\Nmin^{012})^2 \Lmin^{012} \bigr)^{1/2},
\end{align}
regardless of the choice of signs $\pm_j$.
\end{theorem}

The estimate \eqref{Bilinear3} is not included in \cite{Selberg:2010a}, but follows from either \eqref{Bilinear1} or \eqref{Bilinear2} and the fact that $\Nmin^{012} \Nmax^{012} \sim N_0\Nmin^{12}$.

Motivated by the convolution formula \eqref{Convolution}, a triple $(X_0,X_1,X_2)$ of vectors $X_j = (\tau_j,\xi_j) \in \R^{1+2}$ is said to be a \emph{bilinear interaction} if $X_0 = X_1 - X_2$. Given signs $(\pm_0,\pm_1,\pm_2)$ we also define the hyperbolic weights $\hypwt_j = \tau_j \pm_j \abs{\xi_j}$. If all three hyperbolic weights vanish, we say that the interaction is \emph{null}. If this happens, the vectors $X_j$ all lie on the null cone, and moreover it is clear geometrically that the angle $\theta_{12} = \theta(\pm_1\xi_1,\pm_2\xi_2)$ must vanish. The following more or less standard lemma generalizes this statement. For a proof, see e.g.~\cite{Selberg:2008c}.

\begin{lemma}\label{AnglesLemma} Given a bilinear interaction $(X_0,X_1,X_2)$ with $\xi_j \neq 0$, and signs $(\pm_0,\pm_1,\pm_2)$, define $\hypwt_j = \tau_j \pm_j \abs{\xi_j}$ and $\theta_{12} = \theta(\pm_1\xi_1,\pm_2\xi_2)$. Then
$$
  \max\left( \abs{\hypwt_0}, \abs{\hypwt_1}, \abs{\hypwt_2} \right)
  \gtrsim
  \min\left(\abs{\xi_1},\abs{\xi_2}\right)\theta_{12}^2.
$$
Moreover, we either have
$$
  \abs{\xi_0} \ll \abs{\xi_1} \sim \abs{\xi_2}
  \qquad \text{and} \qquad
  \pm_1 \neq \pm_2,
$$
in which case
$$
  \theta_{12} \sim 1
  \qquad \text{and} \qquad
  \max\left( \abs{\hypwt_0}, \abs{\hypwt_1}, \abs{\hypwt_2} \right)
  \gtrsim
  \min\left(\abs{\xi_1},\abs{\xi_2}\right),
$$
or else we have
$$
  \max\left( \abs{\hypwt_0}, \abs{\hypwt_1}, \abs{\hypwt_2} \right)
  \gtrsim
  \frac{\abs{\xi_1}\abs{\xi_2}\theta_{12}^2}{\abs{\xi_0}}.
$$
\end{lemma}

With this information in hand, we can prove Lemma \ref{NullFormLemma}. By Lemma \ref{AnglesLemma} we have $\theta_{12} \lesssim (\Lmax^{012}/\Nmin^{12})^p$ for $0 \le p \le 1/2$. Taking $p=3/8$ and using \eqref{Bilinear3},
$$
  \norm{\Proj_{K^{\pm_0}_{N_0,L_0}}\mathfrak B_{\theta_{12}}(u_1,u_2)}
  \lesssim
  \left(\frac{\Lmax^{012}}{\Nmin^{12}}\right)^{3/8}
  \left( N_0 \Nmin^{12} \Lmin^{012}\Lmed^{012} \right)^{3/8}
  \norm{u_1}\norm{u_2},
$$
proving Lemma \ref{NullFormLemma}.

The following improves the estimate \eqref{Bilinear1} in certain situations.

\begin{theorem}\label{Anisotropic} \emph{(\cite{Selberg:2010a}.)}
Let $\omega \in \mathbb S^1$, $0 < \alpha \ll 1$ and $I \subset \R$ a compact interval. Then for all $u_1,u_2 \in L^2(\R^{1+2})$ such that $\widetilde{u_j}$ is supported in $K^{\pm_j}_{N_j,L_j}$,
and assuming in addition that
$$
  \supp \widehat{u_1} \subset
  \left\{ (\tau,\xi) \colon
  \theta(\xi,\omega^\perp) \ge \alpha \right\},
$$
we have
$$
  \norm{\Proj_{\xi_0 \cdot \omega \in I}
  (u_1 u_2)}
  \lesssim \left( \frac{\abs{I}(\Nmin^{12})^{1/2}(L_1L_2)^{3/4}}{\alpha} \right)^{1/2}
  \norm{u_1}
  \norm{u_2}.
$$
The same estimate holds for $\norm{\Proj_{\xi_1 \cdot \omega \in I} u_1 \cdot u_2}$ and $\norm{u_1 \cdot \Proj_{\xi_2 \cdot \omega \in I} u_2}$. 
\end{theorem}

Here $\omega^\perp \subset \R^2$ is the orthogonal complement of $\omega$, and $\abs{I}$ is the length of $I$.

The next result is a null form estimate. Recall that $T_r(\omega) \subset \R^2$, for $r > 0$ and $\omega \in \mathbb S^1$, denotes a tube (actually a strip, since we are in the plane) of radius comparable to $r$ around $\R\omega$.

\begin{theorem}\label{NullRay} \emph{(\cite{Selberg:2010a}.)} Let $r > 0$ and $\omega \in \mathbb S^1$. Then for all $u_1,u_2 \in L^2(\R^{1+2})$ such that $\widetilde{u_j}$ is supported in $K^{\pm_j}_{N_j,L_j}$,
$$
  \bignorm{\mathfrak B_{\theta_{12}}(\Proj_{\R \times T_r(\omega)} u_1,u_2)}
  \lesssim \left( r L_1 L_2 \right)^{1/2}
  \norm{u_1}\norm{u_2}.
$$
\end{theorem}

The key point here is that we are able to exploit concentration of the Fourier supports near a null ray, which is not possible for the standard product $u_1 \overline{u_2}$. We remark that in \cite{Selberg:2010a}, the theorem is proved for $\abs{\xi_j} \sim N_j$ on the support of $\widetilde{u_j}$ instead of $\angles{\xi_j} \sim N_j$ as we have here. This only makes a difference if $\Nmin^{12} \sim 1$, but then the trivial estimate
$\bignorm{\Proj_{\R \times T_r(\omega)} u_1 \cdot \overline{u_2}} \lesssim \left( r \Nmin^{12} \Lmin^{12} \right)^{1/2} \norm{u_1}\norm{u_2}$ is stronger.

In the following refinement of Theorem \ref{NullRay} we limit attention to interactions which are nearly null, by restricting the angle to $\theta_{12} \ll 1$; the correspondingly modified null form is denoted $\mathfrak B_{\theta_{12} \ll 1}$.
 
\begin{theorem}\label{ImprovedNullRay} \emph{(\cite{Selberg:2010a}.)}
Let $r > 0$, $\omega \in \mathbb S^1$ and $I \subset \R$ a compact interval. Assume that $N_1, N_2 \gg 1$ and $r \ll \Nmin^{12}$. Then for all $u_1,u_2 \in L^2(\R^{1+2})$ such that $\widetilde{u_j}$ is supported in $K^{\pm_j}_{N_j,L_j}$,
$$
  \norm{\Proj_{\xi_0 \cdot \omega \in I} \mathfrak B_{\theta_{12} \ll 1}(\Proj_{\R \times T_r(\omega)} u_1,u_2)}
  \lesssim \left( r L_1 L_2 \right)^{1/2} 
  \left( \sup_{I_1} \norm{\Proj_{\xi_1 \cdot \omega \in I_1} u_1} \right)
  \norm{u_2},
$$
where the supremum is over all translates $I_1$ of $I$.
\end{theorem}

We end this section by recalling some facts, proved in \cite{Selberg:2008b}, about the bilinear interaction $X_0=X_1-X_2$, where we assume $\xi_j \neq 0$.  Given signs $(\pm_0,\pm_1,\pm_2)$, we  define as before the hyperbolic weights $\hypwt_j = \tau_j \pm_j \abs{\xi_j}$ and the angles $\theta_{jk} = \theta(\pm_j\xi_j,\pm_k\xi_k)$ for $j,k=0,1,2$.

In Lemma \ref{AnglesLemma} we related $\theta_{12}$ to the size of the weights $\hypwt_j$ and $\abs{\xi_j}$. The sign $\pm_0$ was arbitrary, but by keeping track of the sign we can get more. In fact, since $\tau_0=\tau_1-\tau_2$, we have
$
  \hypwt_0 - \hypwt_1 + \hypwt_2
  = \pm_0\abs{\xi_0}-\pm_1\abs{\xi_1} \pm_2\abs{\xi_2},
$
so defining
\begin{equation}\label{SignDef}
  \pm_{12}
  \equiv
  \begin{cases}
  + \qquad &\text{if $(\pm_1,\pm_2) = (+,+)$ and $\abs{\xi_1} > \abs{\xi_2}$},
  \\
  - \qquad &\text{if $(\pm_1,\pm_2) = (+,+)$ and $\abs{\xi_1} \le \abs{\xi_2}$},
  \\
  + \qquad &\text{if $(\pm_1,\pm_2) = (+,-)$},
  \end{cases}
\end{equation}
and correspondingly in the remaining cases $(\pm_1,\pm_2) = (-,-), (-,+)$ by reversing all three signs $(\pm_{12},\pm_1,\pm_2)$ above, it is clear that the following holds:

\begin{lemma}\label{F:Lemma2}
If $\pm_0 \neq \pm_{12}$, then $\max\left(\abs{\hypwt_0},\abs{\hypwt_1},\abs{\hypwt_2}\right)
  \gtrsim \abs{\xi_0}$.
\end{lemma}

In the remaining case $\pm_0 = \pm_{12}$ we have the following estimates.

\begin{lemma}\label{F:Lemma1} \emph{(\cite{Selberg:2008b}.)}
If $\pm_0 = \pm_{12}$, then
$$
  \min\left(\theta_{01},\theta_{02}\right)
  \sim
  \frac{\min\left(\abs{\xi_1},\abs{\xi_2}\right)}{\abs{\xi_0}} \sin \theta_{12}.
$$
Moreover, if $\pm_0 = \pm_{12}$ and $\pm_1 \neq \pm_2$, then
$$
  \max\left(\theta_{01},\theta_{02}\right)
  \sim
  \theta_{12}.
$$
\end{lemma}

\begin{lemma}\label{F:Lemma4} \emph{(\cite{Selberg:2008b}.)}
For all signs,
$$
  \max\left( \abs{\hypwt_0}, \abs{\hypwt_1}, \abs{\hypwt_2} \right)
  \gtrsim \abs{\xi_0} \min\left(\theta_{01},\theta_{02}\right)^2.
$$
\end{lemma}

\begin{lemma}\label{F:Lemma3} \emph{(\cite{Selberg:2008b}.)} If $\pm_0 = \pm_{12}$ and $\pm_1=\pm_2$, then
$$
  \frac{\abs{\xi_1}\abs{\xi_2}\theta_{12}^2}{\abs{\xi_0}} \sim \min\left( \abs{\xi_0}, \abs{\xi_1}, \abs{\xi_2}\right) \max(\theta_{01},\theta_{02})^2
  \lesssim
  \max\left( \abs{\hypwt_0}, \abs{\hypwt_1}, \abs{\hypwt_2} \right),
$$
whereas if $\pm_0 = \pm_{12}$ and $\pm_1 \neq \pm_2$, then
$$
  \max\left(\theta_{01},\theta_{02}\right)
  \sim
  \theta_{12}.
$$
\end{lemma}

We now have at our disposal all the tools required to finish the proof of the main dyadic estimate \eqref{MainDyadic}. Recall that the DKG case \eqref{D:97} has been completely dealt with, so the remaining null regimes are \eqref{D:98} and \eqref{D:99}.

\section{Proof of the dyadic quadrilinear estimate, Part I}\label{G}

By symmetry, we may assume
$$
  L_1 \le L_2,
  \qquad
  L_3 \le L_4,
$$
We distinguish the cases (i) $L_2 \le L_0'$, (ii) $L_4 \le L_0$ and (iii) $L_2 > L_0'$, $L_4 > L_0$, but in this section we further restrict (i) and (ii) to $L_0 \sim L_0'$, leaving the remaining cases for the next section. By symmetry it suffices to consider
\begin{subequations}\label{G:8}
\begin{alignat}{2}
  \label{G:8a}
  &\theta_{12} \ll \phi \lesssim \theta_{34},&
  \qquad &\bigl( \implies \abs{q_{1234}} \lesssim \phi\theta_{34} \bigr)
  \\
  \label{G:8b}
  &\theta_{12}, \theta_{34} \ll \phi,&
  \qquad &\bigl( \implies \abs{q_{1234}} \lesssim \phi^2 \bigr),
\end{alignat}
\end{subequations}
where the estimates on the right hold by Lemma \ref{NullLemma1}. By Lemma \ref{AnglesLemma},
\begin{equation}\label{D:134:10}
  \theta_{12} \lesssim \gamma \equiv \min \left( \gamma^*, \biggl(\frac{N_0\Lmax^{0'12}}{N_1N_2}\biggr)^{1/2} \right),
  \qquad \text{for some $0 < \gamma^* \ll 1$}.
\end{equation}
In fact, here we can choose any $0 < \gamma^* \ll 1$ that we want, by adjusting the implicit constant in \eqref{G:8a}. By Lemma \ref{AnglesLemma} we also have
\begin{equation}\label{E:4}
  \theta_{34} \lesssim \gamma' \equiv \min \left(1, \frac{\Lmax^{034}}{\Nmin^{34}} \right)^{1/2}.
\end{equation}
Observe that
\begin{equation}\label{E:2}
  \phi \le \min(\theta_{01},\theta_{02}) + \min(\theta_{03},\theta_{04}),
\end{equation}
since $\theta_{jk} \le \theta_{0j} + \theta_{0k}$. By Lemma \ref{F:Lemma4},
\begin{equation}\label{E:2:2}
  \min(\theta_{01},\theta_{02})
  \lesssim \biggl(\frac{\Lmax^{0'12}}{N_0} \biggr)^{1/2},
  \qquad
  \min(\theta_{03},\theta_{04})
  \lesssim \biggl(\frac{\Lmax^{034}}{N_0} \biggr)^{1/2}.
\end{equation}
We assume that $u_j \in L^2(\R^{1+2})$ for $j=1,2,3,4$ has nonnegative Fourier transform $\widetilde{u_j}$ supported in $K^{\pm_j}_{N_j,L_j}$. To simplify, we introduce the shorthand
\begin{equation}
  \label{E:3:1}
  u_{0'12}
  = \Proj_{K^{\pm_0}_{N_0,L_0'}}\left( u_1 \overline{u_2} \right),
  \qquad
  u_{043}
  = \Proj_{K^{\pm_0}_{N_0,L_0}}\left( u_4 \overline{u_3} \right).
\end{equation}
We define $\pm_{12}$ and $\pm_{43}$ as in \eqref{SignDef}, recalling that $\xi_0 = \xi_1-\xi_2 = \xi_4-\xi_3$. Note the following important relations:
\begin{alignat}{2}
  \label{E:3:2}
  &\pm_0=\pm_{12}, \; \theta_{12} \ll 1, \; N_0 \ll N_1 \sim N_2&
  \quad \implies \quad
  &\theta_{01} \sim \theta_{02} \sim \frac{N_1}{N_0} \theta_{12},
  \\
  \label{E:3:3}
  &\pm_0=\pm_{12}, \; \theta_{12} \ll 1, \; N_1 \lesssim N_0 \sim N_2&
  \quad \implies \quad
  &\theta_{12} \sim \theta_{01} \sim \frac{N_0}{N_1} \theta_{02}.
\end{alignat}
This follows from Lemmas \ref{F:Lemma1} and \ref{F:Lemma3}, and the fact, from the proof of Lemma \ref{F:Lemma1} in \cite{Selberg:2008b}, that $\theta_{02} \le \theta_{01}$ if $\abs{\xi_1} \le \abs{\xi_2}$. Note also that \eqref{E:3:2} can only happen if $\pm_1=\pm_2$, by Lemma \ref{AnglesLemma}. Of course, \eqref{E:3:3} applies symmetrically if $N_2 \lesssim N_0 \sim N_1$. Analogous estimates apply to the index 043.

\subsection{The case $L_2 \le L_0' \sim L_0$}

Then we treat the cases \eqref{G:8a} and \eqref{G:8b} simultaneously by using Lemma \ref{NullLemma2} to estimate $\abs{q_{1234}} \lesssim \theta_{13}\theta_{24}$, and pairing up $u_1$ with $u_3$, and $u_2$ with $u_4$, by changing variables from $(\tau_0',\tau_0,\xi_0)$ to
$$
  \tilde\tau_0' = \tau_1+\tau_3,
  \qquad
  \tilde\tau_0 = \tau_2+\tau_4,
  \qquad
  \tilde\xi_0 = \xi_1+\xi_3 = \xi_2+\xi_4.
$$
Then $\tilde\tau_0' - \tilde\tau_0 = \tau_0' - \tau_0$, so the symbol of $T_{L_0,L_0'}$ is invariant under the change of variables: $a_{L_0,L_0'}^{\pm_0}(\tau_0,\tau_0',\xi_0) =
  a_{L_0,L_0'}^{\pm_0}(\tilde\tau_0,\tilde\tau_0',\tilde\xi_0)$. This is where we use the assumption $L_0 \sim L_0'$. Using  Lemma \ref{Tlemma} we conclude that
\begin{align*}
  J_{\boldN,\boldL}
  &\lesssim
  \int
  T_{L_0,L_0'} \mathcal F\mathfrak B'_{\theta_{13}}(u_1,u_3)(\tilde X_0)
  \cdot \mathcal F \mathfrak B'_{\theta_{24}}(u_2,u_4)(\tilde X_0)
  \, d\tilde X_0
  \\
  &\lesssim
  \norm{\mathfrak B'_{\theta_{13}}(u_1,u_3)}
  \norm{\mathfrak B'_{\theta_{24}}(u_2,u_4)},
\end{align*}
where the null form $\mathfrak B'_{\theta_{13}}(u_1,u_3)$ is defined by inserting $\theta_{13}$ in the convolution formula $\widetilde{u_1u_3}(X_0) \simeq \int \widetilde{u_1}(X_1)\widetilde{u_3}(X_3) \delta(X_0 - X_1 - X_3) \, dX_1 \, dX_3$. The estimates for $\mathfrak B_{\theta_{12}}$ in the previous section hold also for this null form.

Recalling \eqref{D:134:10} and applying Lemma \ref{WhitneyLemma2} to the pair $(\pm_1\xi_1,\pm_2\xi_2)$ before making the above change of variables, we obtain similarly
$$
  J_{\boldN,\boldL}
  \lesssim
  \sum_{\omega_1,\omega_2}
  \norm{\mathfrak B'_{\theta_{13}}(u_1^{\gamma,\omega_1},u_3)}
  \norm{ \mathfrak B'_{\theta_{24}}(u_2^{\gamma,\omega_2},u_4)},
$$
where the sum is over $\omega_1,\omega_2 \in \Omega(\gamma)$ satisfying $\theta(\omega_1,\omega_2) \lesssim \gamma$ and we write
$$
  u_j^{\gamma,\omega_j} = \Proj_{\pm_j\xi_j \in \Gamma_{\gamma}(\omega_j)} u.
$$
Since the spatial frequency $\xi_j$ of $u_j^{\gamma,\omega_j}$ is restricted to a tube of radius comparable to $N_j\gamma$ about $\R\omega_j$, we can apply Theorem \ref{NullRay}, obtaining
\begin{align*}
  J_{\boldN,\boldL}^{\mathbf\Sigma}
  &\lesssim
  \left( N_1 N_2 \gamma^2 L_1 L_2 L_3 L_4 \right)^{1/2}
  \sum_{\omega_1,\omega_2}\norm{u_1^{\gamma,\omega_1}}\norm{u_2^{\gamma,\omega_2}}\norm{u_3}\norm{u_4}
  \\
  &\lesssim
  \left( N_0 L_0' L_1 L_2 L_3 L_4 \right)^{1/2}
  \norm{u_1}\norm{u_2}\norm{u_3}\norm{u_4},
\end{align*}
where we summed $\omega_1,\omega_2$ as in \eqref{OmegaSum}, and used the definition\eqref{D:134:10} of $\gamma$, taking into account the assumption $L_2 \le L_0'$. Interpolating with the crude estimate\begin{equation}\label{TrivialEstimate}
  J_{\boldN,\boldL}
  \lesssim
  \norm{u_{0'12}}
  \norm{u_{043}}
  \lesssim
  \bigl( N_0^2 \Lmin^{0'12} \bigr)^{1/2}
  \left(N_0^2 \Lmin^{034} \right)^{1/2}
  \norm{u_1} \norm{u_2} \norm{u_3} \norm{u_4},
\end{equation}
which follows from \eqref{SobolevType}, we get the desired estimate \eqref{MainDyadic}, recalling that $L_0 \sim L_0'$.

\subsection{The case $L_4 \le L_0 \sim L_0'$}

If $\theta_{34} \ll 1$, then we have the analog of \eqref{D:134:10}, so by symmetry the argument in the previous subsection applies, with the roles of the indices 12 and 34 reversed. We therefore assume $\theta_{34} \sim 1$. Then $\Nmin^{34} \lesssim L_0$, by Lemma \ref{AnglesLemma}. Moreover, we may assume $L_2 > L_0'$, since the case $L_2 \le L_0'$ is done. Now trivially estimate $\abs{q_{1234}} \lesssim 1$. Then with notation as in \eqref{E:3:1},
\begin{align*}
  J_{\boldN,\boldL}
  &\lesssim
  N_0^{3/4}(L_0'L_1)^{3/8} (\Nmin^{34})^{3/4}(L_3L_4)^{3/8}
  \norm{u_1} \norm{u_2} \norm{u_3} \norm{u_4}
  \\
  &\lesssim
  N_0^{3/4} \left( L_0 L_0' L_1 L_2 L_3 L_4 \right)^{3/8}
  \norm{u_1} \norm{u_2} \norm{u_3} \norm{u_4},
\end{align*}
where we used Lemma \ref{Tlemma}, Theorem \ref{BasicBilinearThm}, the assumption $L_2 > L_0'$ and the fact that $\Nmin^{34} \lesssim L_0 \sim L_0'$.

\subsection{The case $L_2 > L_0'$ and $L_4 > L_0$}

So far we could treat \eqref{G:8a} and \eqref{G:8b} simultaneously, but from now on we need to separate the two, and we divide into subcases depending on which term dominates in the right hand side of \eqref{E:2}:
\begin{subequations}\label{G:10}
\begin{alignat}{2}
  \label{G:10a}
  &\theta_{12} \ll \phi \lesssim \theta_{34},&
  \qquad
  &\min(\theta_{01},\theta_{02}) \ge \min(\theta_{03},\theta_{04}),
  \\
  \label{G:10b}
  &\theta_{12} \ll \phi \lesssim \theta_{34},&
  \qquad
  &\min(\theta_{01},\theta_{02}) < \min(\theta_{03},\theta_{04}),
  \\
  \label{G:10c}
  &\theta_{12}, \theta_{34} \ll \phi,&
  \qquad
  &\min(\theta_{01},\theta_{02}) < \min(\theta_{03},\theta_{04}),
  \\
  \label{G:10d}
  &\theta_{12}, \theta_{34} \ll \phi,&
  \qquad
  &\min(\theta_{01},\theta_{02}) \ge \min(\theta_{03},\theta_{04}).
\end{alignat}
\end{subequations}
Note that the last two are symmetric, so we only consider the first three. Subcase \eqref{G:10b} is by far the most difficult, and will be split further into  subcases.

\subsection{Subcase $\theta_{12} \ll \phi \lesssim \theta_{34}$, $\min(\theta_{01},\theta_{02}) \ge \min(\theta_{03},\theta_{04})$}\label{G:12:0}

By \eqref{E:4}--\eqref{E:2:2},
$$
  \abs{q_{1234}}
  \lesssim \phi\theta_{34}
  \lesssim (\phi\theta_{34})^{3/4}
  \lesssim
  \left( \frac{L_2}{N_0} \right)^{3/8}
  \left( \frac{L_4}{\Nmin^{34}} \right)^{3/8},
$$
hence
\begin{align*}
  J_{\boldN,\boldL}
  &\lesssim
  \left( \frac{L_2}{N_0} \frac{L_4}{\Nmin^{34}} \right)^{3/8} 
  \left( N_0^2 L_0'L_1 \right)^{3/8} \left( N_0\Nmin^{034} L_0L_3 \right)^{3/8}
  \prod_{j=1}^4 \norm{u_j}
  \\
  &=
  N_0^{3/4} (L_0L_0' L_1L_2L_3L_4)^{3/8} \norm{u_1} \norm{u_2} \norm{u_3} \norm{u_4},
\end{align*}
where we used Lemma \ref{Tlemma} and Theorem \ref{BasicBilinearThm}.

\subsection{Subcase $\theta_{12} \ll \phi \lesssim \theta_{34}$, $\min(\theta_{01},\theta_{02}) < \min(\theta_{03},\theta_{04})$}\label{G:14:0}

Then
\begin{equation}\label{G:14:2}
  \abs{q_{1234}} \lesssim \phi\theta_{34}
  \lesssim \min(\theta_{03},\theta_{04}) \left( \frac{L_4}{\Nmin^{34}} \right)^{p}
\end{equation}
for $0 \le p \le 1/2$. By \eqref{D:134:10} and Lemma \ref{WhitneyLemma2} applied to $(\pm_1\xi_1,\pm_2\xi_2)$,
\begin{equation}\label{G:14:4}
  J_{\boldN,\boldL}
  \lesssim
  \sum_{\omega_1,\omega_2} \left( \frac{L_4}{\Nmin^{34}} \right)^{1/4}
  \norm{\mathfrak B_{\theta_{03}}' \left( \Proj_{K^{\pm_0}_{N_0,L_0}} \mathcal F^{-1} T_{L_0,L_0'}^{\pm_0} \mathcal F u_{0'12}^{\gamma,\omega_1,\omega_2}, u_3 \right)} \norm{u_4},
\end{equation}
where
\begin{equation}\label{G:14:6}
  u^{\gamma,\omega_1,\omega_2}_{0'12}
  =
  \Proj_{K^{\pm_0}_{N_0,L_0'}} 
  \left( u_1^{\gamma,\omega_1}
  \overline{ u_2^{\gamma,\omega_2} } \right)
\end{equation}
and the sum is over $\omega_1,\omega_2 \in \Omega(\gamma)$ with $\theta(\omega_1,\omega_2) \lesssim \gamma$. The spatial Fourier support of $u^{\gamma,\omega_1,\omega_2}_{0'12}$ is contained in a tube of radius comparable to $\Nmax^{12}\gamma$ around $\R\omega_1$. Therefore, by Theorem \ref{NullRay}, Lemma \ref{Tlemma} and Theorem \ref{BasicBilinearThm},
\begin{equation}\label{G:14:8}
\begin{aligned}
  J_{\boldN,\boldL}
  &\lesssim
  \sum_{\omega_1,\omega_2} \left( \frac{L_4}{\Nmin^{34}} \right)^{1/4}
  \left(\Nmax^{12}\gamma L_0L_3\right)^{1/2}
  \norm{u_{0'12}^{\gamma,\omega_1,\omega_2}}
  \norm{u_3} \norm{u_4}
  \\
  &\lesssim
  \left( \frac{L_4^{1/2}}{(\Nmin^{34})^{1/2}}
  \Nmax^{12} \left( \frac{N_0L_2}{N_1N_2} \right)^{1/2} L_0L_3
  N_0 (\Nmin^{012})^{1/2} (L_0'L_1)^{3/4} \right)^{1/2} 
  \prod_{j=1}^4 \norm{u_j}
  \\
  &\lesssim
  \left( \frac{N_0}{\Nmin^{34}} \right)^{1/4}
  N_0^{3/4} L_0^{1/2}(L_0')^{3/8}(L_1L_2)^{5/16}(L_3L_4)^{3/8}
  \prod_{j=1}^4 \norm{u_j}.
\end{aligned}
\end{equation} 
Here we summed $\omega_1,\omega_2$ as in \eqref{OmegaSum} and used \eqref{D:134:10} (recalling $L_0' < L_2$), the fact that $\Nmin^{012} \Nmax^{012} \sim N_0\Nmin^{12}$, and the assumptions $L_1 \le L_2$, $L_3 \le L_4$.

Interpolating with the trivial estimate \eqref{TrivialEstimate} we then obtain \eqref{MainDyadic} if $N_0 \lesssim \Nmin^{34}$, but also whenever we are able to gain an extra factor $(\Nmin^{34}/N_0)^{1/4}$. In particular, this happens if $\pm_0 \neq \pm_{43}$, since then $N_0 \lesssim L_4$ by Lemma \ref{F:Lemma2}, so instead of \eqref{E:4} we can use $\theta_{34} \lesssim 1 \lesssim (L_4/N_0)^{1/4}$ in \eqref{G:14:2}, thereby gaining  the desired factor. Thus, we may assume $\pm_0 = \pm_{43}$, and the same argument shows that we may assume $\theta_{34} \ll 1$. Moreover, we can assume $\pm_0=\pm_{12}$, since otherwise Lemma \ref{F:Lemma2} implies $N_0 \lesssim L_2$, hence the argument in section \ref{G:12:0} applies. Next observe that by \eqref{E:3:3} and \eqref{G:14:2}, since $\pm_0=\pm_{43}$ and $\theta_{34} \ll 1$,
\begin{equation}\label{G:14:10}
  N_3 \ll N_0 \sim N_4 \implies
  \theta_{04} \lesssim \frac{N_3}{N_0} \theta_{34}, \quad \theta_{03} \sim \theta_{34},
  \quad
  \abs{q_{1234}} \lesssim \frac{N_3}{N_0} \theta_{03} \theta_{34},
\end{equation}
so we gain a factor $N_3/N_0$ in \eqref{G:14:8}, which is more than enough. We are therefore left with $N_4 \ll N_0 \sim N_3$, which is hard; we split further into $N_0 \lesssim N_2$ and $N_2 \ll N_0$, treated in the next two subsections. Here one should keep in mind that $\pm_0=\pm_{12}=\pm_{43}$, $L_1 \le L_2$, $L_3 \le L_4$, $L_2 > L_0'$ and $L_4 > L_0$.

\subsubsection{Subcase $N_0 \lesssim N_2$}\label{G:20:0}

Inserting $\Proj_{\abs{\xi_4} \lesssim N_4}$ in front of $\mathfrak B_{\theta_{03}}'$ in \eqref{G:14:8}, then instead of Theorem \ref{NullRay} we apply Theorem \ref{ImprovedNullRay}, the hypotheses of which are satisfied: First, since $N_4 \ll N_0 \sim N_3$, we have $N_0,N_3 \gg 1$ and $\theta_{03} \ll 1$ [by the analog of \eqref{G:14:10}]. Second, the hypothesis $r \ll \Nmin^{12}$  in the theorem now becomes
\begin{equation}\label{G:20:2}
  \Nmax^{12} \gamma \ll N_0,
\end{equation}
with $\gamma$ as in \eqref{D:134:10}. But if \eqref{G:20:2} fails, then $N_0 \ll N_1 \sim N_2$, and $N_0 \lesssim L_2$ in view of the definition \eqref{D:134:10} of $\gamma$, so the argument in section \ref{G:12:0} applies. Thus, we can assume that \eqref{G:20:2} holds, hence Theorem \ref{ImprovedNullRay} applies, so in \eqref{G:14:8} we can replace $\norm{u_{0'12}^{\gamma,\omega_1,\omega_2}}$ by
$$
  \sup_{I} \norm{\Proj_{\xi_0 \cdot \omega_1 \in I} u_{0'12}^{\gamma,\omega_1,\omega_2}},
$$
where the supremum is over all intervals $I \subset \R$ with $\abs{I} = N_4$. But since $\gamma \ll 1$, Theorem \ref{Anisotropic} implies, via duality,
\begin{equation}\label{G:20:6}
  \sup_I \norm{\Proj_{\xi_0 \cdot \omega_1 \in I}
  u_{0'12}^{\gamma,\omega_1,\omega_2}}
  \lesssim
  \left[N_4 (\Nmin^{01})^{1/2} (L_0'L_1)^{3/4} \right]^{1/2}
  \norm{u_1^{\gamma,\omega_1}} \norm{u_2^{\gamma,\omega_2}},
\end{equation}
so in the second line of \eqref{G:14:8}, $N_0 (\Nmin^{012})^{1/2} (L_0'L_1)^{3/4}$ inside the square root is replaced by $N_4 (\Nmin^{01})^{1/2} (L_0'L_1)^{3/4}$, so in effect we gain a factor $(N_4/N_0)^{1/2}$, recalling that $N_0 \lesssim N_2$.

\subsubsection{Subcase $N_0 \gg N_2$}

If $N_2 \sim 1$, we simply estimate
\begin{equation}\label{G:20:8:2}
  J_{\boldN,\boldL}
  \lesssim
  \norm{u_{0'12}} \bignorm{u_{043}}
  \lesssim
  \left( N_2^{3/2} (L_1 L_2)^{3/4}
  \cdot N_0^{3/2} (L_0 L_3)^{3/4} \right)^{1/2}
  \prod_{j=1}^4 \norm{u_j},
\end{equation}
by Lemma \ref{Tlemma} and Theorem \ref{BasicBilinearThm}. We therefore assume $1 \ll N_2 \ll N_0 \sim N_1$. This ensures that $\angles{\xi_2} \sim N_2$ can be replaced by $\abs{\xi_2} \sim N_2$. By \eqref{E:3:3} and \eqref{D:134:10},
\begin{equation}\label{G:20:8}
  \theta_{12} \sim \theta_{02} \sim \frac{N_0}{N_2} \theta_{01},
  \qquad \text{hence}
  \qquad
  \theta_{01} \lesssim \alpha \equiv \frac{N_2}{N_0} \gamma.
\end{equation}
Now modify \eqref{G:14:4} by applying Lemma \ref{WhitneyLemma2} again, this time to $(\pm_0\xi_0,\pm_1\xi_1)$:
\begin{multline}\label{G:20:10}
  J_{\boldN,\boldL}
  \lesssim
  \sum_{\omega_1,\omega_2} \sum_{\omega_0',\omega_1'} \left( \frac{L_4}{N_4} \right)^{1/2}
  \\
  \times
  \norm{\Proj_{\abs{\xi_4} \lesssim N_4}
  \mathfrak B_{\theta_{03} \ll 1}' \left( \Proj_{K^{\pm_0}_{N_0,L_0}} \mathcal F^{-1} T_{L_0,L_0'}^{\pm_0} \mathcal F u_{0'12}^{\gamma,\omega_1,\omega_2;\alpha,\omega_0',\omega_1'}, u_3 \right)} \norm{u_4},
\end{multline}
where the second sum is over $\omega_0',\omega_1' \in \Omega(\alpha)$ satisfying $\theta(\omega_0',\omega_1') \lesssim \alpha$, and
\begin{align}
  \label{G:20:12}
  u_{0'12}^{\gamma,\omega_1,\omega_2;\alpha,\omega_0',\omega_1'}
  &= \Proj_{\pm_0\xi_0 \in \Gamma_{\alpha}(\omega_0')} 
  \Proj_{K^{\pm_0}_{N_0,L_0'}} \left( u_1^{\gamma,\omega_1;\alpha,\omega_1'}
  u_2^{\gamma,\omega_2} \right)
  \\
  \label{G:20:14}
  u_1^{\gamma,\omega_1;\alpha,\omega_1'}
  &=
  \Proj_{\pm_1\xi_1 \in \Gamma_{\alpha}(\omega_1')} u_1^{\gamma,\omega_1}.
\end{align}
The spatial Fourier support of \eqref{G:20:12} is contained in a tube of radius comparable to $N_0\alpha \sim N_2\gamma$ around $\R\omega_0'$, whereas the one for \eqref{G:14:6} is of radius comparable to $N_1\gamma$, so we gain a factor $(N_2/N_0)^{1/2}$ when applying Theorem \ref{ImprovedNullRay}, compared to our estimates in the previous subsection. On the other hand, we now have the additional sum over $\omega_0',\omega_1'$. To come out on top, we have to make sure that this sum does not cost us more than a factor $(N_0/N_2)^{1/4}$. For the bilinear interaction $X_0'=X_1-X_2$ in \eqref{G:20:12} we have, by \eqref{B:112}, recalling also $\theta(\omega_0',\omega_1') \lesssim \alpha$ and $N_1 \sim N_0$,
$$
  X_0' \in H_{\max(L_0',N_0\alpha^2)}(\omega_1'),
  \qquad
  X_1 \in H_{\max(L_1,N_0\alpha^2)}(\omega_1').
$$
Therefore,
\begin{equation}\label{G:20:16}
  X_2 = X_1-X_0' \in H_d(\omega_1'),
  \qquad \text{where}
  \qquad
  d = \max\bigl(\Lmax^{0'1},N_0\alpha^2\bigr),
\end{equation}
so we can insert $\Proj_{H_d(\omega_1')}$ in front of $u_2^{\gamma,\omega_2}$ in \eqref{G:20:12}.
Adapting the argument from the previous subsection we then get
\begin{equation}\label{G:20:18}
\begin{aligned}
  J_{\boldN,\boldL}
  &\lesssim
  \left( \frac{L_4}{N_4} \right)^{1/4}
  \left(N_2\gamma L_0L_3 N_4 N_0^{1/2} (L_0'L_1)^{3/4} \right)^{1/2}
  \\
  &\qquad
  \times
  \sum_{\omega_1,\omega_2} \sum_{\omega_0',\omega_1'}
  \bignorm{u_1^{\gamma,\omega_1;\alpha,\omega_1'}}
  \norm{\Proj_{H_d(\omega_1')} u_2^{\gamma,\omega_2}}
  \bignorm{u_3} \bignorm{u_4}
  \\
  &\lesssim
  \left( \frac{L_4^{1/2}}{N_4^{1/2}}
  N_2 \left(\frac{L_2}{N_2}\right)^{1/2} L_0L_3
  N_4 N_0^{1/2} (L_0'L_1)^{3/4}  \right)^{1/2}
  \\
  &\qquad
  \times
  B^{1/2}
  \sum_{\omega_1,\omega_2}
  \bignorm{u_1^{\gamma,\omega_1}}
  \norm{u_2^{\gamma,\omega_2}}
  \bignorm{u_3} \bignorm{u_4},
\end{aligned}
\end{equation}
where
\begin{equation}\label{G:20:20}
  B
  =
  \sup_{(\tau,\xi), \; \abs{\xi} \sim N_2}
  \sum_{\omega_1' \in \Omega(\alpha)} \chi_{H_d(\omega_1')}(\tau,\xi).
\end{equation}
If we can prove that
\begin{equation}\label{G:20:22}
  B
  \lesssim
  \left(\frac{N_0}{N_2}\right)^{1/2},
\end{equation}
then summing $\omega_1,\omega_2$ as in \eqref{OmegaSum} we get the desired estimate.

By Lemma \ref{HyperLemma},
$
  B
  \lesssim
  1 + (d/N_2\alpha^2)^{1/2},
$
where $d = \max\bigl(\Lmax^{0'1},N_0\alpha^2\bigr)$, so if $d = N_0\alpha^2$, we get \eqref{G:20:22}. The other possibility is $d=\Lmax^{0'1}$, which happens when $N_0\alpha^2 \le \Lmax^{0'1}$. Then instead of \eqref{G:20:22} we only get
\begin{equation}\label{G:20:24}
  B
  \lesssim
  \left( \frac{\Lmax^{0'1}}{N_2\alpha^2} \right)^{1/2},
\end{equation}
but to compensate we can use the following replacement for \eqref{G:20:6}:
\begin{equation}\label{G:20:26}
  \norm{\Proj_{\xi_0 \cdot \omega_1 \in I}
  u_{0'12}^{\gamma,\omega_1,\omega_2}}
  \lesssim
  \bigl(N_4 (N_2\gamma) \Lmin^{0'1} \bigr)^{1/2}
  \norm{u_1^{\gamma,\omega_1}} \norm{u_2^{\gamma,\omega_2}},
\end{equation}
which by \cite[Lemma 1.2]{Selberg:2010a} reduces to the trivial fact that the intersection of the strips $\{ \xi_0 \colon \xi_0 \cdot \omega_1 \in I \}$ and $T_r(\omega_2)$ has area $O(r\abs{I})$, where in the present case $r \sim N_2\gamma$ and $\abs{I} = N_4$. Modifying \eqref{G:20:18} accordingly, we again get the desired estimate.

\subsection{Subcase $\theta_{12}, \theta_{34} \ll \phi$, $\min(\theta_{01},\theta_{02}) < \min(\theta_{03},\theta_{04})$}

Then
\begin{equation}\label{G:22:2}
  \abs{q_{1234}} \lesssim \phi^2
  \lesssim \min(\theta_{03},\theta_{04}) \left(\frac{L_4}{N_0} \right)^{p}
  \qquad (0 \le p \le 1/2).
\end{equation}
Comparing with \eqref{G:14:2}, we then we get \eqref{G:14:8} with an extra factor $(\Nmin^{34}/N_0)^{1/4}$, implying the desired estimate.

\section{Proof of the dyadic quadrilinear estimate, Part II}\label{E}

It remains to consider the cases where
$$
  L_0 \ll L_0' \qquad \text{or} \qquad L_0 \gg L_0'
$$
and either $L_2 \le L_0'$ or $L_4 \le L_0$ (as before we assume $L_1 \le L_2$ and $L_3 \le L_4$ by symmetry). It suffices to consider the cases \eqref{G:10a}--\eqref{G:10c}, the last two of which we split further into
\begin{subequations}\label{E:0}
\begin{alignat}{2}  
  \label{E:0b}
  &L_2 \le L_0',&
  \qquad 
  &L_4 > L_0,
  \\
  \label{E:0c}
  &L_2 \le L_0',&
  \qquad
  &L_4 \le L_0,
  \\
  \label{E:0d}
  &L_2 > L_0',&
  \qquad
  &L_4 \le L_0.
\end{alignat}
\end{subequations}
We may assume $\Nmin^{12}, \Nmin^{34} \gg 1$, as otherwise trivial estimates analogous to \eqref{G:20:8:2} apply.

\subsection{Subcase $\theta_{12} \ll \phi \lesssim \theta_{34}$, $\min(\theta_{01},\theta_{02}) \ge \min(\theta_{03},\theta_{04})$}

By \eqref{E:4}--\eqref{E:2:2},
\begin{equation}\label{E:6}
  \abs{q_{1234}}
  \lesssim \phi\theta_{34}
  \lesssim (\phi\theta_{34})^{3/4}
  \lesssim
  \biggl( \frac{\Lmax^{0'12}}{N_0} \biggr)^{3/8}
  \left( \frac{\Lmax^{034}}{\Nmin^{34}} \right)^{3/8},
\end{equation}
so with notation as in \eqref{E:3:1},
$$
  J_{\boldN,\boldL}
  \lesssim
  \biggl( \frac{\Lmax^{0'12}}{N_0} \biggr)^{3/8}
  \left( \frac{\Lmax^{034}}{\Nmin^{34}} \right)^{3/8}
  \bignorm{T_{L_0,L_0'}^{\pm_0}\mathcal F u_{0'12}} \bignorm{u_{043}}.
$$
If we apply Lemma \ref{Tlemma} and \eqref{Bilinear3}, we get the desired estimate except in the case $N_0 \ll N_1 \sim N_2$, but then we can apply the following:

\begin{lemma}\label{E:Lemma1} If $L_0' \ll L_0$ or $L_0' \gg L_0$,
$$
  \bignorm{T_{L_0,L_0'}^{\pm_0}\mathcal F u_{0'12}}
  \lesssim
  \left(N_0^{1/2} \Nmin^{012} \Lmin^{0'12} (\Lmed^{0'12})^{1/2} \right)^{1/2} 
  \norm{u_1} \norm{u_2}.
$$
\end{lemma}

\begin{proof}[Proof of Lemma \ref{E:Lemma1}]
If $\Lmax^{0'12} = \Lmax^{12}$, this holds by Lemma \ref{Tlemma} and \eqref{Bilinear2}, so we assume $\Lmax^{0'12} = L_0'$. Since $\theta_{12} \ll 1$, we have $\theta_{12} \lesssim \gamma$ with $\gamma$ as in \eqref{D:134:10}, and we reduce to estimating $S = \sum_{\omega_1,\omega_2}
  \bignorm{T_{L_0,L_0'}^{\pm_0} \mathcal F u^{\gamma,\omega_1,\omega_2}_{0'12}},
$
where $\omega_1,\omega_2 \in \Omega(\gamma)$ with $\theta(\omega_1,\omega_2) \lesssim \gamma$. By \eqref{B:112},
\begin{equation}\label{E:12:4}
  \supp \mathcal F u^{\gamma,\omega_1,\omega_2}_{0'12}
  \subset H_{d'}(\omega_1),
  \qquad \text{where}
  \qquad
  d' = \max\left(\Lmax^{12},\Nmax^{12}\gamma^2\right),
\end{equation}
so by Lemma \ref{Tlemma2},
$$
  S
  \lesssim
  \sum_{\omega_1,\omega_2}
  \left(\frac{d'}{L_0'}\right)^{p}
  \norm{u^{\gamma,\omega_1,\omega_2}_{0'12}} \qquad (0 \le p \le 1/2).
$$
Taking $p=1/4$, we note that if $d' = \Lmax^{12}$, we get the desired estimate by using \eqref{Bilinear2} and summing $\omega_1,\omega_2$ as in \eqref{OmegaSum}. If $d' = \Nmax^{12}\gamma^2 \sim  N_0L_0'/\Nmin^{12}$, on the other hand, then \eqref{Bilinear1} implies the estimate we need.
\end{proof}

\subsection{Subcase $\theta_{12} \ll \phi \lesssim \theta_{34}$, $\min(\theta_{01},\theta_{02}) < \min(\theta_{03},\theta_{04})$, $L_2 \le L_0'$, $L_4 > L_0$}\label{E:0b:0}

Observe that \eqref{G:14:2} holds. Now repeat the argument leading to \eqref{G:14:8}, but use Lemma \ref{E:Lemma1} instead of Lemma \ref{Tlemma} and Theorem \ref{BasicBilinearThm}, hence\begin{equation}\label{G:14:8:2}
\begin{aligned}
  J_{\boldN,\boldL}
  &\lesssim
  \sum_{\omega_1,\omega_2} \left( \frac{L_4}{\Nmin^{34}} \right)^{1/4}
  \left(\Nmax^{12}\gamma L_0L_3\right)^{1/2}
  \bignorm{T_{L_0,L_0'}^{\pm_0} \mathcal F u_{0'12}^{\gamma,\omega_1,\omega_2}}
  \norm{u_3} \norm{u_4}
  \\
  &\lesssim
  \left( \frac{L_4^{1/2}}{(\Nmin^{34})^{1/2}}\Nmax^{12} \left( \frac{N_0L_0'}{N_1N_2} \right)^{1/2} L_0L_3
  N_0 (\Nmin^{012})^{1/2} (L_1L_2)^{3/4} \right)^{1/2} 
  \prod_{j=1}^4 \norm{u_j}
  \\
  &\lesssim
  \left( \frac{N_0}{\Nmin^{34}} \right)^{1/4}
  N_0^{3/4} L_0^{1/2}(L_0')^{1/4}(L_1L_2L_3L_4)^{3/8}
  \prod_{j=1}^4 \norm{u_j},
\end{aligned}
\end{equation}
so interpolating with the trivial estimate \eqref{TrivialEstimate} we obtain \eqref{MainDyadic} if $N_0 \lesssim \Nmin^{34}$, but also whenever we are able to gain an extra factor $(\Nmin^{34}/N_0)^{1/4}$. Now we continue as in section \ref{G:14:0}, reducing finally to the difficult case $N_4 \ll N_0 \sim N_3$. Then we proceed as in section \ref{G:20:0}. We may assume \eqref{G:20:2} [otherwise $N_0 \lesssim L_0'$, and then \eqref{E:6} holds], hence Theorem \ref{ImprovedNullRay} applies, so in \eqref{G:14:8:2} we can replace $\bignorm{T_{L_0,L_0'}^{\pm_0} \mathcal F u_{0'12}^{\gamma,\omega_1,\omega_2}}$ by
\begin{equation}\label{G:14:8:4}
  \sup_{I} \norm{T_{L_0,L_0'}^{\pm_0} \mathcal F
  \Proj_{\xi_0 \cdot \omega_1 \in I}  
  u_{0'12}^{\gamma,\omega_1,\omega_2}},
\end{equation}
where the supremum is over $I \subset \R$ with $\abs{I} = N_4$. By Theorem \ref{Anisotropic},
\begin{equation}\label{E:42}
  \sup_{I} \norm{\Proj_{\xi_0 \cdot \omega_1 \in I} u_{0'12}^{\gamma,\omega_1,\omega_2}}
  \lesssim
  \left(N_4 (\Nmin^{12})^{1/2} (L_1L_2)^{3/4} \right)^{1/2}
  \norm{u_1^{\gamma,\omega_1}} \norm{u_2^{\gamma,\omega_2}}.
\end{equation}
If we combine this with Lemma \ref{Tlemma}, we get
\begin{equation}\label{E:42:0}
  \text{l.h.s.\eqref{G:14:8:4}}
  \lesssim
  \left(N_4 (\Nmin^{12})^{1/2} (L_1L_2)^{3/4} \right)^{1/2}
  \norm{u_1^{\gamma,\omega_1}} \norm{u_2^{\gamma,\omega_2}},
\end{equation}
but we need
\begin{equation}\label{E:42:1}
  \text{l.h.s.\eqref{G:14:8:4}}
  \lesssim
  \left(N_4 (\Nmin^{012})^{1/2} (L_1L_2)^{3/4} \right)^{1/2}
  \norm{u_1^{\gamma,\omega_1}} \norm{u_2^{\gamma,\omega_2}}.
\end{equation}
If this holds, then we gain the necessary factor $(N_4/N_0)^{1/4}$ in \eqref{G:14:8:2}.

We prove \eqref{E:42:1} for $N_0 \ll N_1 \sim N_2$, as otherwise it reduces to \eqref{E:42:0}. Recalling \eqref{E:12:4} from the proof of Lemma \ref{E:Lemma1}, we use Lemma \ref{Tlemma2} followed by either \eqref{E:42} or
\begin{equation}\label{E:42:2}
  \sup_I \norm{\Proj_{\xi_0 \cdot \omega_1 \in I}
  u_{0'12}^{\gamma,\omega_1,\omega_2}}
  \lesssim
  \left(N_4 (\Nmin^{01})^{1/2} (L_0'L_1)^{3/4} \right)^{1/2}
  \norm{u_1^{\gamma,\omega_1}} \norm{u_2^{\gamma,\omega_2}},
\end{equation}
which follows from Theorem \ref{Anisotropic} via duality, recalling $\gamma \ll 1$. Specifically, if $d'=\Lmax^{12}$, we use \eqref{E:42:2}, whereas \eqref{E:42} is used if $d' =\Nmax^{12}\gamma^2$. Then \eqref{E:42:1} follows.

\subsection{Subcase $\theta_{12} \ll \phi \lesssim \theta_{34}$, $\min(\theta_{01},\theta_{02}) < \min(\theta_{03},\theta_{04})$, $L_2 \le L_0'$, $L_4 \le L_0$}\label{G:40:0:0}

For the remainder of section \ref{E}, we change the notation from \eqref{D:134:10}, writing now
\begin{equation}\label{D:134:10:2}
  \theta_{12} \lesssim \gamma \equiv \biggl(\frac{N_0\Lmax^{0'12}}{N_1N_2}\biggr)^{1/2}.
\end{equation}
By \eqref{E:2} and \eqref{E:2:2},
\begin{equation}\label{G:40:0:2}
  \phi
  \lesssim \min(\theta_{03},\theta_{04})
  \lesssim
  \left( \frac{L_0}{N_0} \right)^{p} \qquad (0 \le p \le 1/2),
\end{equation}
hence $\abs{q_{1234}} \lesssim \phi\theta_{34} \lesssim (L_0/N_0)^{p} \theta_{34}$, so applying Lemma \ref{WhitneyLemma2} to $(\pm_1\xi_1,\pm_2\xi_2)$ and Lemma \ref{WhitneyLemma1} to $(\pm_3\xi_3,\pm_4\xi_4)$, and recalling \eqref{E:4},
\begin{multline}\label{G:40:2}
  J_{\boldN,\boldL}
  \lesssim
  \sum_{\omega_1,\omega_2}
  \sum_{0 < \gamma_{34} \lesssim \gamma'}
  \sum_{\omega_3,\omega_4}
  \left( \frac{L_0}{N_0} \right)^{p}
  \gamma_{34}
  \\
  \times
  \int T_{L_0,L_0'}^{\pm_0} \mathcal F u_{0'12}^{\gamma,\omega_1,\omega_2}(X_0)
  \cdot
  \mathcal F u_{043}^{\gamma_{34},\omega_4,\omega_3}(X_0) \, dX_0,
\end{multline}
where $\gamma'$ is defined as in \eqref{E:4}, $u_{0'12}^{\gamma,\omega_1,\omega_2}$ is defined as in \eqref{G:14:6}, $u^{\gamma_{34},\omega_4,\omega_3}_{043}$ is similarly defined, and the sum is over $\omega_1,\omega_2 \in \Omega(\gamma)$ with $\theta(\omega_1,\omega_2) \lesssim \gamma$, dyadic $\gamma_{34}$ and $\omega_3,\omega_4 \in \Omega(\gamma_{34})$ satisfying $3\gamma_{34} \le \theta(\omega_3,\omega_4) \le 12\gamma_{34}$, hence $\theta_{34} \sim \gamma_{34}$ in $u^{\gamma_{34},\omega_4,\omega_3}_{043}$.

Recall that the spatial Fourier support of $u^{\gamma,\omega_1,\omega_2}_{0'12}$ is contained in a tube of radius $r \sim \Nmax^{12} \gamma$ around $\R\omega_1$. Covering $\R$ by almost disjoint intervals $I$ of length $r$,
$$
  u_{0'12}^{\gamma,\omega_1,\omega_2}
  =
  \sum_I
  \Proj_{\xi_0 \cdot \omega_1 \in I} u_{0'12}^{\gamma,\omega_1,\omega_2},
$$
where the sum has cardinality $O(N_0/r)$. Fix $I$. Then $\xi_0$ is restricted to a cube of sidelength $r$, and tiling by translates of this cube we may assume without loss of generality that the $\xi_j$ are restricted to such cubes $Q_j$, for $j=1,2,3,4$.

By \eqref{E:12:4},
\begin{equation}\label{G:40:14}
  \tau_0' + \xi_0 \cdot \omega_1 = O(d'),
  \qquad
  \text{where}
  \qquad
  d'=\max\left(L_2,\Nmax^{12}\gamma^2\right).
\end{equation}
Moreover, as proved in \cite[Section 9.4]{Selberg:2008b},
\begin{equation}\label{G:40:16}
  \tau_0 + \xi_0 \cdot \omega_3 = c + O(d),
  \qquad
  \text{where}
  \qquad
  d=\max\left(L_4,\frac{r^2}{\Nmin^{34}},r\gamma_{34}\right),
\end{equation}
and $c \in \R$ depends on $(Q_3,Q_4)$ and $(\omega_3,\omega_4)$. So by Lemmas \ref{Tlemma} and \ref{Tlemma2}, we can dominate the integral in \eqref{G:40:2} by the product of
\begin{equation}\label{G:40:30:2}
  \left(\min\left(1,\frac{d'}{L_0'}\right) \right)^{3/8}
  \norm{u_{0'12}^{\gamma,\omega_1,\omega_2}}
\end{equation}
and
\begin{equation}\label{G:40:30:4}
  \left(\frac{d}{L_0}\right)^{1/4}
  \norm{u_{043}^{\gamma_{34},\omega_4,\omega_3}}.
\end{equation}
By Theorem \ref{Anisotropic},
\begin{equation}\label{G:40:32}
  \norm{u_{0'12}^{\gamma,\omega_1,\omega_2}}
  \lesssim C\,
  \bignorm{u_1^{\gamma,\omega_1}}
  \bignorm{u_2^{\gamma,\omega_2}}
\end{equation}
holds with
\begin{align}
  \label{G:40:32:4}
  C^2 &\sim r(\Nmin^{01})^{1/2} (L_0'L_1)^{3/4},
  \\
  \label{G:40:32:2}
  C^2 &\sim r(\Nmin^{12})^{1/2}(L_1L_2)^{3/4}.
\end{align}
Noting that
\begin{equation}\label{G:40:32:6}
  \frac{d'}{L_0'} 
  \sim
  \max\left(\frac{L_2}{L_0'},\frac{N_0}{\Nmin^{12}}\right),
\end{equation}
we use \eqref{G:40:32:2} if $N_0 \sim \Nmax^{12}$, and otherwise the minimum of \eqref{G:40:32:4} and \eqref{G:40:32:2}, hence
\begin{equation}\label{G:40:32:8}
  \eqref{G:40:30:2} \lesssim \left(r(\Nmin^{012})^{1/2}(L_1L_2)^{3/4}\right)^{1/2}
  \bignorm{u_1^{\gamma,\omega_1}}
  \bignorm{u_2^{\gamma,\omega_2}}.
\end{equation}
Next we claim that
\begin{equation}\label{G:40:34}
  \norm{u_{043}^{\gamma_{34},\omega_4,\omega_3}}
  \lesssim C\,
  \bignorm{u_3^{\gamma_{34},\omega_3}}
  \bignorm{u_4^{\gamma_{34},\omega_4}}
\end{equation}
holds with
\begin{align}
  \label{G:40:36}
  C^2
  &\sim
  r^2 L_3,
  \\
  \label{G:40:38}
  C^2
  &\sim
  r (\Nmin^{34})^{1/2} (L_3 L_4)^{3/4},
  \\
  \label{G:40:40}
  C^2
  &\sim
  \frac{r L_3 L_4}{\gamma_{34}}.
\end{align}
In fact, \eqref{G:40:40} holds due to the assumption $\theta(\omega_3,\omega_4) \ge 3\gamma_{34}$, by the argument in \cite[Section 3.3]{Selberg:2010a}; \eqref{G:40:38} holds by Theorem \ref{Anisotropic}, and \eqref{G:40:36} reduces to a trivial volume estimate (see \cite[Lemma 1.1]{Selberg:2010a}). Interpolating \eqref{G:40:36} and \eqref{G:40:40} we also get
\begin{equation}\label{G:40:41}
  C^2
  \sim
  \frac{r^2 L_3 L_4^{1/2}}{(r\gamma_{34})^{1/2}}
  \le
  \frac{r^2 (L_3 L_4)^{3/4}}{(r\gamma_{34})^{1/2}},
\end{equation}
and since $d^{1/2}$ times the minimum of \eqref{G:40:36}, \eqref{G:40:38} and \eqref{G:40:41} is $\lesssim r^2 (L_3 L_4)^{3/4}$,
\begin{equation}\label{G:40:42}
  \eqref{G:40:30:4}
  \lesssim
  \left(\frac{r^2(L_3 L_4)^{3/4}}{L_0^{1/2}}\right)^{1/2}
  \bignorm{u_3^{\gamma_{34},\omega_3}}
  \bignorm{u_4^{\gamma_{34},\omega_4}}.
\end{equation}
Estimating the integral in \eqref{G:40:2} by the product of \eqref{G:40:32:8} and \eqref{G:40:42}, summing the $\omega$'s as in \eqref{OmegaSum}, estimating $\gamma_{34} \lesssim \gamma_{34}^{1/2}$ and using $\sum\gamma_{34}^{1/2} \sim (\gamma')^{1/2}$, where the sum is over dyadic $0 < \gamma_{34} \lesssim \gamma'$, we conclude, taking $p=3/8$,
\begin{equation}\label{G:40:42:2}
\begin{aligned}
  J_{\boldN,\boldL}
  &\lesssim \left( \sum_I \frac{r}{N_0} \right)
  (r\gamma')^{1/2}
  \left( \frac{\Nmin^{012}}{L_0} \right)^{1/4}
  N_0^{5/8} (L_0L_1L_2L_3L_4)^{3/8}
  \prod_{j=1}^4 \norm{u_j}
  \\
  &\lesssim
  \frac{(\Nmax^{12}\gamma\gamma')^{1/2}}{(L_0L_0')^{1/4}}
  \left( \frac{\Nmin^{012}}{N_0} \right)^{1/4}
  N_0^{7/8} \left(L_0L_0'L_1L_2L_3L_4\right)^{3/8}
  \prod_{j=1}^4 \norm{u_j},
\end{aligned}
\end{equation}
where we summed $I$ using the fact that the index set has cardinality $O(N_0/r)$, and used the definition $r \sim \Nmax^{12}\gamma$. Thus, if the expression
\begin{equation}\label{G:40:52}
  A = \frac{(\Nmax^{12}\gamma\gamma')^2\Nmin^{012}}{N_0L_0L_0'}
\end{equation}
is $O(1$), we get the desired estimate. In view of \eqref{D:134:10:2}, \eqref{E:4} and the assumptions $L_2 \le L_0'$ and $L_4 \le L_0$,
\begin{equation}\label{G:40:54}
  A
  \lesssim
  \frac{(\Nmax^{12})^2\Nmin^{012}}{N_0L_0L_0'}
  \frac{N_0L_0'}{N_1N_2}
  \min\left(1,\frac{L_0}{\Nmin^{34}}\right).
\end{equation}
In particular,
\begin{equation}\label{G:40:56}
  A 
  \lesssim
  \frac{\Nmax^{12}\Nmin^{012}}{\Nmin^{12}\Nmin^{34}}
  \lesssim \frac{N_0}{\Nmin^{34}},
\end{equation}
where we used the fact that $\Nmin^{012} \Nmax^{012} \sim N_0\Nmin^{12}$. The only remaining case is then $\Nmin^{34} \ll N_0$. If $\pm_0\neq\pm_{43}$, then $N_0 \lesssim L_0$ by Lemma \ref{F:Lemma2}, so we can estimate the minimum in \eqref{G:40:54} by $1 \lesssim L_0/N_0$, gaining a factor $\Nmin^{34}/N_0$ compared to \eqref{G:40:56}. If, on the other hand, $\pm_0=\pm_{43}$, then by Lemma \ref{F:Lemma1} and \eqref{E:4},
$$
  \min(\theta_{03},\theta_{04}) \lesssim \frac{\Nmin^{34}}{N_0} \theta_{34}
  \lesssim
  \frac{\Nmin^{34}}{N_0} \left( \frac{L_0}{\Nmin^{34}} \right)^{1/2}
  = \left(\frac{\Nmin^{34}}{N_0} \right)^{1/2} \left( \frac{L_0}{N_0} \right)^{1/2},
$$
which means that compared to \eqref{G:40:0:2} we gain a factor $(\Nmin^{34}/N_0)^{3/8}$ (since we took $p=3/8$ above), which then appears to the fourth power in $A$, so we have more than enough improvement.

\subsection{Subcase $\theta_{12} \ll \phi \lesssim \theta_{34}$, $\min(\theta_{01},\theta_{02}) < \min(\theta_{03},\theta_{04})$, $L_2 > L_0'$, $L_4 \le L_0$}\label{G:40:51}

The only difference from the previous case is that now $L_2 > L_0'$, instead of $L_2 \le L_0'$. This difference only shows up in the expression \eqref{D:134:10:2} for $\gamma$, however, and this expression is not used explicitly until the estimate \eqref{G:40:32:6}. But in the present case, $d'/L_0' > 1$, so the minimum in \eqref{G:40:30:2} is equal to one, and instead of \eqref{G:40:32:8} we use \eqref{G:40:32} with $C$ as in \eqref{G:40:32:4}. The argument then goes through without problems except when $N_2 \ll N_0 \sim N_1$. To be precise, instead of \eqref{G:40:52} we will now have
\begin{equation}\label{G:40:52:2}
  A = \frac{(\Nmax^{12}\gamma\gamma')^2\Nmin^{01}}{N_0L_0L_2},
\end{equation}
leading to
\begin{equation}\label{G:40:54:2}
  A
  \lesssim
  \frac{(\Nmax^{12})^2\Nmin^{01}}{N_0L_0L_2}
  \cdot \frac{N_0L_2}{N_1N_2}
  \cdot \min\left(1,\frac{L_0}{\Nmin^{34}}\right)
  = \frac{\Nmin^{01}}{\Nmin^{012}} \times \text{r.h.s.}\eqref{G:40:54},
\end{equation}
so we are done except for $N_2 \ll N_0 \sim N_1$. Then we must gain a factor $N_2/N_0$ in \eqref{G:40:52:2}. We assume $N_2 \gg 1$, since otherwise \eqref{G:20:8:2} applies, and we assume $\pm_0=\pm_{12}$, as otherwise \eqref{E:6} applies. Thus \eqref{G:20:8} holds, and we use this to make an extra angular decomposition for $(\pm_0\xi_0,\pm_1\xi_1)$. In view of \eqref{G:20:16}, we then replace $u_1^{\gamma,\omega_1}$ and $u_2^{\gamma,\omega_2}$ by $u_1^{\gamma,\omega_1;\alpha,\omega_1'}$ and $\Proj_{H_d(\omega_1')} u_2^{\gamma,\omega_2}$, with $d$ as in  \eqref{G:20:16} and $\omega_1' \in \Omega(\alpha)$.
The spatial output $\xi_0$ is restricted to a tube of radius $r' \sim N_0\alpha \sim N_2\gamma$
around $\R\omega_1'$, replacing $r \sim N_0 \gamma$ used in the previous section. Decomposing into cubes as in the previous section, applying \eqref{G:40:32}, \eqref{G:40:32:4} and \eqref{G:40:42}, with $r$ replaced by $r'$, we get
$$
  J_{\boldN,\boldL}
  \lesssim
  \frac{(r'\gamma')^{1/2}}{(L_0L_2)^{1/4}}
  N_0^{7/8} \left(L_0L_0'L_1L_2L_3L_4\right)^{3/8}
  B^{1/2} 
  \prod_{j=1}^4 \norm{u_j},
$$
where $B$ is given by \eqref{G:20:20}. So now instead of \eqref{G:40:52} we have
$$
  A = \frac{(r'\gamma')^2}{L_0L_2} B^2,
$$
and \eqref{G:40:54} is replaced by
$$
  A
  \lesssim
  \frac{N_2^2}{L_0L_2}
  \cdot \frac{N_0L_2}{N_1N_2}
  \min\left(1,\frac{L_0}{\Nmin^{34}}\right) B^2
  \lesssim
  \frac{N_2}{L_0}
  \min\left(1,\frac{L_0}{\Nmin^{34}}\right) B^2.
$$
When \eqref{G:20:22} holds we are done, since then we get
\begin{equation}\label{G:40:56:2}
  A \lesssim \frac{N_0}{L_0} \min\left(1,\frac{L_0}{\Nmin^{34}}\right).
\end{equation}
and by the same argument as at the end of the previous subsection we also know how to deal with the case $\Nmin^{34} \ll N_0$. If \eqref{G:20:22} fails, we only have \eqref{G:20:24}. But to compensate we can use the fact that \eqref{G:40:32} holds with $C^2 \sim r' (N_2\gamma) \Lmin^{0'1}$, as follows from \eqref{G:20:26}. In effect we then get \eqref{G:40:56:2}.

\subsection{Subcase $\theta_{12},\theta_{34} \ll \phi$, $\min(\theta_{01},\theta_{02}) < \min(\theta_{03},\theta_{04})$, $L_2 \le L_0'$, $L_4 > L_0$}

Then
$$
  \abs{q_{1234}} \lesssim \phi^2
  \lesssim \min(\theta_{03},\theta_{04})^2
  \lesssim \min(\theta_{03},\theta_{04}) \left( \frac{L_4}{N_0} \right)^{1/2},
$$
and we proceed as in section \ref{G:14:0}, but recalling also that we have Lemma \ref{E:Lemma1} at our disposal. The result is that we can dominate $J_{\boldN,\boldL}$ by the last line of \eqref{G:14:8}, but without the factor $(N_0/\Nmin^{34})^{1/4}$, so interpolating with the trivial estimate \eqref{TrivialEstimate} we obtain \eqref{MainDyadic}.

\subsection{Subcase $\theta_{12},\theta_{34} \ll \phi$, $\min(\theta_{01},\theta_{02}) < \min(\theta_{03},\theta_{04})$, $L_2 \le L_0'$, $L_4 \le L_0$}

We modify the argument from subsection \ref{G:40:0:0}. Since $\theta_{12},\theta_{34} \ll 1$, \eqref{D:134:10:2} holds, and
\begin{equation}\label{D:134:10:4}
  \theta_{34} \lesssim \gamma' \equiv \left(\frac{N_0L_0}{N_3N_4}\right)^{1/2}.
\end{equation}
Now $\abs{q_{1234}} \lesssim \phi^2$, and \eqref{G:40:0:2} holds, hence the factor $\gamma_{34}$ in \eqref{G:40:2} must be replaced by $(L_0/N_0)^q$ for some $0 \le q \le 1/2$. Taking $q=0$ or $q=1/4$ we get \eqref{G:40:42:2}, but with the factor
$$
  (\gamma')^{1/2}
  \sim \min\left(1,\frac{L_0}{\Nmin^{34}}\right)^{1/4}
$$  
replaced by
$$
  \min\left(1,\frac{L_0}{N_0}\right)^{1/4} \sum_{0 < \gamma_{34} \lesssim \gamma'} 1,
$$
bur of course the sum diverges.

To fix the problem, observe that the separation assumption $\theta(\omega_3,\omega_4) \ge 3\gamma_{34}$ is only needed when we apply the null form estimate \eqref{G:40:40}, i.e.~when $r\gamma_{34}$ dominates in the definition of $d$ in \eqref{G:40:16}, but then
$$
  \gamma_{34} \gtrsim \frac{r}{\Nmin^{34}} \sim \frac{\Nmax^{12}}{\Nmin^{34}}\gamma
  = \frac{\Nmax^{12}}{\Nmin^{34}} \left( \frac{N_0L_0'}{N_1N_2} \right)^{1/2}.
$$
On the other hand, we also have the upper bound \eqref{D:134:10:4} for $\gamma_{34}$. The cardinality of the this set of dyadic numbers $\gamma_{34}$ is $O(\log \angles{L_0})$. Recall that we used symmetry to assume that the second term in \eqref{E:2} dominates, hence we will also pick up the symmetric factor $O(\log\angles{L_0'})$ in the final estimate. So to summarize, if $\theta_{34} \gtrsim r/\Nmin^{34}$, then effectively the factor $(\gamma')^{1/2}$ in the last line of \eqref{G:40:42:2} is replaced by $\min\left(1,L_0/N_0\right)^{1/4} \log\angles{L_0} \log\angles{L_0'}$, hence we gain a factor $\Nmin^{34}/N_0$ in the right hand side of \eqref{G:40:56}, so we get the desired estimate \eqref{MainDyadic}.

It remains to consider $\theta_{34} \ll r/\Nmin^{34}$, but then we do not need the separation, so here we can avoid summation over $\gamma_{34}$ altogether by using Lemma \ref{WhitneyLemma2} instead of Lemma \ref{WhitneyLemma1}, hence we do not pick up any logarithmic factors.

\subsection{Subcase $\theta_{12},\theta_{34} \ll \phi$, $\min(\theta_{01},\theta_{02}) < \min(\theta_{03},\theta_{04})$, $L_2 > L_0'$, $L_4 \le L_0$}

This follows by the argument from section \ref{G:40:51} with the same modifications as in the previous subsection.

\section{Proof of the trilinear estimate}\label{S}

Here we prove \eqref{Trilinear} for given signs $\pm_1,\pm_2$:
\begin{equation}\label{TrilinearNew}
  \abs{I}
  \lesssim T^{1/4} [\norm{\psi_0}^2+D_T(0)] \norm{\psi_1}_{X_{\pm_1}^{0,1/2;1}}
  \norm{\psi_2}_{X_{\pm_2}^{0,1/4;1}},
\end{equation}
where
$$
  D_T(0) =
  T^{1/2} \sum_{0 < N_0 < 1/T}
  \norm{\Proj_{\abs{\xi_0} \sim N_0} (\mathbf E_0^{\text{df}},B_0^3)}
  +
  \norm{\Proj_{\abs{\xi} \ge 1/T} (\mathbf E_0^{\text{df}},B_0^3)}_{H^{-1/2}}
$$
and
\begin{align*}
  I &=
  \int \rho A_\mu^{\mathrm{hom.}}
  \Innerprod{\boldsymbol\alpha^\mu \mathbf\Pi_{\pm_1} \psi_1}{\mathbf\Pi_{\pm_2} \psi_2} \, dt \, dx
  \\
  &\simeq
  \int
  \widetilde{\rho A_{j}^{\mathrm{hom.}}}(X_0)
  \sigma^j(X_1,X_2)
  \abs{\widetilde{\psi_1}(X_1)}
  \abs{\widetilde{\psi_2}(X_2)}
  \, d\mu^{12}_{-X_0}
  \, dX_0
\end{align*}
with
$$
  \sigma^j(X_1,X_2) = \innerprod{\boldsymbol\alpha^j\mathbf\Pi(\pm_1\xi_1)z_1(X_1)}{\mathbf\Pi(\pm_2\xi_2)z_2(X_2)}.
$$
Here $X_j=(\tau_j,\xi_j)$ and we write $\widetilde{\psi_j} = z_j \bigabs{\widetilde{\psi_j}}$ with $\abs{z_j} = 1$. The convolution measure $d\mu^{12}_{-X_0}$ is given by the rule in \eqref{Convolution}, hence $X_0 = X_2 - X_1$. Recall also that we can insert the time cut-off $\rho_T$ in front of the $\psi_j$ in $I$ whenever needed.

Corresponding to the regions $\abs{\xi_0} < 1/T$ and $\abs{\xi_0} \ge 1/T$ we split
$$
  I =  I_{\abs{\xi_0} < 1/T} + I_{\abs{\xi_0} \ge 1/T}.
$$
and claim that
\begin{equation}\label{LowFreq}
  I_{\abs{\xi_0} < 1/T}
  \lesssim
  \left( \sum_{0 < N_0 \le 1/T}
  \norm{\Proj_{\abs{\xi_0} \sim N_0} (\mathbf E_0^{\text{df}},B_0^3)}
  +
  T^{-1/2}
  \norm{\psi_0}^2
  \right)
  \norm{\psi_1}\norm{\psi_2}
\end{equation}
and
\begin{equation}\label{HighFreq0}
  I_{\abs{\xi_0} \ge 1/T}
  \lesssim
  \left( \norm{\Proj_{\abs{\xi} \ge 1/T} (\mathbf E_0^{\text{df}},B_0^3)}_{H^{-1/2}} + \norm{\psi_0}^2 \right)
  \norm{\psi_1}_{X^{0,1/4;1}_{\pm_1}}
  \norm{\psi_2}_{X^{0,1/4;1}_{\pm_2}}.
\end{equation}
But we are allowed to insert $\rho_T$ in front of the $\psi$'s, and in \eqref{LowFreq} we use \eqref{Cutoff1} to get
\begin{equation}\label{Tgain1}
  \norm{\rho_T\psi_1}\norm{\rho_T\psi_2}
  \lesssim
  T^{1/2}\norm{\psi_1}_{X^{0,1/2;1}_{\pm_1}}
  T^{1/4}\norm{\psi_2}_{X^{0,1/4;1}_{\pm_2}},
\end{equation}
whereas in \eqref{HighFreq0} we get from \eqref{Cutoff2},
\begin{equation}\label{Tgain2}
  \norm{\rho_T\psi_1}_{X^{0,1/4;1}_{\pm_1}}
  \lesssim
  T^{1/4}
  \norm{\psi_1}_{X^{0,1/2;1}_{\pm_1}}.
\end{equation}
Combining \eqref{LowFreq}--\eqref{Tgain2} we obtain \eqref{TrilinearNew}, hence it suffices to prove the claimed estimates \eqref{LowFreq} and \eqref{HighFreq0}.

For convenience we shall denote by $c = 1/T \gg 1$ the cutoff point between low and high frequencies.

By our choice of data, $A_0^{\mathrm{hom.}} = 0$. Using \eqref{WaveSplitting} we split $A_j^{\mathrm{hom.}} = A_{j,+}^{\mathrm{hom.}} + A_{j,-}^{\mathrm{hom.}}$ for $j=1,2$, and we split $I$ accordingly. Note that
$$
  \widetilde{A_{j,\pm_0}^{\mathrm{hom.}}}(X_0)
  =
  \delta(\tau_0\pm_0\abs{\xi_0}) \left( \frac{\widehat{a_j}(\xi_0)}{2} \pm_0 \frac{i\widehat{\dot a_j}(\xi_0)}{2\abs{\xi_0}} \right)
  =
  \delta(\tau_0\pm_0\abs{\xi_0}) \frac{g_j^{\pm_0}(\xi_0)}{\abs{\xi_0}^{1/2}},
$$
where
$$
  g_j^{\pm_0}(\xi_0)
  =
  \abs{\xi_0}^{1/2}
  \left( \frac{\widehat{a_j}(\xi_0)}{2} \pm_0 \frac{i\widehat{\dot a_j}(\xi_0)}{2\abs{\xi_0}} \right).
$$
Since $\mathbf a = - \Delta^{-1} (\partial_2 B^3_0, - \partial_1 B^3_0, 0)$ and $\dot{\mathbf a} = - \mathbf E_0 = - \mathbf E^{\mathrm{df}}_0 - \Delta^{-1}\nabla(\abs{\psi_0}^2)$,
\begin{equation}\label{S:6}
\begin{aligned}
  \norm{\chi_{\abs{\xi_0} \ge c} g^{\pm_0}}
  &\lesssim \norm{\Proj_{\abs{\xi} \ge c} (\mathbf E_0^{\text{df}},B_0^3)}_{H^{-1/2}} + \norm{\abs{\psi_0}^2}_{H^{-3/2}}
  \\
  &\lesssim \norm{\Proj_{\abs{\xi} \ge c} (\mathbf E_0^{\text{df}},B_0^3)}_{H^{-1/2}} + \norm{\psi_0}^2,
\end{aligned}
\end{equation}
where $\bignorm{\abs{\psi_0}^2}_{H^{-3/2}} \lesssim \norm{\psi_0}^2$ by Lemma \ref{SobolevLemma}. 

\subsection{Estimate for $I=I_{\abs{\xi_0} \ge c}$}

We want \eqref{HighFreq0}, but in view of \eqref{S:6} it suffices to prove
\begin{equation}\label{HighFreq}
  I
  \lesssim
  \left( \norm{\chi_{\abs{\xi_0} \ge c} g^{\pm_0}} + \norm{\psi_0}^2 \right)
  \norm{\psi_1}_{X^{0,1/4;1}_{\pm_1}}
  \norm{\psi_2}_{X^{0,1/4;1}_{\pm_2}}
\end{equation}
for
$$
  I
  =
  \int \chi_{\abs{\xi_0} \ge c}
  \widehat{\rho}(\tau_0\pm_0\abs{\xi_0})
  \frac{g_j^{\pm_0}(\xi_0)}
  {\abs{\xi_0}^{1/2}}\sigma^j(X_1,X_2)
  \abs{\widetilde{\psi_1}(X_1)}
  \abs{\widetilde{\psi_2}(X_2)}
  \, d\mu^{12}_{-X_0}
  \, dX_0
$$
with any combination of signs $\pm_0,\pm_1,\pm_2$. Taking the absolute value and using dyadic decomposition we get, since $\widehat\rho$ is rapidly decreasing,
\begin{equation}\label{S:20}
  \abs{I}
  \lesssim
  \sum_{\boldN,\boldL} \frac{I_{\boldN,\boldL}}{N_0^{1/2}L_0},
\end{equation}
where $\boldN = (N_0,N_1,N_2)$ with $N_k \sim \angles{\xi_j}$, $\boldL = (L_0,L_1,L_2)$ with $L_k = \angles{\tau_k\pm_k\abs{\xi_k}}$ and
$$
  I_{\boldN,\boldL}
  =
  \int \chi_{\abs{\xi_0} \ge c} 
  \frac{\chi_{K^{\pm_0}_{N_0,L_0}}(X_0)}{\angles{\tau_0\pm_0\abs{\xi_0}}}
  \abs{\sigma^j(X_1,X_2)g_j^{\pm_0}(\xi_0)}  \widetilde{u_1}(X_1)
  \widetilde{u_2}(X_2)
  \, d\mu^{12}_{-X_0}
  \, dX_0
$$
with $\widetilde{u_k} = \chi_{K^{\pm_k}_{N_k,L_k}} \bigabs{\widetilde{\psi_k}}$. Note the implicit summation over $j=1,2$.

Since $\nabla \cdot \mathbf a = 0$ and $\nabla \cdot \dot{\mathbf a} = - \abs{\psi_0}^2$, we observe that
\begin{equation}\label{S:10}
  \xi_0^j g_j^{\pm_0}(\xi_0) \simeq \abs{\xi_0}^{-1/2} \widehat{\abs{\psi_0}^2}(\xi_0).
\end{equation}
Using this property, it was proved in \cite{Selberg:2008b} that $\sigma^j = \sigma^j(X_1,X_2)$ satisfies
\begin{equation}\label{S:40}
  \abs{\sigma^jg_j^{\pm_0}(\xi_0)}
  \lesssim \theta_{12} \abs{g^{\pm_0}(\xi_0)}
  +
  \min(\theta_{01},\theta_{02}) \abs{g^{\pm_0}(\xi_0)}
  +
  \abs{\xi_0}^{-3/2} \Bigabs{\widehat{\abs{\psi_0}^2}(\xi_0)}
\end{equation}
where
$$
  \theta_{kl} = \theta(\pm_k\xi_k,\pm_l\xi_l).
$$
Correspondingly we split
$$
  I_{\boldN,\boldL}
  \lesssim
  I^1_{\boldN,\boldL}
  +
  I^2_{\boldN,\boldL}
  +
  I^3_{\boldN,\boldL}.
$$

\subsection{Estimate for $I^1_{\boldN,\boldL}$}

Defining
$$
  \widetilde{u_0}(X_0)
  =
  \chi_{K^{\pm_0}_{N_0,L_0}}(X_0)
  \frac{\chi_{\abs{\xi_0} \ge c} \abs{g^{\pm_0}(\xi_0)}}{\angles{\tau_0\pm_0\abs{\xi_0}}}
$$
and using \eqref{Bilinear3} and Lemma \ref{AnglesLemma} we get, for $0 \le p \le 1/2$,
\begin{align*}
  I^1_{\boldN,\boldL}
  &=
  \int
  \theta_{12}
  \widetilde{u_0}(X_0) \widetilde{u_1}(X_1) \widetilde{u_2}(X_2)  \, d\mu^{12}_{-X_0}
  \, dX_0
  \\
  &\lesssim
  \left(\frac{\Lmax^{012}}{\Nmin^{12}}\right)^{p}
  \left(\Nmin^{012}\Nmin^{12}N_0\Lmed^{012}\right)^{1/4} L_0^{1/2}
  \norm{u_0}\norm{u_1}\norm{u_2},
\end{align*}
and estimating $\Lmed^{012}\Lmax^{012} \le L_0L_1L_2$,
\begin{equation}\label{Trilinear1}
  \sum_{\boldN,\boldL} \frac{I^1_{\boldN,\boldL}}{N_0^{1/2}L_0}
  \lesssim
  \sum_{\boldN,\boldL}
  \left(\frac{\Lmax^{012}}{\Nmin^{12}}\right)^{p-1/4}
  \left(\frac{\Nmin^{012}}{N_0}\right)^{1/4}
  \frac{\left(L_1 L_2\right)^{1/4}}{L_0^{1/4}}
  \norm{u_0}\norm{u_1}\norm{u_2}.
\end{equation}

If we exclude for the moment the case $N_0 \ll N_1 \sim N_2$, and take $p=1/4$, \eqref{Trilinear1} gives the desired estimate: We first sum the $N$'s using the factor $(\Nmin^{012}/N_0)^{1/4}$ for the smallest $N$ and Cauchy-Schwarz for the two largest $N$'s. Then we sum $L_0$, and finally we sum $L_1$ and $L_2$ using the definition of the norm on $X^{0,1/4;1}_{\pm}$, obtaining
\begin{align*}
  \sum_{\boldN,\boldL} \frac{I^1_{\boldN,\boldL}}{N_0^{1/2}L_0}
  &\lesssim
  \norm{\frac{\chi_{\abs{\xi_0} \ge c} g^{\pm_0}(\xi_0)}{\angles{\tau_0\pm_0\abs{\xi_0}}}}
  \norm{\psi_1}_{X^{0,1/4;1}_{\pm_1}}
  \norm{\psi_2}_{X^{0,1/4;1}_{\pm_2}}
  \\
  &\lesssim
  \norm{\chi_{\abs{\xi_0} \ge c} g^{\pm_0}}
  \norm{\psi_1}_{X^{0,1/4;1}_{\pm_1}}
  \norm{\psi_2}_{X^{0,1/4;1}_{\pm_2}}
\end{align*}
as required for \eqref{HighFreq}.

There remains the interaction $N_0 \ll N_1 \sim N_2$. Then we need to find a way to sum $N_0$. If $N_0 \ge \Lmax^{012}$, there is no problem, since we can take $p=1/2$ instead of $p=1/4$ in \eqref{Trilinear1}, thereby gaining an extra factor
$$
  \left(\frac{\Lmax^{012}}{\Nmin^{12}}\right)^{1/4}
  \le
  \left(\frac{\Lmax^{012}}{N_0}\right)^{1/4}
$$
which can be used to sum $N_0$ if $N_0 \gtrsim \Lmax^{012}$. But what if $N_0 < \Lmax^{012}$? Then instead of \eqref{Bilinear3} we use \eqref{Bilinear2}, obtaining (estimating trivially $\theta_{12} \lesssim 1$),
$$
  I^1_{\boldN,\boldL}
  \lesssim
  \left(N_0^{3/2} L_0 (\Lmin^{12})^{1/2} \right)^{1/2}
  \norm{u_0}\norm{u_1}\norm{u_2},
$$
hence
\begin{equation}\label{Trilinear2}
  \frac{I^1_{\boldN,\boldL}}{N_0^{1/2}L_0}
  \lesssim
  \frac{N_0^{1/4}  (\Lmin^{12})^{1/4}}{L_0^{1/2}} 
  \norm{u_0}\norm{u_1}\norm{u_2}.
\end{equation}
First sum $N_1 \sim N_2$ using Cauchy-Schwarz, then sum $N_0$ using
\begin{equation}\label{NzeroSum}
  \sum_{N_0 < \Lmax^{012}} N_0^{1/4}
  \sim (\Lmax^{012})^{1/4} \le (L_0\Lmax^{12})^{1/4},
\end{equation}
then sum $L_0$ using the remaining factor $L_0^{-1/4}$, and finally sum $L_1$ and $L_2$ as above.

\subsection{Estimate for $I^2_{\boldN,\boldL}$}

The difference from the previous subsection is that $\theta_{12}$ is replaced by
$$
  \min(\theta_{01},\theta_{02})
  \lesssim
  \left( \frac{\Lmax^{012}}{N_0} \right)^{p} \qquad (0 \le p \le 1/2),
$$
and this is better than the estimate we used for $\theta_{12}$ except if $N_0 \ll N_1 \sim N_2$. But in that case, by \eqref{Bilinear2},
$$
  \frac{I^2_{\boldN,\boldL}}{N_0^{1/2}L_0}
  \lesssim
  \frac{1}{N_0^{1/2}L_0}
  \left( \frac{\Lmax^{012}}{N_0} \right)^{p}
  \left(N_0^{3/2} L_0 (\Lmin^{12})^{1/2} \right)^{1/2}
  \norm{u_0}\norm{u_1}\norm{u_2}.
$$
If $N_0 \ge \Lmax^{012}$, we take $p=1/2$, obtaining (since $\Lmin^{12}\Lmax^{012} \le L_0L_1L_2$)
$$
  \frac{I^2_{\boldN,\boldL}}{N_0^{1/2}L_0}
  \lesssim
  \left(\frac{\Lmax^{012}}{N_0}\right)^{1/4}
  \frac{L_1^{1/4}L_2^{1/4}}{L_0^{1/4}}
  \norm{u_0}\norm{u_1}\norm{u_2},
$$
so summing is no problem. If $N_0 < \Lmax^{012}$, then with $p=0$ we get \eqref{Trilinear2} for $I^2_{\boldN,\boldL}$, and using \eqref{NzeroSum} we can again sum.

\subsection{Estimate for $I^3_{\boldN,\boldL}$}

We may assume $\theta_{12} \ll 1$, since otherwise the estimate for $I^1_{\boldN,\boldL}$ applies to $I_{\boldN,\boldL}$ as a whole by simply estimating $\abs{\sigma^j(X_1,X_2)} \lesssim 1$.

Then by Lemma \ref{AnglesLemma},
$$
  \theta_{12} \lesssim \gamma \equiv \left( \frac{N_0\Lmax^{012}}{N_1N_2} \right)^{1/2},
$$
hence
$$
  I^3_{\boldN,\boldL}
  \lesssim N_0^{-3/2}
  \int
  \chi_{\theta_{12} \lesssim \gamma}
  \widetilde{u_0}(X_0) \widetilde{u_1}(X_1) \widetilde{u_2}(X_2)  
  \, d\mu^{12}_{-X_0}
  \, dX_0
$$
with
$$
  \widetilde{u_0}(X_0)
  =
  \chi_{K^{\pm_0}_{N_0,L_0}}(X_0) \frac{\Bigabs{\widehat{\abs{\psi_0}^2}(\xi_0)}}{\angles{\tau_0\pm_0\abs{\xi_0}}}.
$$
By Lemma \ref{WhitneyLemma2} applied to the pair $(\pm_1\xi_1,\pm_2\xi_2)$,
\begin{equation}\label{S:60}
  I^3_{\boldN,\boldL}
  \lesssim N_0^{-3/2}
  \sum_{\omega_1,\omega_2}
  \int
  \widetilde{u_0}(X_0)
  \widetilde{u_1^{\gamma,\omega_1}}(X_1)
  \widetilde{u_2^{\gamma,\omega_2}}(X_2)
  \, d\mu^{12}_{-X_0}
  \, dX_0,
\end{equation}
where the sum is over $\omega_1,\omega_2 \in \Omega(\gamma)$ with $\theta(\omega_1,\omega_2) \lesssim \gamma$ and $u_j^{\gamma,\omega_j} = \Proj_{\pm_j\xi_j \in \Gamma_{\gamma}(\omega_j)} u$. So in the last integral, $\xi_1,\xi_2$ are both restricted to a tube of radius
$$
  r \sim \Nmax^{12}\gamma \sim \left( \frac{\Nmax^{12}N_0\Lmax^{012}}{\Nmin^{12}} \right)^{1/2}
$$
around $\R\omega_1$, hence the same is true for $\xi_0=\xi_2-\xi_1$, so we get
\begin{align*}
  I^3_{\boldN,\boldL}
  &\lesssim N_0^{-3/2}
  \sum_{\omega_1,\omega_2}
  \norm{\Proj_{\R \times T_r(\omega_1)} u_0}
  \norm{ \Proj_{K^{\pm_0}_{N_0,L_0}} \left( u_1^{\gamma,\omega_1}
  \overline{u_2^{\gamma,\omega_2}} \right) }
  \\
  &\lesssim
  N_0^{-3/2}
  \sum_{\omega_1,\omega_2}
  \norm{\Proj_{T_r(\omega_1)}
  \Proj_{\angles{\xi_0} \sim N_0} \abs{\psi_0}^2}
  \\
  &\quad \times\left( N_0^{1/2} \Nmin^{012} L_0 (\Lmin^{12})^{1/2} \right)^{1/2} \norm{u_1^{\gamma,\omega_1}} 
  \norm{u_2^{\gamma,\omega_2}},
\end{align*}
where we used \eqref{Bilinear2}. Applying the estimate (by Plancherel and Cauchy-Schwarz this reduces to the obvious fact that the area of intersection of a strip of width $r$ and a disk of radius $N_0$ is comparable to $rN_0$)
\begin{equation}\label{S:56}
  \sup_{\omega \in \mathbb S^1} \norm{\Proj_{T_r(\omega)}
  \Proj_{\angles{\xi_0} \sim N_0} \abs{\psi_0}^2}
  \lesssim
  \left( r N_0 \right)^{1/2} \norm{\psi_0}^2,
\end{equation}
and summing $\omega_1,\omega_2$ as in \eqref{OmegaSum}, we then obtain
\begin{align*}
  I^3_{\boldN,\boldL}
  &\lesssim N_0^{-3/2}
  \left( r N_0 \right)^{1/2} \norm{\psi_0}^2
  \left( N_0^{1/2} \Nmin^{012} L_0 (\Lmin^{12})^{1/2} \right)^{1/2}
  \norm{u_1}\norm{u_2}
  \\
  &\lesssim
  \left( \frac{\Nmax^{12}\Nmin^{012}}{N_0\Nmin^{12}} \right)^{1/4}
  L_0^{1/2} \left( \Lmin^{12} \Lmax^{012}\right)^{1/4} \norm{\psi_0}^2
  \norm{u_1}\norm{u_2}
  \\
  &\lesssim
  L_0^{1/2} \left( \Lmin^{12} \Lmax^{012}\right)^{1/4} \norm{\psi_0}^2
  \norm{u_1}\norm{u_2}
  \\
  &\lesssim
  L_0^{3/4} (L_1L_2)^{1/4} \norm{\psi_0}^2
  \norm{u_1}\norm{u_2},
\end{align*}
where we used $\Nmin^{012} \Nmax^{012} \sim N_0\Nmin^{12}$ and $\Lmin^{12} \Lmax^{012} \le L_0L_1L_2$. Thus,
$$
  \sum_{\boldN,\boldL} \frac{I^3_{\boldN,\boldL}}{N_0^{1/2}L_0}
  \lesssim
  \norm{\psi_0}^2 \sum_{\boldN,\boldL} \frac{(L_1^{1/4}\norm{u_1})(L_2^{1/4}\norm{u_2})}{N_0^{1/2} L_0^{1/4}},
$$
so summing the $N$'s is easy (if $N_0 \sim \Nmax^{12}$, all the $N$'s can be summed using the factor $N_0^{-1/2}$; if $N_0 \ll N_1 \sim N_2$, we sum $N_1 \sim N_2$ using Cauchy-Schwarz and $N_0$ using the factor $N_0^{-1/2}$), we can sum $L_0$ using the factor $L_0^{-1/4}$, and finally we sum $L_1$ and $L_2$ using the definition of the norm on $X^{0,1/4;1}_{\pm}$, obtaining
$$
  \sum_{\boldN,\boldL} \frac{I^3_{\boldN,\boldL}}{N_0^{1/2}L_0}
  \lesssim
  \norm{\psi_0}^2
  \norm{\psi_1}_{X^{0,1/4;1}_{\pm_1}}
  \norm{\psi_2}_{X^{0,1/4;1}_{\pm_2}}
$$
as needed for \eqref{HighFreq}.

\subsection{Estimate for $I= I_{\abs{\xi_0} < c}$} Since $\abs{\sigma^j(X_1,X_2)} \lesssim 1$, it suffices to prove the bound \eqref{LowFreq} for
\begin{align*}
  I_1
  &=
  \int_{\abs{\xi_0} < c}
  \abs{\widehat\rho(\tau_0\pm_0\abs{\xi_0})}
  \abs{\widehat{\mathbf a}(\xi_0)} 
  \abs{\widetilde{\psi_1}(X_1)} \abs{\widetilde{\psi_2}(X_2)}
  \, d\mu^{12}_{-X_0}
  \, dX_0,
  \\
  I_2
  &=
  \int_{1 \le \abs{\xi_0} < c}
  \abs{\widehat\rho(\tau_0\pm_0\abs{\xi_0})}
  \frac{\bigabs{\widehat{\dot{\mathbf a}}(\xi_0)}}{\abs{\xi_0}} 
  \abs{\widetilde{\psi_1}(X_1)} \abs{\widetilde{\psi_2}(X_2)}
  \, d\mu^{12}_{-X_0}
  \, dX_0,
  \\
  I_3
  &=
  \int_{\abs{\xi_0} < 1}
  \abs{\frac{\widehat\rho(\tau_0+\abs{\xi_0}) - \widehat\rho(\tau_0-\abs{\xi_0})}
  { \abs{\xi_0} }}
  \bigabs{\widehat{\dot{\mathbf a}}(\xi_0)}
  \abs{\widetilde{\psi_1}(X_1)} \abs{\widetilde{\psi_2}(X_2)}
  \, d\mu^{12}_{-X_0}
  \, dX_0.
\end{align*}
Since $\widehat \rho$ is rapidly decreasing and $\mathbf a = - \Delta^{-1} (\partial_2 B^3_0, - \partial_1 B^3_0, 0)$,
\begin{align*}
  I_1
  &\lesssim
  \int
  \frac{\chi_{\abs{\xi_0} < c} \abs{\widehat{\mathbf a}(\xi_0)}}{\angles{\tau_0\pm_0\abs{\xi_0}}^2}
  \abs{\widetilde{\psi_1}(X_1)} \abs{\widetilde{\psi_2}(X_0+X_1)}
  \, dX_1
  \, dX_0
  \\
  &\le
  \int
  \frac{\chi_{\abs{\xi_0} < c} \abs{\widehat{\mathbf a}(\xi_0)}}{\angles{\tau_0\pm_0\abs{\xi_0}}^2}
  \, dX_0
  \norm{\psi_1}\norm{\psi_2}
  \\
  &\lesssim
  \int
  \chi_{\abs{\xi_0} < c} \abs{\widehat{\mathbf a}(\xi_0)}
  \, d\xi_0
  \norm{\psi_1}\norm{\psi_2}
  \\
  &\lesssim
  \int
  \frac{\chi_{\abs{\xi_0} < c} \abs{\widehat{B^3_0}(\xi_0)}}{\abs{\xi_0}} 
  \, d\xi_0
  \norm{\psi_1}\norm{\psi_2}
  \\
  &\le
  \sum_{0 < N_0 < c}
  \left(\int_{\abs{\xi_0} \sim N_0}
  \frac{1}{\abs{\xi_0}^2}
  \, d\xi_0 \right)^{1/2}
  \norm{\Proj_{\abs{\xi_0} \sim N_0} B^3_0}
  \norm{\psi_1}\norm{\psi_2}
  \\
  &\lesssim
  \sum_{0 < N_0 < c}
  \norm{\Proj_{\abs{\xi_0} \sim N_0} B^3_0}
  \norm{\psi_1}\norm{\psi_2}.
\end{align*}
Similarly, since $\dot{\mathbf a} = - \mathbf E_0 = - \mathbf E^{\mathrm{df}}_0 - \Delta^{-1}\nabla(\abs{\psi_0}^2)$,
\begin{align*}
  I_2
  &\lesssim
  \left( \sum_{1 \le N_0 < c}
  \norm{\Proj_{\abs{\xi_0} \sim N_0} \mathbf E_0^{\text{df}}}
  +
  \int_{\abs{\xi_0} < c} \frac{1}{\angles{\xi_0}^2} \Bigabs{\widehat{\abs{\psi_0}^2}(\xi_0)} \, d\xi_0
  \right)
  \norm{\psi_1}\norm{\psi_2}
  \\
  &\lesssim
  \left( \sum_{1 \le N_0 < c}
  \norm{\Proj_{\abs{\xi_0} \sim N_0} \mathbf E_0^{\text{df}}}
  +
  \left( \int_{\abs{\xi_0} < c} \frac{d\xi_0}{\angles{\xi_0}} \right)^{1/2}
  \norm{\abs{\psi_0}^2}_{H^{-3/2}}
  \right)
  \norm{\psi_1}\norm{\psi_2}
  \\
  &\lesssim
  \left( \sum_{1 \le N_0 < c}
  \norm{\Proj_{\abs{\xi_0} \sim N_0} \mathbf E_0^{\text{df}}}
  +
  c^{1/2}
  \norm{\psi_0}^2
  \right)
  \norm{\psi_1}\norm{\psi_2},
\end{align*}
where $\bignorm{\abs{\psi_0}^2}_{H^{-3/2}} \lesssim \norm{\psi_0}^2$ by Lemma \ref{SobolevLemma}.

Finally, since $\widehat\rho(\tau_0+\abs{\xi_0}) - \widehat\rho(\tau_0-\abs{\xi_0}) = 2\abs{\xi_0} \int_0^1 {\widehat\rho}\,'(\tau_0-\abs{\xi_0}+2s\abs{\xi_0}) \, ds$,
\begin{align*}
  I_3
  &\lesssim
  \int_0^1 \int_{\abs{\xi_0} < 1}
  \abs{{\widehat\rho}\,'(\tau_0-\abs{\xi_0}+2s\abs{\xi_0})}
  \bigabs{\widehat{\dot{\mathbf a}}(\xi_0)}
  \, dX_0 \, ds \norm{\psi_1}\norm{\psi_2}
  \\
  &\lesssim
  \int_{\abs{\xi_0} < 1} \bigabs{\widehat{\dot{\mathbf a}}(\xi_0)}
  \, d\xi_0 \norm{\psi_1}\norm{\psi_2}
  \\
  &\le
  \left( \int_{\abs{\xi_0} < 1} \abs{\widehat{\mathbf E_0^{\text{df}}}(\xi_0)}
  \, d\xi_0
  +
  \int_{\abs{\xi_0} < 1} \frac{1}{\abs{\xi_0}} \Bigabs{\widehat{\abs{\psi_0}^2}(\xi_0)}
  \, d\xi_0
  \right)
  \norm{\psi_1}\norm{\psi_2}
  \\
  &\lesssim
  \left( \norm{\Proj_{\abs{\xi_0} < 1}\mathbf E_0^{\text{df}}}
  +
  \left(\int_{\abs{\xi_0} < 1} \frac{d\xi_0}{\abs{\xi_0}} \right)^{1/2}
  \norm{\abs{D}^{-1/2}\Proj_{\abs{\xi_0} < 1} (\abs{\psi_0}^2)}
  \right)
  \norm{\psi_1}\norm{\psi_2}
  \\
  &\lesssim
  \left( \sum_{0 < N_0 < 1}
  \norm{\Proj_{\abs{\xi_0} \sim N_0} \mathbf E_0^{\text{df}}}
  +
  \norm{\psi_0}^2
  \right)
  \norm{\psi_1}\norm{\psi_2},
\end{align*}
where we estimated
$$
  \norm{\abs{D}^{-1/2}\Proj_{\abs{\xi_0} < 1} (\abs{\psi_0}^2)}
  \le
  \norm{\abs{D}^{-1/2} \angles{D}^{-1} (\abs{\psi_0}^2)}
  \lesssim \norm{\psi_0}^2
$$
by Lemma \ref{SobolevLemma}.

This completes the proof of the trilinear estimate.

\section{Estimates for the electromagnetic field}\label{N}

Here we prove Theorem \ref{EMgrowth}.

Denote by $S^{\mathrm{W}}_\pm(t) = e^{-it(\pm\abs{D})}$ and $S^{\mathrm{KG}}_\pm(t) = e^{-it(\pm\angles{D})}$ the propagators of the evolution operators $-i\partial_t \pm \abs{D}$ and $-i\partial_t \pm \angles{D}$ respectively. Then by Duhamel's principle applied to \eqref{Eeq} and \eqref{Beq},
\begin{align}
  \label{Eexpansion}
  \mathbf E^{\mathrm{df}}_\pm(t)
  &= S^{\mathrm{KG}}_\pm(t) \mathbf E^{\mathrm{df}}_\pm(0)
  \\ \notag
  &\quad -
  \int_0^t S^{\mathrm{KG}}_\pm(t-s) (\pm 2\angles{D})^{-1}  \left[ \mathcal P_{\text{df}}(-\nabla J_0 + \partial_t \mathbf J) - \rho_T\mathbf E^{\mathrm{df}} \right](s) \, ds,
  \\
  \label{Bexpansion}
  B^3_\pm(t)
  &= S^{\mathrm{W}}_\pm(t) B^3_\pm(0)
  -
  \int_0^t  S^{\mathrm{W}}_\pm(t-s) (\pm 2\abs{D})^{-1}\left(\partial_1 J_2 - \partial_2 J_1\right)(s) \, ds,
\end{align}
for $\abs{t} \le T$.  

Since $S^{\mathrm{W}}_\pm(t)$ and $S^{\mathrm{KG}}_\pm(t)$ are unitary,
\begin{align*}
  \norm{S^{\mathrm{KG}}_\pm(t) \mathbf E^{\mathrm{df}}_\pm(0)}_{(T)}
  &=
  \norm{\mathbf E^{\mathrm{df}}_\pm(0)}_{(T)},
  \\
  \norm{S^{\mathrm{W}}_\pm(t) B^3_\pm(0)}_{(T)}
  &=
  \norm{B^3_\pm(0)}_{(T)},
\end{align*}
for all $t$, and this takes care of the first term on the right hand side of \eqref{ModifiedEMgrowth}, hence it remains to prove that, for some $C$ depending only on the charge norm and $\abs{M}$,
\begin{equation}\label{InhomGrowth}
  \sup_{\abs{t} \le T} \norm{I_j(t)}_{(T)} \le CT^{1/2}\log(1/T)
\end{equation}
for the inhomogeneous terms
\begin{align*}
  I_1(t)
  &=
  \int_0^t S^{\mathrm{W}}_\pm(t-s) \abs{D}^{-1} \left(\partial_1 J_2 - \partial_2 J_1\right)(s) \, ds,
  \\
  I_2(t)
  &=
  \int_0^t S^{\mathrm{KG}}_\pm(t-s) \angles{D}^{-1}  \mathcal P_{\text{df}}(-\nabla J_0 + \partial_t \mathbf J)(s) \, ds,
  \\
  I_3(t)
  &=
  \int_0^t S^{\mathrm{KG}}_\pm(t-s) \angles{D}^{-1} (\rho_T\mathbf E^{\mathrm{df}})(s) \, ds.
\end{align*}
These are defined for $\abs{t} \le T$, but after choosing an extension of $\psi$ we can consider all $t \in \R$ and insert the cutoff $\rho_T(t) = \rho(t/T)$ in front of $\psi$, so that
\begin{equation}\label{Jcut}
  J^\mu = \innerprod{\boldsymbol\alpha^\mu\rho_T\psi}{\rho_T\psi}.
\end{equation}

The extensions (or representatives, to be precise) of $\psi_\pm \in X^{0,1/2;1}_\pm(S_T)$, which we still denote $\psi_\pm$ for convenience, can of course be chosen so that
$$
  \norm{\psi_\pm}_{X^{0,1/2;1}_\pm} \le 2\norm{\psi_\pm}_{X^{0,1/2;1}_\pm(S_T)},
$$
and in view of \eqref{IterationBound} we then have
\begin{equation}\label{ExtensionBound}
  \norm{\psi_\pm}_{X^{0,1/2;1}_\pm} \le C_1,
\end{equation}
where $C_1$ only depends on the charge constant. We may further assume
$$
  \mathbf \Pi_\pm \psi_\pm = \psi_\pm,
$$
since this already holds on $S_T$, and replacing $\psi_\pm$ by $\mathbf\Pi_\pm \psi_\pm$ it will hold globally; moreover, applying $\mathbf\Pi_\pm$ does not increase the norm. Having thus chosen the extensions $\psi_\pm$, we define the extension of $\psi$ itself by
$$
  \psi = \psi_+ + \psi_-,
$$
and note that $\mathbf \Pi_\pm \psi = \psi_\pm$ by orthogonality of the projections.

Writing $\psi_j = \rho_T\psi_{\pm_j}$ for given signs $\pm_j$, and applying \eqref{Convolution} to \eqref{Jcut}, we now note that
$$
  \widetilde{J_\kappa}(X_0)
  \simeq
  \sum_{\pm_1,\pm_2}
  \int
  \Innerprod{\boldsymbol\alpha_\kappa\boldsymbol\Pi(\pm_1\xi_1)z_1(X_1)}{\boldsymbol\Pi(\pm_2\xi_2)z_2(X_2)}
  \abs{\widetilde{\psi_1}(X_1)} \abs{\widetilde{\psi_2}(X_2)}
  \, d\mu^{12}_{X_0},
$$
where $X_j = (\tau_j,\xi_j)$ and $\widetilde{\psi_j} = z_j \bigabs{\widetilde{\psi_j}}$ with $\abs{z_j} = 1$.

Observe that the symbol of $(1/i)\partial_\kappa$ is $X_0^\kappa = \tau_0$ for $\kappa=0$, and $X_0^\kappa = \xi_0^\kappa$ for $\kappa=1,2$, where we write $\xi_0=(\xi_0^1,\xi_0^2)$. Thus,
\begin{gather}
  \label{Bfourier}
  \abs{\mathcal F\left(\partial_1 J_2 - \partial_2 J_1\right)(X_0)}
  \le
  \sum_{\pm_1,\pm_2} \sum_{k,l=1,2; k \neq l} F^{\pm_1,\pm_2}_{kl}(X_0),
  \\
  \label{Efourier}
  \abs{\mathcal F\left[\mathcal P_{\text{df}}(-\nabla J_0 + \partial_t \mathbf J)\right](X_0)}
  \le
  \sum_{\pm_1,\pm_2} \sum_{k=1,2} F^{\pm_1,\pm_2}_{k0}(X_0),
\end{gather}
where
$$
  F^{\pm_1,\pm_2}_{\kappa\lambda}(X_0)
  =
  \int \abs{\sigma_{\kappa\lambda}(X_1,X_2)}
  \abs{\widetilde{\psi_1}(X_1)} \abs{\widetilde{\psi_2}(X_2)}
  \, d\mu^{12}_{X_0}
$$
and the symbol
\begin{align*}
  \sigma_{\kappa\lambda}(X_1,X_2)
  &= X_0^\kappa \Innerprod{\boldsymbol\alpha_\lambda\boldsymbol\Pi(\pm_1\xi_1)z_1(X_1)}{\boldsymbol\Pi(\pm_2\xi_2)z_2(X_2)}
  \\
  &\quad -
  X_0^\lambda \Innerprod{\boldsymbol\alpha_\kappa\boldsymbol\Pi(\pm_1\xi_1)z_1(X_1)}{\boldsymbol\Pi(\pm_2\xi_2)z_2(X_2)}
\end{align*}
has the following null structure:

\begin{lemma}\label{N:Lemma} \emph{(\cite{Selberg:2008b}.)} For any choice of signs $\pm_0,\pm_1,\pm_2$, and writing $\theta_{\kappa\lambda} = \theta(\pm_\kappa\xi_\kappa,\pm_\lambda\xi_\lambda)$ for $\kappa,\lambda=0,1,2$, we have
\begin{align}
  \label{N:12}
  \abs{\sigma_{kl}(X_1,X_2)}
  &\lesssim
  \abs{\xi_0}\theta_{12} + \abs{\xi_0}\min(\theta_{01},\theta_{02}),
  \\
  \label{N:14}
  \abs{\sigma_{k0}(X_1,X_2)}
  &\lesssim
  \abs{\xi_0}\theta_{12} + \abs{\xi_0}\min(\theta_{01},\theta_{02})
  + \abs{\tau_0\pm_0\abs{\xi_0}},
\end{align}
for $k,l=1,2$.
\end{lemma}

To simplify the notation, summations over $\pm_1,\pm_2$ [such as in \eqref{Efourier} and \eqref{Bfourier}] will be tacitly assumed from now on. Moreover, the sign $\pm$ appearing in the definitions of the $I_j$ will be denoted $\pm_0$.

Corresponding to \eqref{N:12} and \eqref{N:14}, respectively, we now split
\begin{align*}
  I_1 &= I_{1,1} + I_{1,2},
  \\
  I_2 &= I_{2,1} + I_{2,2} + I_{2,3},
\end{align*}
by restricting in Fourier space. In fact, for all these terms except $I_{2,3}$ we shall prove something stronger than \eqref{InhomGrowth}, namely $\norm{I_{j,k}(t)}_{(1)} \le CT^{1/2}$. In other words, we will show that
\begin{align}
  \label{Hi}
  \sup_{\abs{t} \le T} \norm{\Proj_{\abs{\xi_0} \ge 1} I_{j,k}(t)}_{H^{-1/2}} &\le CT^{1/2},
  \\
  \label{Lo}
  \sup_{\abs{t} \le T} \sum_{0 < N_0 < 1} \norm{\Proj_{\abs{\xi_0} \sim N_0} I_{j,k}(t)} &\le CT^{1/2},
\end{align}
for $j,k=1,2$. This is stronger than \eqref{InhomGrowth} since by Lemma \ref{MagicLemma}, $\norm{f}_{(T)} \lesssim \norm{f}_{(1)}$.

\subsection{Estimate for $I_{1,1}$ with $\abs{\xi_0} \ge 1$}

Now $\abs{\sigma_{kl}} \lesssim \abs{\xi_0}\theta_{12}$, so recalling \eqref{Bfourier} and applying \eqref{Linear3} with $\phi(\xi) = \pm\abs{\xi}$ we get
\begin{equation}\label{EM1}
\begin{aligned}
  &\norm{\Proj_{\abs{\xi_0} \ge 1} I_{1,1}(t)}_{H^{-1/2}}
  \\
  &\quad\lesssim 
  \norm{ \int \frac{1}{\angles{\xi_0}^{1/2}\angles{\tau_0\pm_0\abs{\xi_0}}}
  \theta_{12} \abs{\widetilde{\psi_1}(X_1)} \abs{\widetilde{\psi_2}(X_2)}
  \, d\mu^{12}_{X_0}
  \, d\tau_0 }_{L^2_{\xi_0}}
  \\
  &\quad=
  \sup_{\norm{f}=1}
  \abs{ \int \frac{f(\xi_0)}{\angles{\xi_0}^{1/2}\angles{\tau_0\pm_0\abs{\xi_0}}}
  \theta_{12} \abs{\widetilde{\psi_1}(X_1)} \abs{\widetilde{\psi_2}(X_2)}
  \, d\mu^{12}_{X_0}
  \, d\tau_0 \, d\xi_0 }
  \\
  &\quad\lesssim
  \sup_{\norm{f}=1}
  \sum_{\boldN,\boldL}
  \frac{1}{N_0^{1/2}L_0}
  \left( \frac{\Lmax^{012}}{\Nmin^{12}} \right)^{p}
  L_0^{1/2} \norm{\chi_{\angles{\xi_0} \sim N_0} f} \norm{\Proj_{K^{\pm_0}_{N_0,L_0}} (u_1 \overline{u_2})}
\end{aligned}
\end{equation}
for $0 \le p \le 1/2$, where we used Lemma \eqref{AnglesLemma} to estimate
\begin{equation}\label{Theta12}
  \theta_{12} \lesssim \left( \frac{\Lmax^{012}}{\Nmin^{12}} \right)^{p}
\end{equation}
and we write
$
  \widetilde{u_j} = \chi_{K^{\pm_j}_{N_j,L_j}} \bigabs{\widetilde{\psi_j}}.
$

Assuming $L_1 \le L_2$ by symmetry, we split into the three cases $L_1 \le L_0 \le L_2$, $L_1 \le L_2 \le L_0$ and $L_0 \le L_1 \le L_2$.

\subsubsection{The case $L_1 \le L_0 \le L_2$} Then we take $p=1/4$ and use \eqref{Bilinear2}, so we estimate the above sum by
\begin{align*}
  &\sum_{\boldN,\boldL}
  \chi_{L_1 \le L_0}
  \left( \frac{L_2}{\Nmin^{12}} \right)^{1/4}
  \frac{\left( \Nmin^{012} N_0^{1/2} L_0^{1/2} L_1 \right)^{1/2}}{N_0^{1/2}L_0^{1/2}} 
  \norm{\chi_{\angles{\xi_0} \sim N_0} f} \norm{u_1}\norm{u_2}
  \\
  &=
  \sum_{\boldN,\boldL}
  \chi_{L_1 \le L_0}
  \frac{(\Nmin^{012})^{1/2} L_1^{1/2} L_2^{1/4}}{(N_0\Nmin^{12})^{1/4}L_0^{1/4}}
  \norm{\chi_{\angles{\xi_0} \sim N_0} f} \norm{u_1} \norm{u_2}
  \\
  &\sim
  \sum_{\boldN,\boldL}
  \chi_{L_1 \le L_0}
  \left( \frac{\Nmin^{012}}{\Nmax^{012}} \right)^{1/4}
  \frac{L_1^{1/2} L_2^{1/4}}{L_0^{1/4}}
  \norm{\chi_{\angles{\xi_0} \sim N_0} f} \norm{u_1}\norm{u_2}
\end{align*}
where we used $N_0\Nmin^{12} \sim \Nmax^{012}\Nmin^{012}$. Now we sum the $N's$. Recalling that the two largest $N$'s are comparable, we use $(\Nmin^{012}/\Nmax^{012})^{1/4}$ to sum the smallest $N$, and the two largest $N$'s are summed using Cauchy-Schwarz. Thus we are left with
$$
  \sum_{\boldL}
  \chi_{L_1 \le L_0}
  \frac{L_1^{1/2} L_2^{1/4}}{L_0^{1/4}}
  \norm{f} \norm{\Proj_{K^{\pm_1}_{L_1}}\psi_1}
  \norm{\Proj_{K^{\pm_2}_{L_2}}\psi_2}.
$$
Next we sum $L_0$ using
$$
  \sum_{L_0 \colon L_0 \ge L_1} \frac{L_1^{1/2}}{L_0^{1/4}}
  \sim L_1^{1/4},
$$
and finally, the summations of $L_1$ and $L_2$ give the $X^{0,1/4;1}$-norms of $\psi_1$ and $\psi_2$. So we have shown that the part of the last line of \eqref{EM1} corresponding to $L_1 \le L_0 \le L_2$ is bounded by an absolute constant times
\begin{equation}\label{TimeGain}
  \norm{\psi_1}_{X^{0,1/4;1}_{\pm_1}} \norm{\psi_2}_{X^{0,1/4;1}_{\pm_2}}
  =
  \norm{\rho_T\psi_{\pm_1}}_{X^{0,1/4;1}_{\pm_1}} \norm{\rho_T\psi_{\pm_2}}_{X^{0,1/4;1}_{\pm_2}}
  \le C T^{1/2},
\end{equation}
where we used \eqref{Cutoff2} and \eqref{ExtensionBound}, hence $C$ only depends on the charge constant.

\subsubsection{The case $L_1 \le L_2 \le L_0$} Taking $p=1/4$ and using \eqref{Bilinear1} gives
\begin{align*}
  &\sum_{\boldN,\boldL}
  \chi_{L_1 \le L_0}
  \left( \frac{L_0}{\Nmin^{12}} \right)^{1/4}
  \frac{\left( \Nmin^{012} (\Nmin^{12})^{1/2} L_1 L_2^{1/2} \right)^{1/2}}{N_0^{1/2}L_0^{1/2}} 
  \norm{\chi_{\angles{\xi_0} \sim N_0} f} \norm{u_1}\norm{u_2}
  \\
  &=
  \sum_{\boldN,\boldL}
  \chi_{L_1 \le L_0}
  \left( \frac{\Nmin^{012}}{N_0} \right)^{1/2} \frac{L_1^{1/2} L_2^{1/4}}{L_0^{1/4}}
  \norm{\chi_{\angles{\xi_0} \sim N_0} f} \norm{u_1} \norm{u_2},
\end{align*}
so the argument in the previous subsection works except when $N_0 \ll N_1 \sim N_2$, which we now assume. The problem is then that we have no way of summing $N_0$. To resolve this, divide into $N_0 < L_0$ and $N_0 \ge L_0$. In the latter case we can pick up an extra factor $(L_0/N_1)^{1/4} \ll (L_0/N_0)^{1/4}$ by choosing $p=1/2$ instead of $p=1/4$, allowing us to sum $N_0$. That leaves $N_0 < L_0$. Then we use
$$
  \norm{\Proj_{K^{\pm_0}_{N_0,L_0}} (u_1 \overline{u_2})}
  \lesssim
  \bigl( N_0 L_1\bigr)^{1/4} \bigl( N_1L_2 \bigr)^{1/8}
  \bigl( N_0^2 L_1 \bigr)^{1/4}
  \norm{u_1}
  \norm{u_2}
$$
obtained by interpolation between \eqref{Bilinear1} and \eqref{SobolevType}. Taking $p=1/8$, we thus get
\begin{align*}
  &\sum_{\boldN,\boldL}
  \chi_{L_1 \le L_2 \le L_0} \chi_{N_0 < L_0}
  \left( \frac{L_0}{N_1} \right)^{1/8}
  \frac{\left( N_0^{3/2} N_1^{1/4} L_1 L_2^{1/4} \right)^{1/2}}{N_0^{1/2}L_0^{1/2}} 
  \norm{\chi_{\angles{\xi_0} \sim N_0} f} \norm{u_1}\norm{u_2}
  \\
  &=
  \sum_{\boldN,\boldL}
  \chi_{L_1 \le L_2 \le L_0} \chi_{N_0 < L_0}
  \frac{N_0^{1/4} L_1^{1/2} L_2^{1/8}}{L_0^{3/8}}
  \norm{\chi_{\angles{\xi_0} \sim N_0} f} \norm{u_1} \norm{u_2},
\end{align*}
and summing $N_0 < L_0$ we replace $N_0^{1/4}$ by $L_0^{1/4}$; then we are still left with $L_0^{1/8}$ in the denominator, and summing $L_0 \ge L_2$ we end up with just $L_1^{1/2} \le L_1^{1/4} L_2^{1/4}$, which is what we want.

\subsubsection{The case $L_0 \le L_1 \le L_2$} Then we do not use \eqref{EM1} at all, but apply instead \eqref{Embedding2} with $p=\infty$ followed by \eqref{Linear2} with $b=0$ to obtain
\begin{equation}\label{EM2}
\begin{aligned}
  &\sup_{\abs{t} \le T} \norm{\Proj_{\abs{\xi_0} \ge 1}I_{1,1}(t)}_{H^{-1/2}}
  \lesssim
  \norm{\Proj_{\abs{\xi_0} \ge 1}I_{1,1}}_{X^{-1/2,1/2;1}_{\pm_0}(S_T)}
  \\
  &\quad \lesssim
  T^{1/2}
  \sup_{L_0 \ge 1}
  \norm{ \int \frac{\chi_{K^{\pm_0}_{L_0}}(X_0)}{\angles{\xi_0}^{1/2}}
  \theta_{12} \abs{\widetilde{\psi_1}(X_1)} \abs{\widetilde{\psi_2}(X_2)}
  \, d\mu^{12}_{X_0}}_{L^2_{X_0}}.
\end{aligned}
\end{equation}
Of course, we only do this for the part of $I_{1,1}$ corresponding to the restriction $L_0 \le L_1 \le L_2$, which is tacitly assumed. Now it suffices to show that
\begin{align*}
  \norm{ \int \frac{\chi_{K^{\pm_0}_{L_0}}(X_0)}{\angles{\xi_0}^{1/2}}
  \theta_{12} \abs{\widetilde{\psi}(X_1)} \abs{\widetilde{\psi}(X_2)}
  \, d\mu^{12}_{X_0}}_{L^2_{X_0}}
  \lesssim
  \norm{\psi_1}_{X^{0,1/2;1}_{\pm_1}}\norm{\psi_2}_{X^{0,1/2;1}_{\pm_2}}
\end{align*}
uniformly in $L_0$, since the right hand side equals
$$
  \norm{\rho_T\psi_{\pm_1}}_{X^{0,1/2;1}_{\pm_1}} \norm{\rho_T\psi_{\pm_2}}_{X^{0,1/2;1}_{\pm_2}}
  \lesssim \norm{\psi_{\pm_1}}_{X^{0,1/2;1}_{\pm_1}} \norm{\psi_{\pm_2}}_{X^{0,1/2;1}_{\pm_2}}
  \le C,
$$
where we use \eqref{Cutoff2} and \eqref{ExtensionBound}, so $C$ only depends on the charge constant.

To prove the desired estimate, observe that
\begin{equation}\label{EM2:2}
\begin{aligned}
  &\norm{ \int \frac{\chi_{K^{\pm_0}_{L_0}}(X_0)}{\angles{\xi_0}^{1/2}}
  \theta_{12} \abs{\widetilde{\psi_1}(X_1)} \abs{\widetilde{\psi_2}(X_2)}
  \, d\mu^{12}_{X_0}}_{L^2_{X_0}}
  \\
  &\quad=
  \sup_{\norm{G}=1}
  \abs{ \int \frac{G(X_0)\chi_{K^{\pm_0}_{L_0}}(X_0)}{\angles{\xi_0}^{1/2}}
  \theta_{12} \abs{\widetilde{\psi_1}(X_1)} \abs{\widetilde{\psi_2}(X_2)}
  \, d\mu^{12}_{X_0}
  \, d\tau_0 \, d\xi_0 }
  \\
  &\quad\lesssim
  \sup_{\norm{G}=1}
  \sum_{\boldN,L_1,L_2}
  \frac{1}{N_0^{1/2}}
  \left( \frac{L_2}{\Nmin^{12}} \right)^{p}
  \norm{\chi_{\angles{\xi_0} \sim N_0} G} \norm{\Proj_{K^{\pm_0}_{N_0,L_0}} (u_1 \overline{u_2})}
\end{aligned}
\end{equation}
for $0 \le p \le 1/2$, recalling that $L_0 \le L_2 = \Lmax^{12}$. Take $p=1/4$ and use \eqref{Bilinear2} to estimate the summand by
\begin{align*}
  &\left( \frac{L_2}{\Nmin^{12}} \right)^{1/4}
  \frac{\left( \Nmin^{012} N_0^{1/2} L_1 L_0^{1/2} \right)^{1/2}}{N_0^{1/2}} 
  \norm{\chi_{\angles{\xi_0} \sim N_0} G} \norm{u_1}\norm{u_2}
  \\
  &\quad
  =
  \frac{(\Nmin^{012})^{1/2} L_1^{1/2} L_0^{1/4} L_2^{1/4}}{(N_0\Nmin^{12})^{1/4}}
  \norm{\chi_{\angles{\xi_0} \sim N_0} G} \norm{u_1} \norm{u_2}
  \\
  &\quad
  \lesssim
  \left( \frac{\Nmin^{012}}{\Nmax^{012}} \right)^{1/4}
  L_1^{1/2} L_2^{1/2}
  \norm{\chi_{\angles{\xi_0} \sim N_0} G} \norm{u_1} \norm{u_2}
\end{align*}
where we used $N_0\Nmin^{12} \sim \Nmax^{012}\Nmin^{012}$. This gives the desired bound. For later use we note that the above argument actually works for $L_0 \le L_2$ (we do not need to assume $L_0 \le L_1$).

\subsection{Estimate for $I_{1,2}$ with $\abs{\xi_0} \ge 1$}

The only difference from the previous subsection is that $\theta_{12}$ is replaced by $\min(\theta_{01},\theta_{02})$, so \eqref{Theta12} is replaced by
$$
  \min(\theta_{01},\theta_{02}) \lesssim \left(\frac{\Lmax^{012}}{N_0}\right)^p \qquad (0 \le p \le 1/2),
$$
by Lemma \ref{F:Lemma4}. Therefore, it suffices to look at the case $N_0 \ll N_1 \sim N_2$. By symmetry we assume $L_1 \le L_2$.

\subsubsection{The case $L_0 \le L_2$} Then we modify \eqref{EM2} and \eqref{EM2:2} in the obvious way, and use \eqref{Bilinear2} to estimate the summand in the last line of \eqref{EM2:2} by
$$\left( \frac{L_2}{N_0} \right)^{p}
  \frac{\left( N_0^{3/2} L_1 L_0^{1/2} \right)^{1/2}}{N_0^{1/2}} 
  \norm{G} \norm{u_1}\norm{u_2}
  \lesssim
  \left( \frac{L_2}{N_0} \right)^{p-1/4}
  L_1^{1/2} L_2^{1/2}
  \norm{G} \norm{u_1} \norm{u_2}.
$$
If $N_0 < L_2$ we take $p=0$, otherwise $p=1/2$. In either case we can then sum $N_0$ without problems, and we get the desired estimate.

\subsubsection{The case $L_1 \le L_2 \le L_0$}\label{Reuse} Here we use the obvious analog of \eqref{EM1}. We may assume $\theta_{12} \ll 1$ [otherwise we trivially reduce to \eqref{EM1}], so by Lemma \ref{AnglesLemma},
\begin{equation}\label{N:42}
  \theta_{12} \lesssim \gamma \equiv \biggl(\frac{N_0L_0}{N_1N_2}\biggr)^{1/2}
\end{equation}
Applying Lemma \ref{WhitneyLemma2}, then instead of the summand in the last line of \eqref{EM1} we now have
$$
  \sum_{\omega_1,\omega_2}
  \frac{1}{N_0^{1/2}L_0}
  \left( \frac{L_0}{N_0} \right)^{p}
  \norm{\chi_{H_d(\omega_1)}(X_0) \chi_{\angles{\xi_0} \sim N_0} f(\xi_0)}_{L^2_{X_0}} \norm{\Proj_{K^{\pm_0}_{N_0,L_0}} 
  \left(u_1^{\gamma,\omega_1} \overline{u_2^{\gamma,\omega_2}}\right)},
$$
where the sum is over $\omega_1,\omega_2 \in \Omega(\gamma)$ with $\theta(\omega_1,\omega_2) \lesssim \gamma$, and the restriction of $X_0$ to the thickened null hyperplane $H_d(\omega_1) = \{ X_0 \colon \tau_0 + \xi_0\cdot\omega_1 = O(d) \}$ with
$$
  d = \max\left( L_2, N_1 \gamma^2 \right)
  \sim \max\left( L_2, \frac{N_0L_0}{N_1} \right)
$$
comes from applying \eqref{B:112} to $\mathcal F(u_1^{\gamma,\omega_1} \overline{u_2^{\gamma,\omega_2}})(X_0)$. Now estimate
\begin{align*}
  & 
  \sum_{\omega_1,\omega_2} \frac{1}{N_0^{1/2}L_0}
  \left( \frac{L_0}{N_0} \right)^{p}
  d^{1/2} \norm{\chi_{\angles{\xi_0} \sim N_0} f(\xi_0)}_{L^2_{\xi_0}} \norm{\Proj_{K^{\pm_0}_{N_0,L_0}} 
  \left(u_1^{\gamma,\omega_1} \overline{u_2^{\gamma,\omega_2}}\right)}
  \\
  &\lesssim
  \frac{1}{N_0^{1/2}L_0}
  \left( \frac{L_0}{N_0} \right)^{p}
  \left[\max\left( L_2, \frac{N_0L_0}{N_1} \right)\right]^{1/2}
  \\
  &\quad
  \times\min\left( N_0^{3/2} L_0^{1/2} L_1, N_0 N_1^{1/2} L_1 L_2^{1/2} \right)^{1/2}
  \norm{\chi_{\angles{\xi_0} \sim N_0} f(\xi_0)}_{L^2_{\xi_0}}
  \norm{u_1} \norm{u_2}
  \\
  &\lesssim
  \left( \frac{L_0}{N_0} \right)^{p-1/4}
  \frac{L_1^{1/2} L_2^{1/4}}{L_0^{1/4}}
  \norm{\chi_{\angles{\xi_0} \sim N_0} f(\xi_0)}_{L^2_{\xi_0}}
  \norm{u_1}\norm{u_2},
\end{align*}
where we used Theorem \ref{BasicBilinearThm} and summed $\omega_1,\omega_2$ as in \eqref{OmegaSum}. If $N_0 < L_0$ we take $p=0$, otherwise $p=1/2$, and this allows us to sum $N_0$, leaving us with the sum
$$
  \sum_{L_0 \colon L_0 \ge L_2}
  \frac{L_1^{1/2} L_2^{1/4}}{L_0^{1/4}}
  \sim L_1^{1/2}
  \le L_1^{1/4}L_2^{1/4},
$$
as desired.

\subsection{Estimates for $I_{2,1}$ and $I_{2,2}$ with $\abs{\xi_0} \ge 1$} These follow from the arguments used for $I_{1,1}$ and $I_{1,2}$ in the two previous subsections. Indeed, the only difference is that we apply \eqref{Linear2} and \eqref{Linear3} with $\phi(\xi) = \pm\angles{\xi}$ instead of $\phi(\xi) = \pm\abs{\xi}$, but the same estimates apply, since $\angles{\tau\pm\abs{\xi}} \sim \angles{\tau\pm\angles{\xi}}$. Thus the proof of \eqref{Hi} is complete.

\subsection{Estimates for $I_{j,k}$, $j,k=1,2$, with $\abs{\xi_0} < 1$}

Since we only consider $j,k=1,2$, we have $\abs{\sigma_{\kappa\lambda}(X_1,X_2)} \lesssim \abs{\xi_0}$, hence \eqref{Linear3} gives
\begin{align*}
  &\sum_{0 < N_0 < 1} \norm{\Proj_{\abs{\xi_0} \sim N_0} I_{j,k}(t)}
  \\
  &\quad\lesssim
  \sum_{0 < N_0 < 1} 
  \norm{ \int \frac{\chi_{\abs{\xi_0} \sim N_0}}{\angles{\tau_0\pm_0\abs{\xi_0}}}
  \abs{\widetilde{\psi_1}(X_1)} \abs{\widetilde{\psi_2}(X_2)}
  \, d\mu^{12}_{X_0}
  \, d\tau_0 }_{L^2_{\xi_0}}
  \\
  &\quad\lesssim
  \sum_{0 < N_0 < 1}
  \sum_{L_0}
  \frac{1}{L_0} L_0^{1/2}
  \norm{ \chi_{\abs{\xi_0} \sim N_0}
  \int \abs{\widetilde{\psi_1}(X_1)} \abs{\widetilde{\psi_2}(X_2)}
  \, d\mu^{12}_{X_0} }_{L^2_{X_0}}
  \\
  &\quad\lesssim
  \sum_{0 < N_0 < 1}
  \sum_{\boldL}
  \frac{1}{L_0^{1/2}}
  N_0 (\Lmin^{012})^{1/2}
  \bignorm{\Proj_{K^{\pm_1}_{L_1}} \psi_1}
  \bignorm{\Proj_{K^{\pm_2}_{L_2}} \psi_2}
  \\
  &\quad\lesssim
  \norm{\psi_1}_{X^{0,1/4;1}_{\pm_1}} \norm{\psi_2}_{X^{0,1/4;1}_{\pm_2}},
\end{align*}
where we used \eqref{SobolevType} and estimated $(\Lmin^{012})^{1/2} \le L_1^{1/4}L_2^{1/4}$. Recalling \eqref{TimeGain} we then get \eqref{Lo}, as desired.

\subsection{Estimate for $I_{2,3}$}

Note that
$$
  \norm{I_{2,3}(t)}_{(T)}
  \le
  \sum_{0 < N_0 < 1/T} T^{1/2}  \norm{\Proj_{\abs{\xi_0} \sim N_0} I_{2,3}(t)}
  +
  \sum_{N_0 \ge 1/T} N_0^{-1/2} \norm{\Proj_{\abs{\xi_0} \sim N_0} I_{2,3}(t)}.
$$
But now $\abs{\sigma_{k0}(X_1,X_2)} \lesssim \abs{\tau_0\pm_0\abs{\xi_0}}$, so \eqref{Linear3} gives
\begin{align*}
  \norm{\Proj_{\abs{\xi_0} \sim N_0} I_{2,3}(t)}
  &\lesssim
  \frac{1}{\angles{N_0}}
  \norm{\chi_{\abs{\xi_0} \sim N_0}  \int 
  \abs{\widetilde{\psi_1}(X_1)}
  \abs{\widetilde{\psi_2}(X_2)}
  \, d\mu^{12}_{X_0}
  \, d\tau_0}_{L^2_{\xi_0}}
  \\
  &=
  \frac{1}{\angles{N_0}}
  \norm{ \chi_{\abs{\xi_0} \sim N_0} \int
  f_1(\xi_1) f_2(\xi_1-\xi_0)
  \, d\xi_1}_{L^2_{\xi_0}}
  \\
  &\le
  \frac{1}{\angles{N_0}}
  \norm{ \chi_{\abs{\xi_0} \sim N_0}}_{L^2_{\xi_0}} \norm{f_1}\norm{f_2}
  \sim
  \frac{N_0}{\angles{N_0}}
  \norm{f_1}\norm{f_2},
\end{align*}
where
$$
  f_j(\xi_j) = \int \abs{\widetilde{\psi_j}(\tau_j,\xi_j)} \, d\tau_j
$$
hence
\begin{equation}\label{fjEst}
  \norm{f_j} \lesssim \norm{\psi_j}_{X^{0,1/2;1}_{\pm_j}}
  = \norm{\rho_T\psi_{\pm_j}}_{X^{0,1/2;1}_{\pm_j}}
  \le C
\end{equation}
with $C$ depending only on the charge constant, by \eqref{Cutoff2} and \eqref{ExtensionBound}.

Thus
\begin{align*}
  \norm{I_{2,3}(t)}_{(T)}
  &\le C 
   \left( T^{1/2} \sum_{0 < N_0 < 1} N_0
  +
  T^{1/2} \sum_{1 \le N_0 < 1/T} 1
  +
  \sum_{N_0 \ge 1/T}
  N_0^{-1/2} \right)
  \\
  &\sim
  C
  \left( T^{1/2}
  +
  T^{1/2}\log(1/T)
  +
  T^{1/2}
  \right)
\end{align*}
with $C$ depending only on the charge constant, proving \eqref{InhomGrowth} for $I_{2,3}$.

This concludes the proof of \eqref{InhomGrowth} for $I_1$ and $I_2$, and only $I_3$ remains.

\subsection{Estimate for $I_3$}

By \eqref{Embedding2} with $p=\infty$ and \eqref{Linear2} with $b=0$,
\begin{align*}
  \sup_{\abs{t} \le T} \norm{\Proj_{\abs{\xi_0} \sim N_0} I_3(t)}
  &\lesssim
  \norm{\Proj_{\abs{\xi_0} \sim N_0} I_3(t)}_{X^{0,1/2;1}_{\pm_0}(S_T)}
  \\
  &\lesssim
  T^{1/2} \norm{ \rho_T \Proj_{\abs{\xi_0} \sim N_0} \angles{D}^{-1}\mathbf E^{\mathrm{df}}}_{X^{0,0;\infty}_{\pm_0}}
  \\
  &\lesssim
  T^{1/2} \norm{ \rho_T \Proj_{\abs{\xi_0} \sim N_0} \angles{D}^{-1}\mathbf E^{\mathrm{df}}}
  \\
  &\lesssim
  T^{1/2} \norm{\rho_T} \sup_{\abs{t} \le 1} \norm{\Proj_{\abs{\xi_0} \sim N_0} \angles{D}^{-1}\mathbf E^{\mathrm{df}}(t)}.
\end{align*}
Thus
$$
  \sup_{\abs{t} \le T}\norm{\Proj_{\abs{\xi_0} \sim N_0} I_3(t)}
  \lesssim T \sup_{\abs{t} \le 1} \norm{\Proj_{\abs{\xi_0} \sim N_0} \mathbf E^{\mathrm{df}}(t)}_{H^{-1}}
$$
and similarly
$$
  \sup_{\abs{t} \le T}\norm{\Proj_{\abs{\xi_0} \ge 1/T} I_3(t)}_{H^{-1/2}}
  \lesssim T \sup_{\abs{t} \le 1} \norm{\Proj_{\abs{\xi_0} \ge 1/T} \mathbf E^{\mathrm{df}}(t)}_{H^{-3/2}}
$$
hence
\begin{align*}
  \sup_{\abs{t} \le T} \norm{I_3(t)}_{(T)}
  &\lesssim
  T T^{1/2}\sum_{0 < N_0 < 1/T} \sup_{\abs{t} \le 1} \norm{\Proj_{\abs{\xi_0} \sim N_0} \mathbf E^{\mathrm{df}}(t)}_{H^{-1}}
  \\
  &\quad
  + T \sup_{\abs{t} \le 1} \norm{\Proj_{\abs{\xi_0} \ge 1/T} \mathbf E^{\mathrm{df}}(t)}_{H^{-3/2}},
\end{align*}
and to estimate the right hand side we now apply the following lemma, proved in the next section.

\begin{lemma}\label{EnergyLemma}
Let $s \in \R$. The solution of $\square u = F$ with initial data $u(0) = f$, $\partial_t u(0) = g$ satisfies
$$
  \sup_{\abs{t} \le 1} \norm{u(t)}_{H^s}
  \lesssim \norm{f}_{H^s} + \norm{g}_{H^{s-1}}
  +
  \norm{\angles{\xi}^{s-1} \int \frac{\bigabs{\widetilde F(\tau,\xi)}}{\angles{\abs{\tau}-\abs{\xi}}} 
  \, d\tau}_{L^2_{\xi}},
$$
where the implicit constant is absolute.
\end{lemma}

Applying this to \eqref{WaveEqE}, where $J^\mu$ is now defined for all $t$ by \eqref{Jcut}, we find
\begin{align*}
  &\sup_{\abs{t} \le 1} \norm{\Proj_{\abs{\xi_0} \sim N_0} \mathbf E^{\mathrm{df}}(t)}_{H^{-1}}
  \\
  &\quad
  \lesssim \norm{\Proj_{\abs{\xi_0} \sim N_0} \mathbf E^{\mathrm{df}}_0}_{H^{-1}}
  + \norm{\Proj_{\abs{\xi_0} \sim N_0} \left( \nabla \times (0,0,B^3_0) - \mathcal P_{\text{df}} \mathbf J(0) \right)}_{H^{-2}}
  \\
  &\quad\quad
  + \norm{\int \frac{\chi_{\abs{\xi_0} \sim N_0}}{\angles{\xi_0}^2\angles{\abs{\tau_0}-\abs{\xi_0}}} \abs{\mathcal F \left[ \mathcal P_{\text{df}}(-\nabla J_0 + \partial_t \mathbf J) \right] (X_0)} \, d\tau_0}_{L^2_{\xi_0}}
  \\
  &\quad\lesssim \norm{\Proj_{\abs{\xi_0} \sim N_0} \mathbf E^{\mathrm{df}}_0}
  + \norm{\Proj_{\abs{\xi_0} \sim N_0} B^3_0}
  + \frac{N_0}{\angles{N_0}^2} \norm{\psi_0}^2
  \\
  &\quad\quad
  + \norm{\int \frac{\chi_{\abs{\xi_0} \sim N_0}}{\angles{\xi_0}^2\angles{\abs{\tau_0}-\abs{\xi_0}}} \abs{\mathcal F \left[ \mathcal P_{\text{df}}(-\nabla J_0 + \partial_t \mathbf J) \right] (X_0)} \, d\tau_0}_{L^2_{\xi_0}},
\end{align*}
where we applied \eqref{BallSobolev} to get $\norm{\Proj_{\abs{\xi_0} \sim N_0} \mathcal P_{\text{df}} \mathbf J(0)}_{H^{-2}} \lesssim N_0\angles{N_0}^{-2} \norm{\psi_0}^2$.

Similarly
\begin{align*}
  &\sup_{\abs{t} \le 1} \norm{\Proj_{\abs{\xi_0} \ge 1/T} \mathbf E^{\mathrm{df}}(t)}_{H^{-3/2}}
  \\
  &\quad
  \lesssim \norm{\Proj_{\abs{\xi_0} \ge 1/T} \mathbf E^{\mathrm{df}}_0}_{H^{-3/2}}
  + \norm{\Proj_{\abs{\xi_0} \ge 1/T} \left( \nabla \times (0,0,B^3_0) - \mathcal P_{\text{df}} \mathbf J(0) \right)}_{H^{-5/2}}
  \\
  &\quad\quad
  + \norm{\int \frac{\chi_{\abs{\xi_0} \ge 1/T}}{\angles{\xi_0}^{5/2}\angles{\abs{\tau_0}-\abs{\xi_0}}} \abs{\mathcal F \left[ \mathcal P_{\text{df}}(-\nabla J_0 + \partial_t \mathbf J) \right] (X_0)} \, d\tau_0}_{L^2_{\xi_0}}
  \\
  &\quad\lesssim \norm{\Proj_{\abs{\xi_0} \ge 1/T} \mathbf E^{\mathrm{df}}_0}_{H^{-1/2}}
  + \norm{\Proj_{\abs{\xi_0} \ge 1/T} B^3_0}_{H^{-1/2}}
  + \norm{\psi_0}^2
  \\
  &\quad\quad
  + \norm{\int \frac{\chi_{\abs{\xi_0} \ge 1/T}}{\angles{\xi_0}^{5/2}\angles{\abs{\tau_0}-\abs{\xi_0}}} \abs{\mathcal F \left[ \mathcal P_{\text{df}}(-\nabla J_0 + \partial_t \mathbf J) \right] (X_0)} \, d\tau_0}_{L^2_{\xi_0}},
\end{align*}
where we used \eqref{Jest2}.

Thus
\begin{equation}\label{I3}
  \sup_{\abs{t} \le T} \norm{I_3(t)}_{(T)}
  \lesssim T \left( D_T(0) + \norm{\psi_0}^2 \right)
  + T \left( T^{1/2} \sum_{0 < N_0 < 1/T} a_{N_0}
  + b \right),
\end{equation}
where
\begin{align*}
  a_{N_0}
  &=
  \norm{\int \frac{\chi_{\abs{\xi_0} \sim N_0}}{\angles{\xi_0}^2\angles{\abs{\tau_0}-\abs{\xi_0}}} \abs{\mathcal F \left[ \mathcal P_{\text{df}}(-\nabla J_0 + \partial_t \mathbf J) \right] (X_0)} \, d\tau_0}_{L^2_{\xi_0}},
  \\
  b
  &=
  \norm{\int \frac{\chi_{\abs{\xi_0} \ge 1/T}}{\angles{\xi_0}^{5/2}\angles{\abs{\tau_0}-\abs{\xi_0}}} \abs{\mathcal F \left[ \mathcal P_{\text{df}}(-\nabla J_0 + \partial_t \mathbf J) \right] (X_0)} \, d\tau_0}_{L^2_{\xi_0}}.
\end{align*}

But by \eqref{Efourier} and \eqref{N:14},
$$
  \abs{\mathcal F\left[\mathcal P_{\text{df}}(-\nabla J_0 + \partial_t \mathbf J)\right](X_0)}
  \lesssim
  \int \left( \abs{\xi_0} + \abs{\abs{\tau_0}-\abs{\xi_0}} \right)
  \abs{\widetilde{\psi}(X_1)} \abs{\widetilde{\psi}(X_2)}
  \, d\mu^{12}_{X_0},
$$
hence
\begin{align*}
  a_{N_0}
  &\lesssim
  \norm{\frac{\chi_{\abs{\xi_0} \sim N_0}}{\angles{\xi_0}}
  \int 
  \abs{\widetilde{\psi}(X_1)} \abs{\widetilde{\psi}(X_2)}
  \, d\mu^{12}_{X_0} \, d\tau_0}_{L^2_{\xi_0}}
  \\
  &=
  \norm{\frac{\chi_{\abs{\xi_0} \sim N_0}}{\angles{\xi_0}}
  \int 
  f_1(\xi_1) f_2(\xi_1-\xi_0)  \, d\xi_1}_{L^2_{\xi_0}}
  \\
  &\lesssim
  \norm{\frac{\chi_{\abs{\xi_0} \sim N_0}}{\angles{\xi_0}}}_{L^2_{\xi_0}}
  \norm{f_1}\norm{f_2}
  \sim
  \frac{N_0}{\angles{N_0}} \norm{f_1}\norm{f_2},
\end{align*}
where $f_j(\xi_j) = \int \bigabs{\widetilde{\psi_j}(\tau_j,\xi_j)} \, d\tau_j$ satisfies \eqref{fjEst} with $C$ depending only the charge constant. Similarly,
$$
  b
  \lesssim
  \norm{\frac{\chi_{\abs{\xi_0} \ge 1/T}}{\angles{\xi_0}^{3/2}}}_{L^2_{\xi_0}}
  \norm{f_1}\norm{f_2}
  \sim
  T^{1/2} \norm{f_1}\norm{f_2},
$$
hence
\begin{align*}
  T^{1/2} \sum_{0 < N_0 < 1/T} a_{N_0}
  + b
  &\lesssim
  T^{1/2}\sum_{0 < N_0 < 1}  N_0
  + T^{1/2} \sum_{1 \le N_0 < 1/T} 1
  + T^{1/2}
  \\
  &\lesssim T^{1/2} + T^{1/2}\log(1/T) + T^{1/2},
\end{align*}
with implicit constants depending only on the charge constant, so we finally conclude that
\begin{align*}
  \sup_{\abs{t} \le T} \norm{I_3(t)}_{(T)}
  &\lesssim (1+\norm{\psi_0}^2) T \left[ 1 + D_T(0) \right] + CT^{3/2}\log(1/T)
  \\
  &\le (1+\norm{\psi_0}^2) T^{1/2}\varepsilon + CT^{3/2}\log(1/T),
\end{align*}
where $C$ depends only on the charge constant and we used \eqref{MDtimeNew} in the last step, recalling that $D_T(0) \le \tilde D_T(0)$. Thus $\varepsilon$ depends only on the charge constant and $\abs{M}$, so we have proved \eqref{InhomGrowth} for $I_3$.

Finally, we remark that the estimates proved in this section also give that $\mathbf E^{\mathrm{df}}_\pm$ and $B^3_\pm$ describe continuous curves in the data space \eqref{Data} for $\abs{t} \le T$.

\section{Proof of Lemma \ref{EnergyLemma}}

For the homogeneous part of $u$ this follows from the standard energy inequality, so we assume $f=g=0$, i.e.~$u=\square^{-1}F$. Now split $F=F_1+F_2+F_3$ corresponding to the following three regions in Fourier space: (i) $\abs{\xi} \ge 1$, (ii) $\abs{\xi} < 1$ and $\abs{\tau} \ge 2$, and (iii) $\abs{\xi} < 1$ and $\abs{\tau} < 2$. Set $u_j=\square^{-1}F_j$ for $j=1,2,3$.

From Lemma \ref{DuhamelLemma} we get
$$
  \norm{u_1(t)}_{H^s}
  \lesssim
  \norm{\chi_{\abs{\xi} \ge 1} \frac{\angles{\xi}^{s}}{\abs{\xi}} \int \frac{\bigabs{\widetilde F(\tau,\xi)}}{\angles{\abs{\tau}-\abs{\xi}}} 
  \, d\tau}_{L^2_{\xi}}
  \lesssim
  \norm{\angles{\xi}^{s-1} \int \frac{\bigabs{\widetilde F(\tau,\xi)}}{\angles{\abs{\tau}-\abs{\xi}}} 
  \, d\tau}_{L^2_{\xi}}
$$
for all $t \in \R$.

Lemma \ref{DuhamelLemma} also gives
\begin{align*}
  \widehat u(t,\xi)
  &=
  \widehat{u_+}(t,\xi) + \widehat{u_-}(t,\xi) 
  \\
  &\simeq
  \frac{1}{\abs{\xi}} \int_{-\infty}^\infty
  \left[
  \frac{e^{it\tau}-e^{-it\abs{\xi}}}{\tau+\abs{\xi}}
  -
  \frac{e^{it\tau}-e^{it\abs{\xi}}}{\tau-\abs{\xi}}
  \right]
  \widetilde F(\tau,\xi) \, d\tau
  \\
  &\simeq
  \frac{1}{\abs{\xi}} \int_{-\infty}^\infty
  \frac{-2\abs{\xi}e^{it\tau} + 2\abs{\xi}\cos(t\abs{\xi}) + 2i\tau\sin(t\abs{\xi})}{\tau^2-\abs{\xi}^2}
  \widetilde F(\tau,\xi) \, d\tau,
\end{align*}
so if $\widetilde F$ is supported in $\abs{\tau} \gg \abs{\xi}$ we get
$$
  \abs{\widehat u(t,\xi)}
  \lesssim
  \int
  \left( \frac{1}{\abs{\tau}^2} + \frac{\min(\abs{t},\abs{\xi}^{-1})}{\abs{\tau}} \right)
  \abs{\widetilde F(\tau,\xi)} \, d\tau,
$$
and applying this to $u_2$ yields
$$
  \sup_{\abs{t} \le 1} \norm{u_2(t)}_{H^s}
  \lesssim
  \norm{\chi_{\abs{\xi} < 1} \int_{\abs{\tau} \ge 2} \frac{\bigabs{\widetilde F(\tau,\xi)}}{\abs{\tau}} 
  \, d\tau}_{L^2_{\xi}}
  \lesssim
  \norm{\angles{\xi}^{s-1} \int \frac{\bigabs{\widetilde F(\tau,\xi)}}{\angles{\abs{\tau}-\abs{\xi}}} 
  \, d\tau}_{L^2_{\xi}}.
$$

Finally, by the standard energy inequality we have
\begin{align*}
  \sup_{\abs{t} \le 1} \norm{u_3(t)}_{H^s}
  &\lesssim
  \int_0^1 \norm{F_3(t)}_{H^{s-1}} \, dt
  \lesssim
  \sup_{\abs{t} \le 1} \norm{F_3(t)}_{H^{s-1}}
  \\
  &\lesssim \norm{\angles{\xi}^{s-1} \int \bigabs{\widetilde{F_3}(\tau,\xi)} \, d\tau}_{L^2_\xi}
  \lesssim
  \norm{\angles{\xi}^{s-1} \int \frac{\bigabs{\widetilde F(\tau,\xi)}}{\angles{\abs{\tau}-\abs{\xi}}} 
  \, d\tau}_{L^2_{\xi}},
\end{align*}
completing the proof of the lemma.

\section{Proof of the linear estimates in $X^{s,b;p}_{\phi(\xi)}$}\label{LemmaProofs}

Here we prove \eqref{Linear1} and \eqref{Linear2} by an argument similar to the one used in \cite{Kenig:1994} for the standard $X^{s,b}$ spaces. Moreover, we prove \eqref{Linear3}.

\subsection{Proof of \eqref{Linear1}} Letting $G \in X_{\phi(\xi)}^{s,-1/2;1}$ denote an arbitrary representative of $F \in X_{\phi(\xi)}^{s,-1/2;1}(S_T)$, we reduce to proving
$$
  \norm{u}_{X_{\phi(\xi)}^{s,1/2;1}(S_T)}
  \lesssim
  \norm{f}_{H^s}
  +
  \norm{G}_{X_{\phi(\xi)}^{s,-1/2;1}}.
$$
By density we may assume $G \in \mathcal S(\R^{1+2})$. Denote by $S(t) = e^{-it\phi(D)}$ the free propagator of $-i\partial_t + \phi(D)$. Split the solution of $[-i\partial_t + \phi(D)] u = G$, $u(0) = f$ into homogeneous and inhomogeneous parts, $u = v + w$, where $v(t) = S(t) f$ and $w(t) = \int_0^t S(t-t') G(t') \, dt'$. 

Since $\widetilde{v}(\tau,\xi) = \delta\left(\tau+\phi(\xi)\right)
\widehat f(\xi)$,
\begin{align*}
  \norm{v}_{X^{s,1/2;1}_{\phi(\xi)}(S_T)}
  &\le \norm{\rho v}_{X^{s,1/2;1}_{\phi(\xi)}}
  \\
  &= \sum_L L^{1/2} \norm{ \angles{\xi}^s \chi_{\angles{\tau+\phi(\xi)} \sim L} 
  \widehat\rho\left(\tau+\phi(\xi)\right) \widehat f(\xi)}_{L^2_{\tau,\xi}}
  \\
  &= \sum_L L^{1/2} \norm{\Proj_{\angles{\tau} \sim L} \rho}_{L^2_t}
  \norm{f}_{H^s}
  = \norm{\rho}_{B^{1/2}_{2,1}} \norm{f}_{H^s}.
\end{align*}

Next, taking Fourier transform in space,
\begin{equation}\label{wFourier}
  \widehat w(t,\xi)
  =
  \int_0^t e^{-i(t-t')\phi(\xi)} \widehat G(t',\xi) \, dt'
  \simeq
  \int \frac{e^{it\lambda}-e^{-it\phi(\xi)}}{i(\lambda+\phi(\xi))} \widetilde G(\lambda,\xi) \, d\lambda
\end{equation}
and then also in time,
$$
  \widetilde w(\tau,\xi)
  =
  \int \frac{\delta(\tau-\lambda)-\delta(\tau+\phi(\xi))}{i(\lambda+\phi(\xi))} \widetilde G(\lambda,\xi) \, d\lambda
  =
  \frac{\widetilde G(\tau,\xi)}{i(\tau+\phi(\xi))}
  - \delta(\tau+\phi(\xi)) \widehat g(\xi),
$$
where
$$
  \widehat g(\xi)
  =
  \int \frac{\widetilde G(\lambda,\xi)}{i(\lambda+\phi(\xi))}  \, d\lambda.
$$
Now split $G = G_1 + G_2$ corresponding to the Fourier domains $\abs{\tau+\phi(\xi)} \lesssim 1$ and $\abs{\tau+\phi(\xi)} \gg 1$ respectively. Write $w=w_1+w_2$ accordingly. Expand
$$
  \widehat{w_1}(t,\xi)
  =
  e^{-it\phi(\xi)} \sum_{n=1}^\infty \int \frac{\left[it(\lambda+\phi(\xi))\right]^n}{n! i(\lambda+\phi(\xi))}\chi_{\abs{\lambda+\phi(\xi)} \lesssim 1} \widetilde G(\lambda,\xi) \, d\lambda
$$
hence
\begin{equation}\label{w1}
  w_1(t) =
  \sum_{n=1}^\infty \frac{t^n}{n!} S(t) f_n
\end{equation}
where
\begin{equation}\label{fn}
  \widehat{f_n}(\xi) = \int \left[i(\lambda+\phi(\xi))\right]^{n-1}
  \chi_{\abs{\lambda+\phi(\xi)} \lesssim 1} \widetilde G(\lambda,\xi) \, d\lambda
\end{equation}
and clearly
$$
  \norm{f_n}_{H^s}
  \lesssim
  \norm{G}_{X^{s,-1/2;1}_{\phi(\xi)}}.
$$
Thus
\begin{align*}
  \norm{w_1}_{X^{s,1/2;1}_{\phi(\xi)}(S_T)}
  &\le
  \norm{\rho w_1}_{X^{s,1/2;1}_{\phi(\xi)}}
  \le
  \sum_{n=1}^\infty \frac{1}{n!} \norm{t^n \rho(t) S(t)f_n}_{X^{s,1/2;1}_{\phi(\xi)}}
  \\
  &\le
  \sum_{n=1}^\infty \frac{1}{n!} \norm{t^n \rho(t)}_{B^{1/2}_{2,1}}
  \norm{f_n}_{H^s}
  \lesssim
  \left(\sum_{n=1}^\infty \frac{n2^{n-1}}{n!}\right)
  \norm{G}_{X^{s,-1/2;1}_{\phi(\xi)}}
\end{align*}
since $ \norm{t^n \rho(t)}_{B^{1/2}_{2,1}} \lesssim \norm{t^n \rho(t)}_{H^1} \lesssim 2^n + n2^{n-1}$. Finally, split $w_2 = a - b$ where
\begin{gather}
  \label{aDef}
  \widetilde a(\tau,\xi)
  =
  \frac{\chi_{\abs{\tau+\phi(\xi)} \gg 1} \widetilde G(\tau,\xi)}{i(\tau+\phi(\xi))},
  \\
  \label{bDef}
  \widetilde b(\tau,\xi)
  =
  \delta(\tau+\phi(\xi)) \widehat h(\xi),
  \qquad
  \widehat h(\xi)
  =
  \int \frac{\chi_{\abs{\lambda+\phi(\xi)} \gg 1} \widetilde G(\lambda,\xi)}{i(\lambda+\phi(\xi))}  \, d\lambda.
\end{gather}
Thus
$$
  \norm{a}_{X^{s,1/2;1}_{\phi(\xi)}}
  \sim
  \sum_{L \gg 1} L^{1/2} \frac{1}{L} \norm{\angles{D}^s\Proj_{\angles{\tau+\phi(\xi)} \sim L} G}
  \le
  \norm{G}_{X^{s,-1/2;1}_{\phi(\xi)}}.
$$
Moreover,
$
  \norm{h}_{H^s}
  \lesssim
  \sum_{L \gg 1} L^{-1} L^{1/2} \norm{\angles{D}^s\Proj_{\angles{\tau+\phi(\xi)} \sim L} G}
$
by Cauchy-Schwarz, so
$$
  \norm{b}_{X^{s,1/2;1}_{\phi(\xi)}(S_T)}
  \le
  \norm{\rho b}_{X^{s,1/2;1}_{\phi(\xi)}}
  \lesssim
  \norm{h}_{H^s}
  \lesssim
  \norm{G}_{X^{s,-1/2;1}_{\phi(\xi)}},
$$
and this completes the proof of \eqref{Linear1}.

\subsection{Proof of \eqref{Linear2}}

The argument here is similar, but we modify the splitting $G = G_1 + G_2$, letting it corresponding to $\abs{\tau+\phi(\xi)} \lesssim 1/T$ and $\abs{\tau+\phi(\xi)} \gg 1/T$ respectively. Then \eqref{w1} holds with $f_n$ given by the obvious modification of \eqref{fn}, hence
$$
  \norm{f_n}_{H^s}
  \lesssim
  \sum_{L \lesssim 1/T} L^{n-1} L^{1/2} \norm{\angles{D}^s\Proj_{\angles{\tau+\phi(\xi)} \sim L}G}
  \lesssim T^{-n+1/2+b} \norm{G}_{X^{s,b;\infty}_{\phi(\xi)}},
$$
where we estimated
$$
  \sum_{L \lesssim 1/T} L^{n-1/2-b}
  \sim T^{-n+1/2+b}
$$
for $b < 1/2$, recalling that $n \ge 1$, hence $n-1/2-b > 0$. Thus
\begin{align*}
  \norm{w_1}_{X^{s,1/2;1}_{\phi(\xi)}(S_T)}
  &\le
  \norm{\rho_T w_1}_{X^{s,1/2;1}_{\phi(\xi)}}
  \\
  &\le
  \sum_{n=1}^\infty \frac{1}{n!} T^n \norm{(t/T)^n \rho(t/T) S(t)f_n}_{X^{s,1/2;1}_{\phi(\xi)}}
  \\
  &\lesssim
  \sum_{n=1}^\infty \frac{1}{n!} T^n\norm{t^n \rho(t)}_{B^{1/2}_{2,1}}
   T^{-n+1/2+b}  \norm{G}_{X^{s,b;\infty}_{\phi(\xi)}}
  \\
  &\lesssim
  T^{1/2+b} \left(\sum_{n=1}^\infty \frac{n2^{n-1}}{n!}\right)
  \norm{G}_{X^{s,b;\infty}_{\phi(\xi)}},
\end{align*}
where we used the elementary estimate
$$
  \norm{\rho_T}_{B^s_{2,1}}
  \lesssim T^{1/2-s}
  \norm{\rho}_{B^s_{2,1}}
  \qquad (0 < s \le 1/2)
$$
with $s=1/2$ and $\rho(t)$ replaced by $t^n\rho(t)$.

The splitting $w_2 = a - b$ is defined as in \eqref{aDef} and \eqref{bDef} but with the obvious modifications, and we have
\begin{align*}
  \norm{a}_{X^{s,1/2;1}_{\phi(\xi)}}
  &\sim
  \sum_{L \gg 1/T} L^{1/2} \frac{1}{L} \norm{\angles{D}^s\Proj_{\angles{\tau+\phi(\xi)} \sim L} G}
  \\
  &\le
  \sum_{L \gg 1/T} L^{-1/2-b}
  \norm{G}_{X^{s,b;\infty}_{\phi(\xi)}}
  \sim T^{1/2+b} \norm{G}_{X^{s,b;\infty}_{\phi(\xi)}},
\end{align*}
provided that $-1/2-b < 0$, i.e.~$b > -1/2$. Since, by Cauchy-Schwarz,
$$
  \norm{h}_{H^s}
  \lesssim
  \sum_{L \gg 1/T} \frac{1}{L} L^{1/2} \norm{\angles{D}^s\Proj_{\angles{\tau+\phi(\xi)} \sim L} G},
$$
we also have
$$
  \norm{b}_{X^{s,1/2;1}_{\phi(\xi)}(S_T)}
  \le
  \norm{\rho b}_{X^{s,1/2;1}_{\phi(\xi)}}
  \lesssim
  \norm{h}_{H^s}
  \lesssim
  T^{1/2+b} \norm{G}_{X^{s,b;\infty}_{\phi(\xi)}},
$$
completing the proof of \eqref{Linear2}.

\subsection{Proof of \eqref{Linear3}} With $w(t) = \int_0^t S(t-t') G(t') \, dt'$, \eqref{wFourier} gives
$$
  \widehat{w}(t,\xi)
  \simeq
  e^{-it\phi(\xi)} \int \frac{e^{it(\lambda+\phi(\xi))}-1}{i(\lambda+\phi(\xi))} \widetilde G(\lambda,\xi) \, d\lambda,
$$
implying \eqref{Linear3}.

\bibliographystyle{amsalpha} 
\bibliography{mybibliography}

\end{document}